\documentclass{article}

\usepackage[colorlinks,linkcolor=black,menucolor=red,backref=false,bookmarks=true]{hyperref}

\usepackage[margin=1in]{geometry}
\usepackage{latexsym}
\usepackage{indentfirst}
\usepackage{graphicx}
\usepackage{amsmath,amsfonts}
\usepackage{amssymb}
\usepackage{amsthm}
\usepackage{tikz}
\usepackage{bm,bbm}
\usepackage{accents}
\usepackage{changepage}
\usepackage{color}
\usepackage{bigints}
\usepackage{enumerate}
\usepackage{amsmath, nccmath}
\usepackage{yhmath}
\usepackage{dirtytalk}
\usepackage{comment}
\usepackage{enumitem}
\usepackage{cleveref}

\newtheorem{prop}{Proposition}[section]
\newtheorem*{defi*}{Definition}
\newtheorem{defi}[prop]{Definition}

\newtheorem{lem}[prop]{Lemma}
\newtheorem{rem}[prop]{Remark}
\newtheorem{thm}[prop]{Theorem}
\newtheorem{coro}[prop]{Corollary}

\def\XXint#1#2#3{{\setbox0=\hbox{$#1{#2#3}{\int}$ }
\vcenter{\hbox{$#2#3$ }}\kern-.6\wd0}}

\def\eps{{\varepsilon}}

\newtheorem{theorem} {\sc  Theorem\rm} [section]

\newtheorem{proposition}[theorem] {\sc  Proposition\rm}

\newcounter{marnote}

\DeclareFontFamily{OT1}{rsfs}{}
\DeclareFontShape{OT1}{rsfs}{m}{n}{ <-7> rsfs5 <7-10> rsfs7 <10-> rsfs10}{}
\DeclareMathAlphabet{\mycal}{OT1}{rsfs}{m}{n}

\def\be{\begin{equation}}
\def\ee{\end{equation}}

\newcommand{\R}{\mathbb{R}}

\newcommand{\Ex}{\mathbbm{E}}
\newcommand{\p}{\mathbbm{P}}

\def\be{\begin{equation}}
\def\ee{\end{equation}}
\def\bea#1\eea{\begin{align}#1\end{align}}
\def\non{\nonumber}

\numberwithin{equation}{section}
\title {Quantitative boundary H\"{o}lder estimates for the inhomogeneous Poisson problem through a probabilistic approach}

\author{
Iulian C\^{i}mpean$^{1,2}$\thanks{E-mails: \texttt{iulian.cimpean@unibuc.ro;iulian.cimpean@imar.ro}}
\and
Ionel Popescu $^{1,2}$\thanks{E-mails: \texttt{ioionel@gmail.com,ionel.popescu@unibuc.ro}}
\and
Arghir Zarnescu $^{3,2, 4 }$\thanks{E-mails: \texttt{azarnescu@bcamath.org}}
\vspace{0.5 cm}
}

\date{\small
$^1$University of Bucharest, Faculty of Mathematics and Computer Science,\\ 
14 Academiei str., 70109, Bucharest, Romania\\
$^2$ Simion Stoilow Institute of Mathematics of the Romanian Academy,\\ P.O. Box 1-764, RO-014700 Bucharest, Romania \\
$^3$ BCAM, Basque Center for Applied Mathematics, Mazarredo 14, 48009, Bilbao, Bizkaia, Spain\\
$^4$ IKERBASQUE, Basque Foundation for Science, Plaza Euskadi 5, 48009, Bilbao, Bizkaia, Spain
}

\begin{document}

\maketitle

\begin{abstract}
In this paper we derive quantitative boundary H\"older estimates, with explicit constants,  for the inhomogeneous Poisson problem in a bounded open set $D\subset \mathbb{R}^d$.

Our approach has two main steps: firstly, we consider an arbitrary $D$ as above and prove that the boundary $\alpha$-H\"older regularity of the solution the Poisson equation is controlled, with explicit constants, by the H\"older seminorm of the boundary data, the $L^
\gamma$-norm of the forcing term with $\gamma>d/2$, and the $\alpha/2$-moment of the exit time from $D$ of the Brownian motion. 

Secondly, we derive explicit estimates for the $\alpha/2$-moment of the exit time in terms of the distance to the boundary, the regularity of the domain $D$, and $\alpha$.
Using this approach, we derive explicit estimates for the same problem in domains satisfying exterior ball conditions, respectively exterior cone/wedge conditions, in terms of simple geometric features.

As a consequence we also obtain explicit constants for  pointwise estimates for the Green function and for the gradient of the solution.

The obtained estimates can be employed to bypass the curse of high dimensions when aiming to approximate the solution of the Poisson problem using neural networks, obtaining polynomial scaling with dimension, which in some cases can be shown to be optimal.

\end{abstract}

\tableofcontents

\section{Introduction}

Let  $D\subset \mathbb{R}^d, d\geq 2$ be a bounded and open set. 
In this article we derive, by probabilistic tools, fully explicit boundary H\"older estimates for the solution to the inhomogeneous Poisson problem 
\begin{equation} \label{eq:PDE}
\frac{1}{2}\Delta u=-f \; \textrm{ in } D, \quad
\phantom{\Delta} u=g \,\textrm{ on }\partial D,
\end{equation}
where $f\in L^p(D), p>d/2$,
whilst $g$ is, for the moment, continuous on the boundary of $D$ denoted throughout by $\partial D$. 

If $f\equiv 0$ it is was H. Lebesgue in $1913$ in \cite{lebesgue1913cas} who first showed that one needs some restrictions on the domain for the problem to be solvable. The adequate necessary and sufficient conditions were provided in $1924$ in \cite{wiener1924dirichlet} by N. Wiener through his celebrated  criterion. 

Furthermore, aiming to understand the regularity,  for $f\equiv 0$ it is well known that harmonic functions are analytic in the interior of the domain. Thus the limitations to regularity appear at the boundary. One can naturally expect that if $g\in C^{0,\alpha}(\partial D)$ then we would have at best $u\in C^{0,\alpha}(\bar{D})$ and that there would exist a constant  $C>0$ depending only on $d$, $D$, and $\alpha$, such that 

\begin{equation}\label{est:optglbbdryhold}
\|u\|_{C^{0,\alpha}(\bar{D})}\le  C\|g\|_{C^{0,\alpha}(\partial D)},
\end{equation} 
where for any compact set $E\subset\mathbb{R}^d$ we denote  the H\"older and Lipschitz norms 
\begin{equation*}
\|h\|_{C^{0,\alpha}(E)}:=\sup_{x\in E}|h(x)|+[h]_{D,\alpha},\quad [h]_{D,\alpha}:=\sup_{x\not=x'\in E}\frac{|h(x)-h(x')|}{\|x-x'\|^\alpha}, \textrm{ for } \alpha\in (0,1] 
\end{equation*}
It is known, see for instance \cite{aikawa2002holder}, Remark $2$, that \eqref{est:optglbbdryhold} holds for $C^1$ domains $D$, for any $0<\alpha<1$, while the extreme case, the Lipschitz case, when $\alpha=1$ does not hold for any domain (see Theorem $2$ in \cite{aikawa2002holder}). 
Furthermore, if $D$ is Lipschitz there is some $\beta<1$ (depending only on the Lipschitz constant of $D$) such that \eqref{est:optglbbdryhold} only holds  for $\alpha\in (0,\beta)$.

More recently, for the homogeneous Poisson problem, i.e. $g=0$ in \eqref{eq:PDE}, it was shown in \cite[Theorem 0.2]{lemenant2013boundary} that for any $0<\alpha<1$ there exists $\delta>0$ such that the solution is $\alpha$-H\"older at $x_0\in \partial D$ provided that $D$ is $\delta$-Reifenberg flat at $x_0$. 

In fact, the problem of boundary H\"older regularity has been solved for various classes of elliptic equations, i.e. equation of type \eqref{eq:PDE} with $\Delta$ replaced by some general operator $\sf L$; see, e.g. \cite[Theorems 8.27 \& 8.29]{GTbook} for elliptic linear operators in domains satisfying an exterior cone condition, or \cite{bdrHolderJMPA} for more general equations and domains.

A H\"older estimate at the boundary  typically has the following form:
\begin{equation}\label{eq:typical}
 \exists \beta,C,r_0>0:\quad   |u(x)-u(x_0)|\leq C(\|u\|_{L^p(D)}+\|f\|_{L^\gamma(D)}+[g]_{\partial D,\alpha})\|x-x_0\|^\beta \quad \mbox{for } x\in B(x_0,r_0)\cap D,
\end{equation}
where $p=\infty$ (e.g. in \cite{aikawa2002holder},\cite{GTbook}, or \cite{bdrHolderJMPA}, or $p<\infty$ (e.g. in \cite{lemenant2013boundary})).

Taking into account the generic estimate \eqref{eq:typical}, we would like to point out two aspects that are in fact central to our work:
\begin{enumerate}
    \item[(i)] Typically, \eqref{eq:typical} is derived as a {\it qualitative} estimate in the sense that it is not clear how the constants $C$ and $\beta$ depend on the domain $D$ and on the dimension $d$.
    \item[(ii)] In the above generic estimate \eqref{eq:typical}, the presence of the term $\|u\|_{L^p(D)}$ in the right-hand side of the estimate seems more than required. Indeed, because of its presence, if we replace the boundary data $g$ by $g+n$ where $ n>0$, then the corresponding solution to \eqref{eq:PDE} becomes $u+n$, so the value of the left-hand side of \eqref{eq:typical} does not change, whilst the right hand side of \eqref{eq:typical} converges to $\infty$ as $n$ goes to $\infty$. 
    Our approach (see e.g. \Cref{prop:general_estimate} and Theorems \ref{thm:u_g_ball},\ref{thm:u_cone_2D},\ref{thm:u_wedge_3D},\ref{thm:u_cone_3D} from below) reveals that no term involving $u$ is necessary in the right-hand side of \eqref{eq:typical}.
\end{enumerate}

Our main focus in this article is to  reconsider the inhomogeneous Poisson equation as a role model for elliptic PDEs, and provide {\it quantitative} estimates, namely {\it explicitly computable constants and H\"older exponents}, in terms  of the geometry of the domain, the dimension of the space, the H\"older regularity of the boundary data, and/or the $L^p$-integrability of the source data; as a byproduct of our approach, we shall show that in the generic estimate \eqref{eq:typical} one can simply drop the term $\|u\|_{L^p(D)}$.
Obtaining such constants is crucial for applications, in particular in two instances:

\begin{itemize}
\item {\it Numerical estimates:} In many instances, in the so-called computer-assisted proofs, one can prove theoretically the existence of a solution of a PDE, based on an approximate/numerically constructed one, as long as one can explicitly bound the numerical error, see for instance an introduction to this approach in \cite{rump2010verification} with an application in  \cite{breuer2003multiple} in the context of determining multiple solutions and \cite{chen2025singularity} for determining singularities in the solutions.  Such explicit constants are computed for $H^2$ estimates in \cite{plum1992explicit} only in dimensions $2$ and $3$.
\item {\it Breaking the curse of high dimensions:} A significant amount of work has been recently devoted to analyzing and implementing neural networks as universal approximators for solutions to various problems, in particular to high-dimensional PDEs. 
A challenging aim in this direction is proving, if possible, that neural networks (shallow or deep) can approximate a given continuous function $u$, defined on a bounded domain in $\mathbb{R}^d$, to any prescribed accuracy error $\varepsilon$, with a computational cost that scales at most polynomial with respect to $d$ and $\varepsilon^{-1}$; in this respect, having explicit and quantitative information about the regularity of $u$ is in most (if not all) cases crucial.
For example, in \cite[Section 2.3]{gonon2022uniform}, Sobolev type estimates with explicit track of the constants in terms of the dimension $d$ are derived, in order to prove that the heat equation in $\mathbb{R}^d$ can be uniformly approximated in a bounded domain by neural networks whose number of parameters grow at most polynomially with respect to $d$ and the reciprocal of the prescribed approximation precision.
Furthermore, a recent result \cite[Corollary 1.2]{shen2021neural} underpins the role of H\"older spaces in the theory of approximation with neural networks, proving the following: {\it If $u:[0,1]^d\rightarrow\mathbb{R}$ is $\alpha$-H\"older continuous with H\"older constant $\lambda$, then for every $N\geq 1$ one can construct a three-hidden-layer neural network $\Phi$ with width $\max(d,N)$ and $2(d+N+1)$ non-zero parameters such that
\begin{equation}\label{thm:three layers}
    |u(x)-\Phi(x)|\leq 3\lambda (2\sqrt{d})^\alpha 2^{-\alpha N}, \quad x\in [0,1]^d.
\end{equation}}
Note that if $u$ is $\alpha$-H\"older on $\overline{D}\subset [0,1]^d$, then $u$ can be extended to an $\alpha$-H\"older function on $[0,1]^d$ with the same H\"older constant, using e.g. \cite[Remark 2.31]{BCILPZ2022}. 
Consequently, \eqref{thm:three layers} can be directly used in conjunction with our main results Theorems \ref{thm:u_g_ball}, \ref{thm:u_cone_2D}, \ref{thm:u_wedge_3D}, and \ref{thm:u_cone_3D} presented in \Cref{subsec:mainresults} below, in order to approximate the solution $u$ to \eqref{eq:PDE} by three-hidden-layers neural networks, with similar quantitative estimates as in \eqref{thm:three layers}, and with $\lambda$ and $\alpha$ explicitly expressed in terms of the dimension of the space, the geometry of the domain, and the regularity of the data.

\end{itemize}

In order to obtain an elegant, self-contained result that leads to H\"older quantitative estimates with explicit constants for the inhomogeneous Poisson equation \ref{eq:PDE},
the main approach we take is a probabilistic one, building on the classical representation of the solution in terms of data $g$  and $f$ and the exit time of the Brownian motion from the domain;  then, in analogy to the Harnack inequality typically employed in the analytical methods in order to prove H\"older regularity, we rely on the classical maximal inequality of Doob and Burkholder-Davis-Gundy (BDG) inequality, which we recall in detail in the Appendix. 
We mention that the advantage of using the latter inequalities as a substitute for Harnack inequality comes from the fact that the constants involved in Doob's and BDG inequalities are explicit and sharp in terms of the involved exponents and the dimension $d$; see \Cref{s:appendix} for details.

 In a nutshell, our approach to derive quantitative boundary H\"older estimates for \eqref{eq:PDE}, builds on the following two main steps:
\begin{enumerate}
    \item[Step 1] We first consider an arbitrary bounded open set $D\subset \mathbb{R}^d$ and prove that the boundary $\alpha$-H\"older regularity of the solution $u$ to \eqref{eq:PDE} is controlled, with explicit constants, by the H\"older regularity of the Dirichlet data $g$, the $L^p$-norm of the source $f$, and the function $v_{D,\alpha}$ defined in \eqref{eq:v_Dalpha}; see \Cref{prop:general_estimate} for such an estimate. 
    We note that $v_{D,\alpha}$ depends only on the operator/diffusion process (in our case the Laplacian/Brownian motion) and on the domain.
    \item[Step 2] Based on Step 1, given a a domain $D$ that has certain regularity properties, we derive corresponding explicit estimates of the function $v_{D,\alpha}$ in terms of the distance function to the boundary, the regularity of the domain $D$, and $\alpha$.
    In Section~\ref{sec:exitgrl} we provide two different approaches to explicitly estimate $v_{D,\alpha}$ in general: (1) By controlling the decay of the oscillation, in Subsection~\ref{subsec:reversedoubling} (reminiscent, in a probabilistic re-interpretation, of the approach of \cite{bdrHolderJMPA}), and (2) By using Ito's formula, martingales, and barrier-type functions. 
    We apply our above mentioned methods, in the last three sections, to the case of domains satisfying exterior ball conditions, respectively exterior cone/wedge conditions. 
\end{enumerate}

The technical statements of the main results are provided subsequently in  Subsection~\ref{subsec:mainresults}. 

An immediate consequence of the above estimates are explicit pointwise estimates for gradient and the Green function near the boundary in terms of the distance to the boundary, provided in  Subsection~\ref{subsec:furtherconseq}.

\medskip

\noindent{\bf Possible extension to more general operators.} 
The approach developed in this paper for the quantitative boundary H\"older estimate for the inhomogeneous Poisson equation \eqref{eq:PDE} paves the path to a similar quantitative analysis for more general elliptic PDEs. 
We expect that (most part of) the work done herein could be extended to generator operators $(L,D(L))$ in $D$ that correspond to a strong Markov process that solves the martingale problem.

\paragraph{Organization of the paper.} 
In \Cref{subsec:mainresults} we recall the notion of {\it probabilistic solution} to the Poisson problem \eqref{eq:PDE} and we present the main concrete outcomes of the general techniques developed in this work.
More precisely, in Theorems \ref{thm:u_g_ball}, \ref{thm:u_cone_2D}, \ref{thm:u_wedge_3D}, and \ref{thm:u_cone_3D}, we derive explicit H\"older estimates for the solution to Poisson problem \eqref{eq:PDE} in Euclidean domains that satisfy an exterior sphere, cone, or wedge condition.

In \Cref{subsec:furtherconseq} we apply the main results presented in \Cref{subsec:mainresults} to derive explicit estimates at the boundary for the gradients of harmonic functions and for the Green function, in a domain that satisfy an exterior sphere, cone, or wedge condition.
We do this in \Cref{coro:gradient_u_g} and \Cref{coro:G_D_alpha}, respectively.

In \Cref{S:general estimates} we first introduce (in \Cref{subsec:exittimegeneral}) the main functions $v_{D,\alpha}$ and $h_{D,\Gamma_0}$ which depend only on the Laplace operator (or the corresponding diffusion process, namely the Brownian motion) and the domain $D$ and discuss their basic properties.
Then, in \Cref{subsec:ugufexit} we show that the H\"older regularity of the solution to the Poisson equation \eqref{eq:PDE} can be explicitly analyzed in terms of $v_{D,\alpha}$, $h_{D,\Gamma_0}$, and the regularity of the data. We do this in arbitrary bounded open sets in $\mathbb{R}^d$, in \Cref{prop:u_g} and \Cref{prop:u_f}.

In \Cref{sec:exitgrl} we derive two generic approaches to estimate the aforementioned functions $v_{D,\alpha}$, $h_{D,\Gamma_0}$ in terms of the distance function to the boundary.
The first approach is presented in \Cref{subsec:reversedoubling} and focuses on controlling the decay of the oscillation of $v_{D,\alpha}$ and $h_{D,\Gamma_0}$, and is reminiscent to the approach of \cite{bdrHolderJMPA}; the main results here are \Cref{thm:Phi_hv} and its Corollaries \ref{coro:u_g_delta} and \ref{coro:u_f_delta}.
The second strategy is presented in \Cref{ss:second_approach}, and it relies on Ito's formula applied for some suitable barrier function, the martingale part being further estimated from below and above using martingale sharp inequalities; the main result here are \Cref{prop:W} and \Cref{coro:W_special}.

In \Cref{S:ball} we apply the second strategy developed in \Cref{sec:exitgrl} in order to derive explicit and sharp estimates for $v_{D,\alpha}$ and $h_{D,\Gamma_0}$ for domains that satisfy an exterior sphere condition; the main result here is \Cref{coro:annulus}.

In \Cref{S:cone-wedge} we apply the two strategies developed in \Cref{sec:exitgrl} in order to derive explicit (and, as much as possible, sharp) estimates for $v_{D,\alpha}$ and $h_{D,\Gamma_0}$ for domains that satisfy an exterior cone, or wedge condition.
More precisely, in \Cref{ss:cone_2d} we apply the second strategy developed in \Cref{sec:exitgrl} to derive explicit estimates for $v_{D,\alpha}$ and $h_{D,\Gamma_0}$ for planar domains that satisfy an exterior cone condition; the main result here is \Cref{coro:D-cone-2D}.
Then, in \Cref{ss:wedge} we apply the second strategy developed in \Cref{sec:exitgrl} to derive explicit estimates for $v_{D,\alpha}$ and $h_{D,\Gamma_0}$ for three-dimensional domains that satisfy an exterior wedge condition; the main result here is \Cref{coro:D-wedge-3D}.
Finally, in \Cref{ss:cone_3d} we apply the first strategy developed in \Cref{sec:exitgrl} to derive explicit estimates for $v_{D,\alpha}$ and $h_{D,\Gamma_0}$ for domains in $\mathbb{R}^d,d\geq 3$, which satisfy an exterior cone condition; the main result here is \Cref{coro:delta_omega}.

Some fundamental inequalities for martingales are collected in the Appendix, in which we also list the paper specific notations and the explicit values of the constants involved in the main body of the paper.

\subsection{The main results}\label{subsec:mainresults}

We give the precise assumptions on $f$ and $g$ in $\mathbf{H_f}$ and $\mathbf{H_g}$ below.
Throughout, for $1\leq p\leq \infty$, $L^p(D)$ denotes the usual Lebesgue space, namely the Banach spaces of all (equivalence classes of) $p$-integrable functions on $D$ with respect to the Lebesgue measure $\lambda$; the corresponding $p$-norm is denoted by $\|\cdot\|_{L^p(D)}$.

\begin{rem} We note that if $f\in L^p(D)$ then the maximal Sobolev space regularity of the solution of \eqref{eq:PDE}  that one can expect is that $u\in W^{2,p}(D)$. Then, out of Sobolev embeddings, in order to have $W^{2,p}(D)\subset C^{0,\alpha}(D)$ for some $\alpha>0$ we need to have
$\alpha=2-\frac{d}{p}>0$, so in particular we need $p>d/2$, a minimal integrability condition for the forcing $f$ that we will assume in the main results. 
\end{rem}

Furthermore, for a vector $x\in \mathbb{R}^d$ we denote by $\|x\|$ its Euclidean norm, whilst the diameter of a bounded set $D\subset \mathbb{R}^d$ is given by
\begin{equation*}
    {\sf diam}(D):=\sup\{\|x-y\| : x,y\in D\}.
\end{equation*}
Also, throughout we use the notations 
\begin{equation}\label{eq:min,max}
    a\wedge b:=\min(a,b),  \quad a\vee b:=\max(a,b), \quad a^+:=0\vee a,\quad \quad a^-:=0\vee (-a) \quad a,b\in \mathbb{R}.
\end{equation}

Let $(B(t)=(B^{(1)}(t),\dots,B^{(d)}(t)))_{t\geq 0}$ be a standard Brownian motion on $\mathbb{R}^d$, defined on a filtered probability space $(\Omega,\mathcal{F},(\mathcal{F}_t)_{t\geq 0},\mathbb{P})$ satisfying the usual hypotheses.
Further, we set
\begin{equation}
    B^x(t):=x+B(t), \quad t\geq 0, x\in \mathbb{R}^d,
\end{equation}
whilst for each $A\in \mathcal{B}(\mathbb{R}^d)$ we denote by $\tau_A^x$ the {\it hitting time} of $A$ by the Brownian motion $(B^x(t))_{t\geq 0}$, namely
\begin{equation}
    \tau^x_A:=\inf\{t>0 : B^x(t)\in A\},\quad x\in \mathbb{R}^d.
\end{equation}
It is well known that $\tau_A^x$ is an $\left(\mathcal{F}_t\right)$-stopping time; see e.g. \cite{Ba10}.
\begin{defi}\label{defi:prob_solution}
Let $f:D\rightarrow \mathbb{R}_+,\; g: \partial D \rightarrow \mathbb{R}_+$ be measurable.
By {\rm the probabilistic solution} to problem \eqref{eq:PDE} we mean the function defined point-wise on $\overline{D}$ through
\begin{equation}\label{eq:prob_solution}
   u:\overline{D}\longrightarrow \mathbb{R}, \quad u(x)=\mathbb{E}\left\{g\left(B^x({\tau^x_{\partial D}})\right)\right\} + \mathbb{E}\left\{\int_0^{\tau_{\partial D}^x}f\left(B^x(t)\right) dt\right\}, \quad x\in \overline{D}.
\end{equation}
By setting $f=0$ and then $g=0$ in the above formula, let us further denote
\begin{equation}\label{eq:prob_sol_split}
    u_g(x):=\mathbb{E}\left\{g\left(B^x({\tau^x_{\partial D}})\right)\right\}, \quad \mbox{and} \quad u_f(x):=\mathbb{E}\left\{\int_0^{\tau_{\partial D}^x}f\left(B^x(t)\right) dt\right\}, \quad x\in \overline{D}.
\end{equation}
Consequently, $u=u_g+u_f$ on $\overline{D}$.

\noindent{If} $f$ and $g$ are real valued, then 
\begin{equation*}
    u_g:=u_{g^+}-u_{g^-}, \quad u_g:=u_{f^+}-u_{f^-}, \quad u:=u_g+u_f,
\end{equation*}
whenever $u_{g^{\pm}}, u_{f^{\pm}}$ are finite.
\end{defi}
\begin{rem}\label{rem:PDEcorrespondence}
The following facts are well known:

\begin{enumerate}
    \item[(i)] If $g\in C(\partial D)$ then $u_g$ is the unique local solution to the Dirichlet problem
    \begin{equation} \label{eq:PDE_g}
    \Delta u_g=0 \textrm{ in } D, \quad u=g \textrm{ on }\partial D,
    \end{equation}
    in the sense of e.g. \cite{ChZh95}, namely $u_g$ is the unique function in $C(D)\cap H^{1}_{\sf loc}(D)$ such that 
    \begin{equation*}
     \lim\limits_{D\ni x\to x_0\in \partial D}u_g(x)=g(x_0) \quad \mbox{for every } x_0 \mbox{ regular point (\mbox{see \cite{ChZh95}} for the definition) on } \partial D. 
    \end{equation*}
    \item[(ii)] If $f$ is bounded then $u_f$ coincides with the integral solution to the homogeneous Poisson problem
    \begin{equation} \label{eq:PDE_f} 
    \frac{1}{2}\Delta u=-f,\textrm{ in } D, \quad u=0 \textrm{ on }\partial D,
    \end{equation}
    defined by convoluting $2f$ with the Green function $G_D\geq 0$ of $-\Delta_0$, the Laplace operator with zero boundary conditions on the domain $D$, namely
\begin{equation}
    u_f(x)=2\int_D G_D(x,y) f(y) dy, \quad x\in D.
\end{equation}
\end{enumerate}
\end{rem}

The main goal of the paper is to investigate by probabilistic tools the boundary H\"older regularity of $u$, providing explicit bounds in terms of the geometry of $D$ and the dimension $d$ of the underlying space. 
To this end, for a bounded opend set $D\in \mathbb{R}^d$, a point $x_0\in \partial D$, and an open boundary portion $\Gamma_0\subset\partial D$ that contains $x_0$, we introduce the following notations:
\begin{align}
&|g|_\alpha^{x_0}:=\sup\limits_{\partial D\ni x\neq x_0\in \partial D}\frac{|g(x)-g(x_0)|}{\|x-x_0\|^\alpha}, \quad |g|_\alpha^{x_0,\Gamma_0}:=\sup\limits_{\Gamma_0\ni x\neq x_0\in \partial D}\frac{|g(x)-g(x_0)|}{\|x-x_0\|^\alpha},\quad \alpha>0,\label{eq:notation_holder_seminorms_1}\\
&|g|_\alpha:=\sup\limits_{\partial D \ni x\neq x'\in \partial D}\frac{|g(x)-g(x')|}{\|x-x'\|^\alpha}, \quad \alpha\in(0,1].\label{eq:notation_holder_seminorms_2}
\end{align}

Throughout the paper we shall assume the following regularity for $f$ and $g$:

\medskip
{$\mathbf{H_f}$.} In all the results that follows, the source term $f:D\rightarrow\mathbb{R}$ is required to be in $L^\gamma(D)$, for some $\gamma \in (d/2,\infty]$.

{$\mathbf{H_g}$.} The boundary data $g:\partial D\rightarrow \mathbb{R}$ is in principle required to be $\alpha$-H\"older continuous, i.e. $|g|_\alpha<\infty$ for some $\alpha\in(0,1)$.
However, for some results the above global H\"older regularity is relaxed to $|g|_\alpha^{x_0}<\infty$ (resp. $|g|_\alpha^{x_0,\Gamma_0}<\infty$) for some fixed boundary point $x_0\in \partial D$ (resp. $r_0>0$), and in such cases we sometimes treat the case $0<\alpha\neq 1$ as well.

Our forthcoming results are heavily based on decay estimates for the function $v_{D,\alpha}, \alpha>0$, given by
\begin{equation}\label{eq:v_Dalpha}
v_{D,\alpha}(x)=\mathbbm{E}\left\{\left(\tau^x_{\partial D}\right)^\alpha\right\}, \quad x\in \overline{D},
\end{equation}
as $x$ approaches the boundary.
When $\alpha=1$ we shall drop the index $\alpha$ from notation, and in this case $v_D$ is {\it the probabilistic solution} to the problem
\begin{equation}\label{e:v}
\frac{1}{2}\Delta v_{D}=-1 \textrm{ in } D, \quad v_{D}=0, \textrm{ on }\partial D, \quad \mbox{in the sense of } \Cref{defi:prob_solution}.
\end{equation}
Recall that for a general domain $D$, the boundary condition $v_D(z)=0$ for the probabilistic solution is fulfilled if and only if $z\in \partial D$ is a regular point for $D$ (see \Cref{rem:PDEcorrespondence}).
When $\alpha\neq 1$, note that $v_{D,\alpha}$ can be represented as
\begin{equation*}
    v_{D,\alpha}(x)=\int_0^\infty \alpha t^{\alpha-1} \mathbb{P}(\tau^x_D>t) dt=\int_0^\infty \alpha t^{\alpha-1} u(t,x) dt, \quad \lambda\mbox{-a.e. } x\in D,
\end{equation*}
where $u$ is the $L^2$-solution of the parabolic equation
\begin{equation*}
    \frac{d}{dt}u(t)=\frac{1}{2}\Delta u(t), \;t>0, x\in D, \quad u(0,x)=1, \quad u(t,\cdot)=0 \mbox{ on } \partial D, t>0,
\end{equation*}
namely $u(t)=e^{t\Delta_0}u_0$ with $u_0=1_D\in L^2(D)$, where $\left(\Delta_0, D(\Delta_0)\right)$ is the Dirichlet Laplacian on $L^2(D)$; see for example \cite[Section 4, a)]{ma2012introduction}.

The following results shows that the boundary H\"older regularity of the solution $u$, in an arbitrary open and bounded set, can be reduced to the regularity of the function $v_{D,\alpha}$. 
It is a shortened version of Propositions \ref{prop:u_g} \& \ref{prop:u_f} from below, which we prove in the main body of the paper.
\begin{prop}[General estimate; see Propositions \ref{prop:u_g} \& \ref{prop:u_f} below]\label{prop:general_estimate}
Let $0<\alpha\neq 1$, $x_0\in \partial D$, $C(\alpha, d)$ be the constant appearing in \eqref{e:bmLBDG}, $p>1, 1/p+1/q=1,\gamma>\frac{dq}{2}$, and $C(d,D,\gamma,q)$ be given by\eqref{eq:C_f}.
Then, for every $x\in D$ we have
    \begin{equation}
        \begin{split}
        |u(x)-g(x_0)|
        \leq 
        &|g|^{x_0}_\alpha 2^{(\alpha-1)^+}\left[C(\alpha,d) v_{D,\alpha/2}(x)
        +\left\|x-x_0\right\|^{\alpha}\right]\\
        &+\|f\|_{L^\gamma(D)}(1_{\gamma=\infty}+1_{\gamma<\infty}C(d,D,\gamma,q))v_D(x)^{1/p}.
        \end{split}
    \end{equation}
\end{prop}

Based on the above general estimate and also on the more detailed estimates from \Cref{prop:u_g} and \Cref{prop:u_f}, in what follows we derive explicit boundary H\"older estimates for domains that satisfy an exterior sphere condition for $d\geq 2$, an exterior cone condition for $d=2$, an exterior wedge condition for $d=3$, or an exterior cone condition for $d\geq 3$, respectively. 
As one can easily see from the above general estimate, the effort reduces to derive explicit estimates for $v_{D,\alpha/2}$ in terms of the distance to the boundary.

\medskip
For two numbers $0<r<R$ and a point $a\in\R^d$ we define the {\it annulus}
\begin{equation}\label{eq:annulus}
     A(a,r,R):=\{ x\in \R^d, r<|x-a|<R \}.
\end{equation}
Also, throughout the paper $B(a,r)$ denotes the Euclidean open ball of center $a\in \mathbb{R}^d$ and radius $r>0$.

Our first quantitative H\"older estimates concern domains that satisfy an {\it exterior sphere condition} as depicted below:

\begin{center}
\includegraphics[width=6cm, height=6cm]{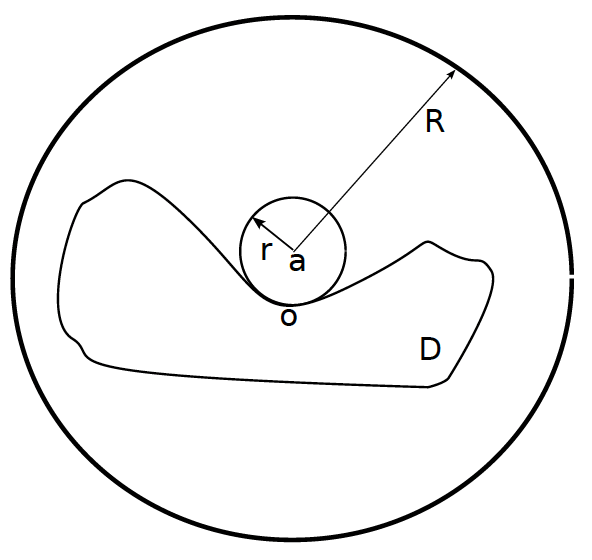}
\centerline{ Figure 1: an example of domain satisfying the exterior sphere condition}
\end{center}

\begin{thm}[Exterior sphere condition: $d\geq 2$]\label{thm:u_g_ball}
Let $D$ be a bounded open subset in $\mathbb{R}^d$, and $x_0=0\in \partial D$ such that there exists an annulus $A(a,r,R)$ with $0\in 
\partial B(a,r)$ and $D\subset A(a,r,R)$ as in Figure $1$. 
Then the following assertions hold:
\begin{enumerate}
    \item[(i)] Let $u_g$ be the solution to \eqref{eq:PDE_g}, $0<\alpha<1$, $C(\alpha,d)$ be the explicit constant given by \eqref{e:bmLBDG}, $c_\alpha$ be the universal constant from \Cref{t:BDG}, and $C(\alpha,d,R,r)$ be given by
    \begin{equation}\label{eq:C(alpha,d,R,r)}
       C(\alpha,d,R,r)
       :=C(\alpha,d) e^{\alpha}\left[\frac{4}{(1-\alpha)c_\alpha} +\frac{(R-r)^{\alpha/2}}{r^{\alpha/2}}d^{\alpha/2}\right]+1.
    \end{equation}
    Then for every $x\in D$
    \begin{align}\label{eq:u_g_holder_x0_sphere}
        |u_g(x)-g(0)|
        &\leq \left[C(\alpha,d,R,r)1_{[0,r/d]}(\|x\|)+{\sf diam}(D)^\alpha \frac{d^{\alpha/2}}{r^\alpha}1_{(r/d,\infty)}(\|x\|)\right]|g|^0_\alpha \|x\|^{\alpha}\\
        &\in \mathcal{O}\left(\|x\|^{\alpha}d^\alpha\right).
    \end{align}
    Moreover, if there exists $r,R$ such that for every $x_0\in \partial D$ there exists an annulus $A(a,r,R)$ with $x_0\in \partial B(a,r)$ and $D\subset A(a,r,R)$, then
    for every $x,y\in D$
    \begin{align}\label{eq:u_g_holder_global_sphere}
        |u_g(x)-u_g(y)|
        &\leq \left[2C(\alpha,d,R,r)1_{[0,r/d]}(\|x-y\|)+{\sf diam}(D)^\alpha \frac{d^{\alpha/2}}{r^\alpha}1_{(r/d,\infty)}(\|x-y\|)\right]|g|_\alpha\left\|x-y\right\|^{\alpha}\\
        &\in \mathcal{O}\left(\|x-y\|^{\alpha}d^\alpha\right).
    \end{align}
    \item[(ii)] Let $u_f$ be the solution to \eqref{eq:PDE_g} and $x\in D$. 
    If $f\in L^\infty(D)$ then
    \begin{equation}
        |u_f(x)|\leq \|x\|\|f\|_{L^\infty(D)}\frac{(R-r)R}{r}.
    \end{equation}
    Furthermore, let $C(d,D,\gamma,q)$ be the explicit constant given by \eqref{eq:C_f}. 
    If $p\geq 1, 1/p+1/q=1$, $f\in L^\gamma(D)$, with $\gamma/q>d/2$, then for every $x\in D$
    \begin{align}
        |u_f(x)|
        &\leq  C(d,D,\gamma,q)\|x\|^{1/p}\left(\frac{(R-r)R}{r}\right)^{1/p} \|f\|_{L^\gamma(D)}\label{eq:u_f_holder}\\
        C(d,D,\gamma,q)& \in \mathcal{O}\left( \left(\frac{d^2(\gamma-q)}{2\gamma-qd}\right)^{1/q}{\sf diam}(D)^{\frac{2\gamma-qd}{pq}}\right)\label{eq:C in O}.
    \end{align}
\end{enumerate}
\end{thm}

\medskip
\noindent{\it Proof of \Cref{thm:u_g_ball} in \Cref{proof:thm:u_g_ball}}.
\medskip

\begin{rem}[Sharpness of $\mathcal{O}\left(\|x\|^\alpha d^\alpha\right)$]
By \Cref{thm:u_g_ball}, (i), if $g$ is $\alpha$-H\"older on $\partial D$, $\alpha \in (0,1)$, and $D$ satisfies a uniform exterior sphere condition, then $u_g$ is also $\alpha$-H\"older continuous up to the boundary.
Moreover, the corresponding $\alpha$-H\"older seminorm $|u_g|_\alpha$ increases at most like $d^\alpha$ with respect to the dimension $d$.
It turns out that the power $\alpha$ is sharp.
To see this, let
\be
u(x)=\frac{\|x\|^{2-d}-R^{2-d}}{r^{2-d}-R^{2-d}}, \quad x\in D:= A(0,r,R),
\ee
which is the solution of the equation 
\begin{equation*}
    \Delta u=0 \quad \mbox{on D}, \quad u(x)=0 \mbox{ for } \|x\|=r, \quad \mbox{and} \quad u(x)=1 \mbox{ for } \|x\|=R. 
\end{equation*}
Let $|u|_\alpha$ be given by
\be
|u|_\alpha:=\sup_{x,y\in D} \frac{|u(x)-u(y)|}{\|x-y\|^\alpha}.
\ee
We aim to obtain a lower bound on $|u|_\alpha$ in terms of $d$.
To this end we note that we can take $ x, y_\eps\in \overline{D}$, with $\| x\|=r$ and $ y_\eps=(1+\eps) x$ for $\varepsilon\in [0,\frac{R}{r}-1)$. 
Then we have
\bea
|u|_\alpha
&\ge \sup_{\eps\in  [0,\frac{R}{r}-1)} \frac{|u(x)-u( y_\eps)|}{|x-y_\eps|^\alpha}
=\sup_{\eps\in  [0,\frac{R}{r}-1)}\frac{1-(1+\eps)^{2-d}}{\eps^\alpha}\frac{1}{r^\alpha(1-(\frac{r}{R})^{d-2})}\non\\
&\ge\frac{1-(1+d^{-1})^{2-d}}{d^{-\alpha}}\frac{1}{r^\alpha(1-(\frac{r}{R})^{d-2})} \quad \mbox{ for any } d \mbox{ large enough}.
\eea 
By setting
$
M_d:=\frac{1-(1+d^{-1})^{2-d}}{d^{-\alpha}}
$
we have
$$
\lim_{d\to\infty}\frac{M_d}{d^\alpha}=1-e^{-1}, \quad \mbox{and thus} \quad \lim_{d\to\infty}\frac{|u|_\alpha}{d^\alpha}\ge \frac{1-e^{-1}}{r^\alpha}.
$$
\end{rem}

\bigskip
In what follows we provide explicit H\"older estimates for less smooth domains, namely for those that admit an {\it exterior cone condition}. 
However, to get such explicit estimates, the cases $d=2$ and $d\geq 3$ turn out to require separate approaches with different outcomes. 
We shall present first the case $d=2$. 
Moreover, before proceeding to domains satisfying an exterior cone condition for $d\geq 3$, we shall present explicit estimates for another type of domains for $d=3$, namely those which satisfy an {\it exterior wedge condition}. 
We do all this systematically, below.

First of all, for an angle $\omega\in [0,\pi]$ and a radius $r>0$ we define the  circular cone $\mathcal{C}(\omega,r)\in \mathbb{R}^d$ (see Figure $2$) by
\begin{equation}\label{eq:C_2d}
    \mathcal{C}(\omega,r):=\left\{x=(x_1,\dots,x_d)\in \mathbb{R}^d : \|x\|\leq r, \;x_1/\|x\|\geq \cos{ (\omega) } \right\}.
\end{equation}
By a slight abuse of notation, in what fallows we shall say that $\mathcal{C}(\omega,r)$ is a cone with vertex $x_0\in \mathbb{R}^d$, angle $\omega$ and radius $r$, whenever $\mathcal{C}(\omega,r)$ is isometric to the cone represented by \eqref{eq:C_2d}. 

In the next result we consider domains in $\mathbb{R}^2$ that satisfy the exterior cone condition as illustrated in Figure $2$ below

\begin{center}
\includegraphics[width=12cm, height=6cm]{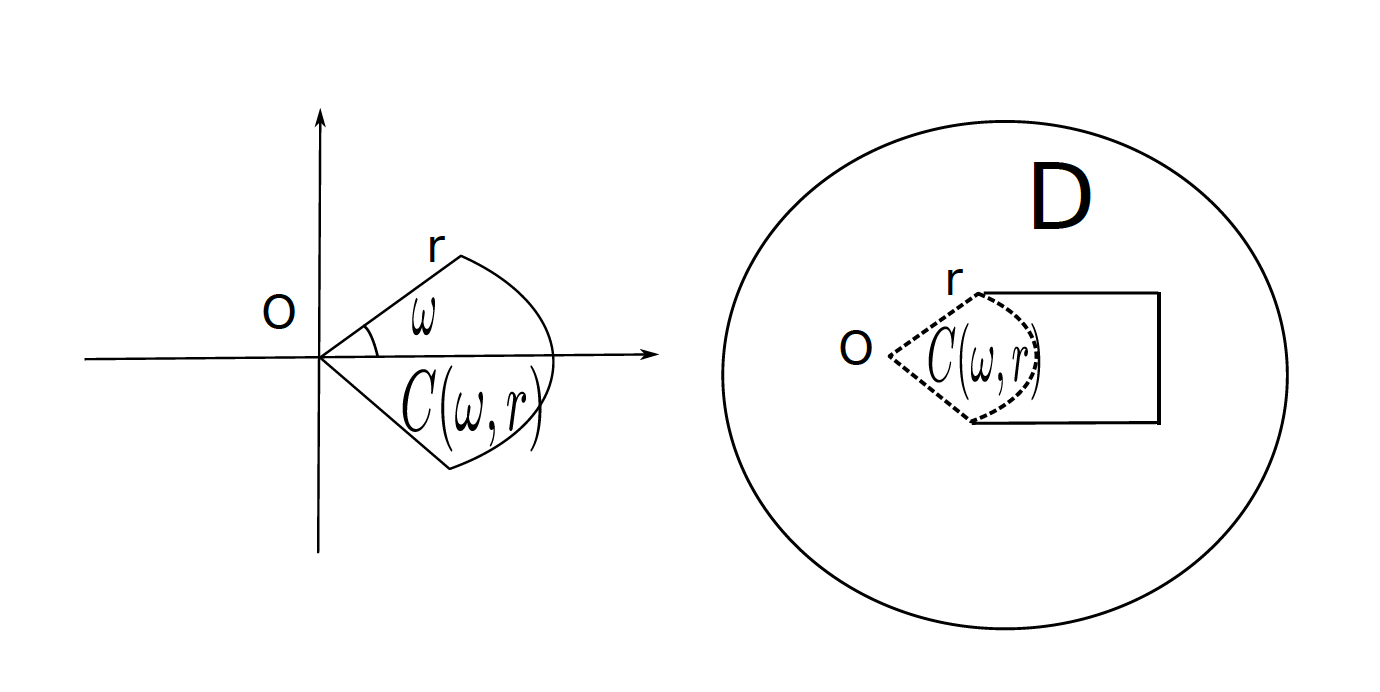}
\centerline{ Figure 2: a cone and a domain satisfying an exterior cone condition}
\end{center}
We also set
\begin{equation}\label{eq:tilde_omega}
    \tilde{\omega}:=\dfrac{\pi}{2(\pi-\omega)}, \quad \omega \in [0,\pi).
\end{equation}

\begin{thm}[Exterior cone condition: $d=2$]\label{thm:u_cone_2D}
Let $D$ be a bounded open subset in $\mathbb{R}^2$, and $0\in \partial D$ such that there exists a cone $\mathcal{C}(\omega,r)$ with vertex $0$, radius $r$, and angle $\omega\in (0,\pi/2]$, such that $\mathcal{C}(\omega,r)\subset \mathbb{R}^2\setminus D$.
\begin{enumerate}
    \item[(i)] Let $u_g$ be the solution to \eqref{eq:PDE_g}, $0<\alpha<1$, $C(\alpha,2)$ be the explicit constant given by \eqref{e:bmLBDG}, and $c_\alpha$ be the universal constant from \Cref{t:BDG}.
   Then for every $x\in D$ we have
    \begin{align}\label{eq:u_g_holder_x0}
        |u_g(x)-g(0)|
        &\leq  
        |g|^{0}_\alpha\left\{C_5\left(\frac{\|x\|}{r}\right)^{\alpha\tilde{\omega}}\; \left[\log{\left(\frac{r}{\|x\|}\right)}\right]^\alpha\right.\\
        &\qquad \left.+\left(\frac{\|x\|}{r}\right)^{\tilde{\omega}}\; \left[C_6\log{\left(\frac{r}{\|x\|}\right)}
        +e\frac{{\sf diam}(D)^\alpha}{2^{\alpha/2}}1_{[re^{-1/\tilde{\omega}},\infty)}(\|x\|) \right]
        +\|x\|^\alpha\right\}\\
        &\in \mathcal{O}\left(\left(\frac{\|x\|}{r}\right)^{\alpha\tilde{\omega}}\; \left[\log{\left(\frac{r}{\|x\|}\right)}\right]^\alpha\right)
    \end{align}
    where
    \begin{align}\label{eq:C_1-C_3}
        &C_5= C(\alpha,2)\tilde{\omega}^\alpha\frac{8r^\alpha}{(1-\alpha)c_\alpha},\quad C_6= C(\alpha,2)\frac{\left({\sf diam}(D)+r\right)^{\alpha}\tilde{\omega}}{2^{\alpha/2}}.
    \end{align}
    Moreover, if there exist $\omega\in (0, \pi/2]$ and $r>0$ such that for any $x_0\in \partial D$ there exists a cone $\mathcal{C}(\omega,r)$ of vertex $x_0$, angle $\omega$, and radius $r$, such that $\mathcal{C}(\omega,r)\subset \mathbb{R}^2\setminus D$, then
    for every $x,y\in D$
    \begin{align}\label{eq:u_g_holder_global}
        |u_g(x)-u_g(y)|
        &\leq  
        2|g|_\alpha\left\{C_5\left(\frac{\|x-y\|}{r}\right)^{\alpha\tilde{\omega}}\; \left[\log{\left(\frac{r}{\|x-y\|}\right)}\right]^\alpha
        \right.\\
        &\qquad \left.+\left(\frac{\|x-y\|}{r}\right)^{\tilde{\omega}}\; \left[C_6\log{\left(\frac{r}{\|x-y\|}\right)}
        +e\frac{{\sf diam}(D)^\alpha}{2^{\alpha/2}}1_{[re^{-1/\tilde{\omega}},\infty)}(\|x-y\|) \right]
        +\|x-y\|^\alpha\right\}\\
        &\in \mathcal{O}\left(\left(\frac{\|x-y\|}{r}\right)^{\alpha\tilde{\omega}}\; \left[\log{\left(\frac{r}{\|x-y\|}\right)}\right]^\alpha\right),
    \end{align}
    where $C_5-C_6$ are as in \eqref{eq:C_1-C_3}.
    \item[(ii)] Let $u_f$ be the solution to \eqref{eq:PDE_g} and $x\in D$. 
    If $f\in L^\infty(D)$ then
    \begin{equation}\label{eq:u_f_infty_2d_cone}
        |u_f(x)|\leq \|f\|_{L^\infty(D)}\left(\frac{\|x\|}{r}\right)^{\tilde{\omega}}\; \left[C_7\log{\left(\frac{r}{\|x\|}\right)}+e\frac{{\sf diam}(D)^2}{2}1_{[re^{-1/\tilde{\omega}},\infty)}(\|x\|)\right],
    \end{equation}
    where
    \begin{equation}
        C_7=C_3+C_4, \quad \mbox{ where } C_3,C_4 \mbox{ are as in } \Cref{coro:D-cone-2D} \mbox{ with } \alpha=2.
    \end{equation}
    Furthermore, if $p\geq 1, 1/p+1/q=1$, $f\in L^\gamma(D)$, with $\gamma/q>1$, then
    \begin{equation}\label{eq:u_f_gamma_2d_cone}
        |u_f(x)|\leq  C(2,D,\gamma,q)\|f\|_{L^\gamma(D)}\left(\frac{\|x\|}{r}\right)^{\tilde{\omega}/p}\; \left[C_7\log{\left(\frac{r}{\|x\|}\right)}+e\frac{{\sf diam}(D)^2}{2}1_{[re^{-1/\tilde{\omega}},\infty)}(\|x\|)\right]^{1/p},
    \end{equation}
    where $C(2,D,\gamma,q)$ is given by \eqref{eq:C_f}.
\end{enumerate}
\end{thm}

\medskip
\noindent{\it Proof of \Cref{thm:u_cone_2D} in \Cref{proof:u_cone_2D}.}
\medskip

As announced above, before turning our attention to domains satisfying an exterior cone condition for $d\geq 3$, we investigate the class of domains satisfying an {\it exterior wedge condition}, for $d=3$. Such domains are in principle more regular that the ones satisfying an exterior cone condition, but can be less regular than Lipschitz domains. 

An example of domain that is not Lipschitz but satisfies the exterior wedge condition is the {\it double brick} domain illustrated below, in Figure $3$:

\begin{center}
\includegraphics[width=7cm, height=6cm]{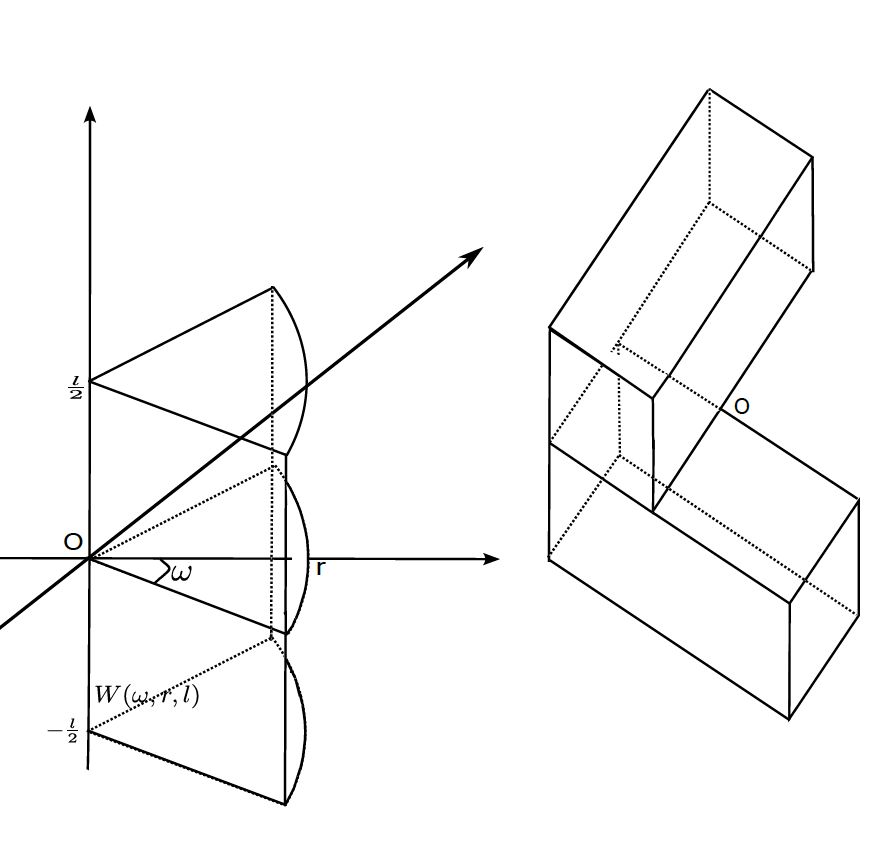}
\centerline{Figure 3: a wedge in $3D$ and {\it the double brick}, a non-Lipschitz domain that satisfies the exterior wedge condition}
\end{center}

As illustrated above, for $r,l>0$, $\omega \in [0,\pi]$, 
we consider the {\it wedge} $\mathcal{W}(\omega,r,l)\subset \mathbb{R}^3$ given by
\begin{equation}\label{eq:wedge_2d}
    \mathcal{W}(\omega,r,l)=\mathcal{C}(\omega,r)\times (-l/2,l/2), 
\end{equation}
where $\mathcal{C}(\omega,r)$ is the cone with vertex $0$ given by \eqref{eq:C_2d}.
By an abuse of notation, in what fallows we shall say that $\mathcal{W}(\omega,r,l)$ is a wedge centered at $x_0\in \mathbb{R}^3$, angle $\omega$, radius $r$ and length $l$, whenever $\mathcal{W}(\omega,r,l)$ is isometric to the wedge represented by \eqref{eq:wedge_2d}. 

\begin{thm}[Exterior wedge condition: $d=3$]\label{thm:u_wedge_3D}
Let $d=3$, $r,l>0$, $\omega \in (0,\pi/2]$, $D$ be a bounded open subset in $\mathbb{R}^3$, and $0\in \partial D$ such that there exists a wedge $\mathcal{W}(\omega,r,l)$ of radius $r$, length $l$, angle $\omega$, satisfying $\mathcal{W}(\omega,r,l)\subset \mathbb{R}^d\setminus D$.
Let $\tilde{\omega}$ be given by \eqref{eq:tilde_omega}.
The following assertions hold.
\begin{enumerate}
    \item[(i)] Let $u_g$ be the solution to \eqref{eq:PDE_g}, $0<\alpha<1$, and $C(\alpha,3)$ be the explicit constant given by \eqref{e:bmLBDG}.
   Then for every $x\in D$
    \begin{align}\label{eq:u_g_holder_x0_wedge}
        |u_g(x)-g(0)|
        &\leq  
        |g|^{0}_\alpha\left\{C_9\left(\frac{\|x\|}{r}\right)^{\alpha\tilde{\omega}}\; \left[\log{\left(\frac{r}{\|x\|}\right)}\right]^\alpha1_{[0,l\wedge(re^{-1/\tilde{\omega}}))}(\|x\|)\right.\\
        &\qquad+\left(\frac{\|x\|}{r}\right)^{\tilde{\omega}}\; \left[C_{12}\log{\left(\frac{r}{\|x\|}\right)}1_{[0,l\wedge(re^{-1/\tilde{\omega}}))}(\|x\|)\right.\\
        &\qquad
        \left.\left.
        +\frac{{\sf diam}(D)^\alpha}{3^{\alpha/2}} \left(e\vee \left(\frac{r}{l}\right)^{\tilde{\omega}}\right)1_{[l\wedge(re^{-1/\tilde{\omega}}),\infty)}(\|x\|)\right]
        +\|x\|^\alpha\right\}\\
        &\in \mathcal{O}\left(\left(\frac{\|x\|}{r}\right)^{\alpha\tilde{\omega}}\; \left[\log{\left(\frac{r}{\|x\|}\right)}\right]^\alpha\right)
    \end{align}
    where
    \begin{align}\label{eq:C_11-C_13}
        &C_9= C(\alpha,3)C_3,\quad C_{10}= C(\alpha,3)C_8,
    \end{align}
    where $C_3$ is given in \Cref{coro:D-cone-2D} whilst $C_8$ is given in \Cref{coro:D-wedge-3D}.
    
    Moreover, if there exist $\omega\in (0,\pi/2]$, $r>0$, $l>0$ such that for any $x_0\in \partial D$ there exists a wedge $\mathcal{W}(\omega,r,l)$ centered at $x_0$, with angle $\omega$, radius $r>0$, and length $l>0$, such that $\mathcal{W}(\omega,r,l)\subset \mathbb{R}^3\setminus D$, then
    for every $x,y\in D$
    \begin{align}\label{eq:u_g_holder_global}
        |u_g(x)-u_g(y)|
        &\leq  
        2|g|_\alpha\left\{C_9\left(\frac{\|x-y\|}{r}\right)^{\alpha\tilde{\omega}}\; \left[\log{\left(\frac{r}{\|x-y\|}\right)}\right]^\alpha1_{[0,l\wedge(re^{-1/\tilde{\omega}}))}(\|x-y\|)\right.\\
        &\qquad+\left(\frac{\|x-y\|}{r}\right)^{\tilde{\omega}}\; \left[C_{10}\log{\left(\frac{r}{\|x-y\|}\right)}1_{[0,l\wedge(re^{-1/\tilde{\omega}}))}(\|x-y\|)\right.\\
        &\qquad
        \left.+\frac{{\sf diam}(D)^\alpha}{3^{\alpha/2}} \left(e\vee \left(\frac{r}{l}\right)^{\tilde{\omega}}\right)1_{[l\wedge(re^{-1/\tilde{\omega}}),\infty)}(\|x-y\|)\right] 
        +\|x-y\|^\alpha\bigg\}\\
        &\in \mathcal{O}\left(\left(\frac{\|x-y\|}{r}\right)^{\alpha\tilde{\omega}}\; \left[\log{\left(\frac{r}{\|x-y\|}\right)}\right]^\alpha\right),
    \end{align}
    where $C_9-C_{10}$ are as in \eqref{eq:C_11-C_13}.
    \item[(ii)] Let $u_f$ be the solution to \eqref{eq:PDE_g} and $x\in D$.
    If $f\in L^\infty(D)$ then    \begin{align}\label{eq:u_f_infty_3d_wedge}
        |u_f(x)|\leq \|f\|_{L^\infty(D)}\left(\frac{\|x\|}{r}\right)^{\tilde{\omega}}\; &\left[C_{11}\log{\left(\frac{r}{\|x\|}\right)}
        1_{[0,l\wedge(re^{-1/\tilde{\omega}}))}(\|x\|)\right.\\
        &\left.+\frac{{\sf diam}(D)^2}{3} \left(e\vee \left(\frac{r}{l}\right)^{\tilde{\omega}}\right)1_{[l\wedge(re^{-1/\tilde{\omega}}),\infty)}(\|x\|)\right],
    \end{align}
    where
    \begin{equation}
        C_{11}=C_3+ C_8, \quad  \mbox{with }C_3, C_8 \mbox{ given in }\Cref{coro:D-wedge-3D} \mbox{ with } \alpha=2.
    \end{equation}
    Furthermore, if $p\geq 1, 1/p+1/q=1$, $0\leq f\in L^\gamma(D)$, with $\gamma/q>3/2$, then
    \begin{align}\label{eq:u_f_gamma_3d_wedge}
        |u_f(x)|\leq  C(3,D,\gamma,q)\|f\|_{L^\gamma(D)}\left(\frac{\|x\|}{r}\right)^{\tilde{\omega}/p}\; &\left[C_{11}\log{\left(\frac{r}{\|x\|}\right)}
        1_{[0,l\wedge(re^{-1/\tilde{\omega}}))}(\|x\|)\right.\\
        &\left.+\frac{{\sf diam}(D)^2}{3} \left(e\vee \left(\frac{r}{l}\right)^{\tilde{\omega}}\right)1_{[l\wedge(re^{-1/\tilde{\omega}}),\infty)}(\|x\|)\right]^{1/p},
    \end{align}
    where $C(3,D,\gamma,q)$ is given by \eqref{eq:C_f}.
\end{enumerate}
\end{thm}

\medskip
\noindent{\it Proof of \Cref{thm:u_wedge_3D} in \Cref{proof:u_wedge_3D}.}
\medskip

\begin{rem}
   Note that in both \Cref{thm:u_cone_2D} and \Cref{thm:u_wedge_3D} we can in fact allow $\omega=0$, since the explicit constants will remain bounded, whilst $\tilde{\omega}=1/2$. 
   In this case, the exterior cone  (resp. the exterior wedge) reduces to a segment  (resp. to a rectangular portion of a plane). 
   A domain that has such a degenerate exterior cone (resp. exterior wedge) can be visualized in the Figure $4$ below.
\end{rem}

\begin{center}
\includegraphics[width=12cm, height=6cm]{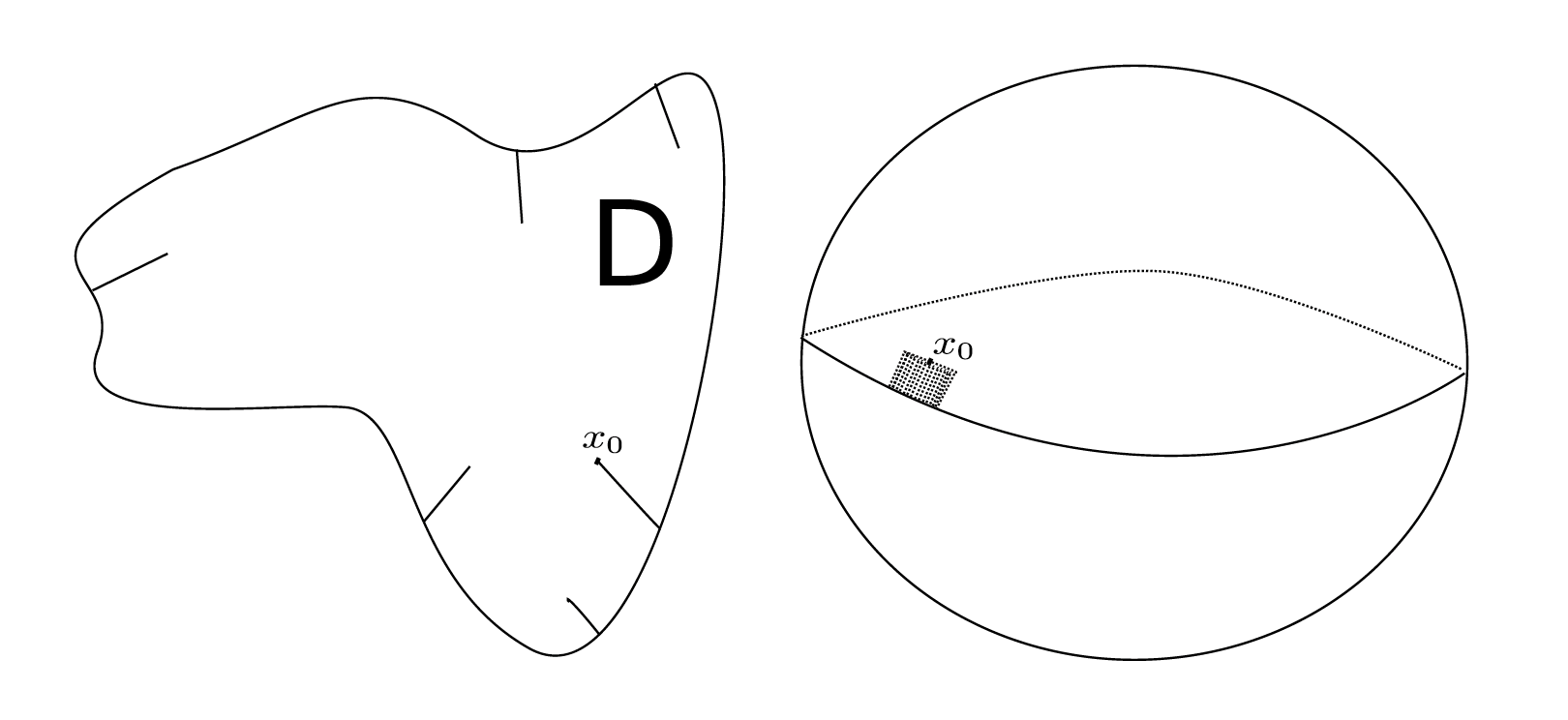}
\centerline{ Figure 4: a domain with a degenerate exterior cone}
\end{center}

\begin{thm}[Exterior cone condition: $d\geq 3$]\label{thm:u_cone_3D}
Let $D$ be a (bounded and open) domain in $\mathbb{R}^d$ and $0\in \partial D$ such that there exists a cone $\mathcal{C}(\omega,r_0)$ with vertex $0$, radius $r_0$, and angle $\omega\in (0,\pi/2]$, such that $\mathcal{C}(\omega,r_0)\subset \mathbb{R}^d\setminus D$.
Further, let $\delta_\omega\in (0,1)$ be given by \eqref{eq:delta_omega}, whilst $C(\alpha,d)$ be the constant appearing in \eqref{e:bmLBDG}.
Then, the following assertions hold
\begin{enumerate}
    \item[(i)] Let $u_g$ be the solution to \eqref{eq:PDE_g}, and $0<\alpha<1$. For every $x\in D$ we have
    \begin{align*}
        |u_g(x)-g(0)| 
        &\leq
        |g|^{0}_\alpha
        \left\{\left[C_1\|x\|^{\alpha}+C_2\|x\|^{|\log_2(\delta_\omega)|}\right]1_{[0,r_0/2)}(\|x\|)
           \right.\\
        &\quad \left.+\frac{{\sf diam}(D)^\alpha}{d^{\alpha/2}} \left(\frac{2}{r_0}\right)^\alpha\|x\|^{\alpha}1_{[r_0/2,\infty)}(\|x\|)\right\}\\
        &\in \mathcal{O}\left(\frac{{\sf diam}(D)^\alpha}{d^{\alpha/2}}(1\vee {\sf diam}(D))^\alpha\|x\|^{\left[\left(\frac{2}{3}\sin(\omega)\right)^{d-1}d^{-1/2}\right]\wedge \alpha}\right) \quad \mbox{w.r.t } d, {\sf diam}(D), \mbox{and } \|x\|,
    \end{align*}
    where $C_1$ and $C_2$ are given by \eqref{eq:C_1C_2} with $\delta$ replaced by $\delta_\omega$.
        
        Moreover, if there exist $\omega\in (0, \pi/2]$ and $r>0$ such that for any $x_0\in \partial D$ there exists a cone $\mathcal{C}(\omega,r)$ of vertex $x_0$, angle $\omega$, and radius $r$, such that $\mathcal{C}(\omega,r)\subset \mathbb{R}^d\setminus D$, then for every $x,y\in D$
        \begin{align}
            |u_g(x)-u_g(y)|
            &\leq  2|g|_\alpha\left\{\left[C_1\|x-y\|^{\alpha}+C_2\|x-y\|^{|\log_2(\delta_\omega)|}\right]1_{[0,r_0/2)}(\|x-y\|)
           \right.\\
           &\left.+\frac{{\sf diam}(D)^\alpha}{d^{\alpha/2}} \left(\frac{2}{r_0}\right)^\alpha\|x-y\|^{\alpha}1_{[r_0/2,\infty)}(\|x-y\|)\right\}\\
            &\in \mathcal{O}\left(\frac{{\sf diam}(D)^\alpha}{d^{\alpha/2}}(1\vee {\sf diam}(D))^\alpha\|x-y\|^{\left[\left(\frac{2}{3}\sin(\omega)\right)^{d-1}d^{-1/2}\right]\wedge \alpha}\right) \quad \mbox{w.r.t } d, {\sf diam}(D), \mbox{and } \|x\|,
        \end{align}
        where $C_1$ and $C_2$ are given by \eqref{eq:C_1C_2} with $\delta$ replaced by $\delta_\omega$.
    \item[(ii)] Let now $u_f$ be the solution to \eqref{eq:PDE_f}, and $x\in D$.
    If $f\in L^{\infty}(D)$ then 
    \begin{equation}
        |u_f(x)|
        \leq \|f\|_{L^\infty(D)} \theta(\|x\|),
    \end{equation}
    where $\theta(\|x\|)$ is given by \eqref{eq:theta} with $\delta$ replaced by $\delta_\omega$, and has the asymptotics
    \begin{equation*}
        \theta(\|x\|)\in \mathcal{O}\left(\frac{{\sf diam}(D)^4}{d^2}\|x\|^{\left[\left(\frac{2}{3}\sin(\omega)\right)^{d-1}d^{-1/2}\right]\wedge 2}\right) \quad \mbox{with respect to } d, {\sf diam}(D), \mbox{ and } \|x\|.
    \end{equation*}

    Furthermore, for every $p>1, 1/p+1/q=1$ and $f\in L^\gamma(D)$ with $ \gamma>dq/2$, we have
    \begin{equation}
        |u_f(x)|\leq C(d,D,\gamma,q)\|f\|_{L^\gamma(D)}\left[\theta(\|x\|)\right]^{1/p},
    \end{equation}
    where recall that $C(d,D,\gamma,q)$ is given by \eqref{eq:C_f} and has the asymptotics (cf. \eqref{eq:C in O})
    \begin{equation*}
        C(d,D,\gamma,q) \in \mathcal{O}\left( \left(\frac{d^2(\gamma-q)}{2\gamma-qd}\right)^{1/q}{\sf diam}(D)^{\frac{2\gamma-qd}{pq}}\right), \quad \mbox{w.r.t. } d,\gamma,q,p,{\sf diam}(D).
    \end{equation*}
\end{enumerate}
\end{thm}

\medskip
\noindent{\it Proof of \Cref{thm:u_cone_3D} in \Cref{proof:u_cone_3D}.}
\medskip

\begin{rem}
    In contrast to the case of domains in $\mathbb{R}^2$ satisfying an exterior cone condition, or to domains in $\mathbb{R}^3$ satisfying an exterior wedge condition (as in \Cref{thm:u_cone_2D} and \Cref{thm:u_wedge_3D}), if we take $\omega=0$ in \Cref{thm:u_cone_3D} then the H\"older regularity of $u_g$ is no longer ensured. 
\end{rem}

 \begin{rem}
    Keeping $\omega$ fixed, note that the H\"older exponent $\left(\frac{2}{3}\sin(\omega)\right)^{d-1}d^{-1/2}$ appearing in \Cref{thm:u_cone_3D} is decreasing exponentially fast with respect to $d$; although \Cref{thm:u_cone_3D} provides only an upper bound, we expect that the exponential decay of the H\"older exponent with respect to $d$ can not be improved.
\end{rem}   

\begin{rem}
    It is worth mentioning that the exterior wedge condition employed in \Cref{thm:u_wedge_3D} can be extended in a straightforward manner to dimensions higher than $3$, so that one can construct families of domains in $\mathbb{R}^d$ indexed by $d$, that are not Lipschitz domains (yet satisfy an exterior cone condition), and for which the harmonic extensions $u_g$ of some (family of) $\alpha$-H\"older continuous boundary data $g$, with $\alpha$ independent of $d$, remain H\"older continuous with a H\"older exponent which is also independent of $d$.
\end{rem}

\subsection{Further consequences of the explicit estimates.}\label{subsec:furtherconseq}
The explicit estimates derived in this paper for $u_g$, $u_f$, as well as for the key involved functions $V_{D,\alpha}, h_{D,\Gamma_0}$ defined in \eqref{e:v} and \eqref{def:hDGamma0}, can be directly used to obtain explicit upper (boundary) estimates for the Green function of a domain $D$ satisfying an exterior sphere, cone, or wedge condition, as well as for the gradients of harmonic functions, as follows.  

\subsubsection{Explicit gradient estimates for harmonic functions.}
Let $d\geq 2$ and $D\subset \mathbb{R}^d$ be a bounded open set. 
Further, let $g:\partial D\rightarrow \mathbb{R}$ be $\alpha$-H\"older continuous for some $\alpha\in (0,1)$, and let $u_g$ be given by \eqref{eq:prob_sol_split}.
Since $u_g$ is harmonic in $D$, by the mean value property we have
\begin{equation*}
    u_g(x)
    =\frac{1}{\lambda(B(0,\eta))}\int_{B(0,\eta)}u(x+y) dy, \quad x\in D, \eta<{\sf dist}(x,\partial D).
\end{equation*}
For $x\in D$ and fixing $\eta<{\sf dist}(x,\partial D)$, we get by Green's formula
\begin{align*}
    \nabla u(x) 
    &= \frac{1}{\lambda(B(0,\eta))}\int_{B(0,\eta)}\nabla u(x+y) dy
    = \frac{1}{\lambda(B(0,\eta))}\int_{\partial B(0,\eta)}u(x+y) n(y) dy\\
    &=\frac{1}{\lambda(B(0,\eta))}\int_{\partial B(0,\eta)}\left[u(x+y)-u(x)\right] n(y) dy,
\end{align*}
where $n(y)$ denotes the outward normal unit vector at $y\in \partial B(0,\eta)$.
Thus,
\begin{equation}\label{eq:general_gradient}
    \|\nabla u(x)\|
    \leq \frac{d}{\eta}\sup_{y\in \partial B(x,\eta)}|u(y)-u(x)|, \quad \eta\leq{\sf dist}(x,\partial D).
\end{equation}
As a consequence of the above estimate \eqref{eq:general_gradient} and the previously presented H\"older estimates for $u_g$, we get the next result.
\begin{prop}\label{coro:gradient_u_g}
    Let $d\geq 2$ and $D\subset \mathbb{R}^d$ be a bounded open subset. 
    Further, let $g:\partial D\rightarrow \mathbb{R}$ be $\alpha$-H\"older continuous for some $\alpha\in (0,1)$, and $u_g$ be given by \eqref{eq:prob_sol_split}.
    Then the following assertions hold:
    \begin{enumerate}
        \item[(i)] Assume that there exist $r,R$ such that for every $x_0\in \partial D$ there exists an annulus $A(a,r,R)$ with $x_0\in \partial B(a,r)$ and $D\subset A(a,r,R)$.
        Further, let $C(\alpha,d)$ be the explicit constant given by \eqref{e:bmLBDG}.
        Then
        for every $x\in D$ we have
        \begin{equation}
            \|\nabla u_g(x)\|
            \leq
            \left({\sf dist}(x,\partial D)\wedge\frac{r}{d}\right)^{\alpha-1}|g|_\alpha 2d
            \left\{C(\alpha,d) e^{\alpha}\left[\frac{4}{(1-\alpha)c_\alpha} +\frac{(R-r)^{\alpha/2}}{r^{\alpha/2}}d^{\alpha/2}\right]+1\right\}
        \end{equation}
        \item[(ii)] Let $d=2$ and assume that there exist $\omega\in (0, \pi/2]$ and $r>0$ such that for any $x_0\in \partial D$ there exists a cone $\mathcal{C}(\omega,r)$ of vertex $x_0$, angle $\omega$, and radius $r$, such that $\mathcal{C}(\omega,r)\subset \mathbb{R}^2\setminus D$.
        Then, letting $\tilde{\omega}$ be given by \eqref{eq:tilde_omega}, we have for every $x\in D$
        \begin{align}
            \|\nabla u_g(x)\|
            &\leq
            2d|g|_\alpha \left\{C_5\left({\sf dist}(x,\partial D)\wedge(re^{-1/\tilde{\omega}})\right)^{\alpha\tilde{\omega}-1}r^{-\alpha\tilde{\omega}}\; \left|\log{\left(\frac{{\sf dist}(x,\partial D)\wedge(re^{-1/\tilde{\omega}})}{r}\right)}\right|^\alpha\right.\\
            &\quad\;+C_6\left({\sf dist}(x,\partial D)\wedge(re^{-1/\tilde{\omega}})\right)^{\tilde{\omega}-1}r^{-\tilde{\omega}}\; \left|\log{\left(\frac{{\sf dist}(x,\partial D)\wedge(re^{-1/\tilde{\omega}})}{r}\right)}\right|\\
            &\quad\;\left.+\left({\sf dist}(x,\partial D)\wedge(re^{-1/\tilde{\omega}})\right)^{1-\alpha}\right.\bigg\},
        \end{align}
        where $C_5-C_6$ are as in \eqref{eq:C_1-C_3}.
        \item[(iii)] Let $d=3$ and assume there exist $\omega\in (0,\pi/2]$, $r>0$, $l>0$ such that for any $x_0\in \partial D$ there exists a wedge $\mathcal{W}(\omega,r,l)$ centered at $x_0$, with angle $\omega$, radius $r>0$, and length $l>0$, such that $\mathcal{W}(\omega,r,l)\subset \mathbb{R}^3\setminus D$. 
        Then for every $x\in D$ we have
        \begin{align}
            \|\nabla u_g(x)\|
            &\leq
            2d|g|_\alpha \left\{C_9\left({\sf dist}(x,\partial D)\wedge l\wedge(re^{-1/\tilde{\omega}})\right)^{\alpha\tilde{\omega}-1}r^{-\alpha\tilde{\omega}}\; \left|\log{\left(\frac{{\sf dist}(x,\partial D)\wedge l\wedge(re^{-1/\tilde{\omega}})}{r}\right)}\right|^\alpha\right.\\
            &\quad\;+C_{10}\left({\sf dist}(x,\partial D)\wedge l\wedge(re^{-1/\tilde{\omega}})\right)^{\tilde{\omega}-1}r^{-\tilde{\omega}}\; \left|\log{\left(\frac{{\sf dist}(x,\partial D)\wedge l\wedge(re^{-1/\tilde{\omega}})}{r}\right)}\right|\\
            &\quad\;\left.+\left({\sf dist}(x,\partial D)\wedge l\wedge(re^{-1/\tilde{\omega}})\right)^{\alpha-1}\right.\bigg\},
        \end{align}
        where $C_9-C_{10}$ are as in \eqref{eq:C_11-C_13}.
        \item[(iv)] Let $d\geq 3$ and assume there exist $\omega\in (0, \pi/2]$ and $r>0$ such that for any $x_0\in \partial D$ there exists a cone $\mathcal{C}(\omega,r)$ of vertex $x_0$, angle $\omega$, and radius $r$, such that $\mathcal{C}(\omega,r)\subset \mathbb{R}^d\setminus D$.
        Then for every $x\in D$ we have
        \begin{equation}
            \|\nabla u_g(x)\|\leq 2d|g|_\alpha\left[C_1\left({\sf dist}(x,\partial D)\wedge (r_0/2)\right)^{\alpha-1}+C_2\left({\sf dist}(x,\partial D)\wedge (r_0/2)\right)^{|\log_2(\delta_\omega)|-1}\right],
        \end{equation}
        where $C_1$ and $C_2$ are given by \eqref{eq:C_1C_2} with $\delta$ replaced by $\delta_\omega$.
    \end{enumerate}
\end{prop}
\begin{proof}
    Assertion (i) follows by \Cref{thm:u_g_ball} and \eqref{eq:general_gradient}, taking $\eta={\sf dist}(x,\partial D)\wedge\frac{r}{d}$.
    Assertion (ii) follows by \Cref{thm:u_cone_2D} and \eqref{eq:general_gradient}, taking $\eta={\sf dist}(x,\partial D)\wedge(re^{-1/\tilde{\omega}})$.
    Assertion (iii) follows by \Cref{thm:u_wedge_3D} and \eqref{eq:general_gradient}, taking $\eta={\sf dist}(x,\partial D)\wedge l\wedge(re^{-1/\tilde{\omega}})$.
    Assertion (iv) follows by \Cref{thm:u_cone_3D} and \eqref{eq:general_gradient}, taking $\eta={\sf dist}(x,\partial D)\wedge(r_0/2)$.
\end{proof}
\begin{rem}\label{rem:simplified_gradients}
    Note that in \Cref{coro:gradient_u_g}, for $x$ in a sufficiently small (and explicit) neighborhood of $\partial D$, the estimates in (i)-(iv) can be seen in terms of $d$ and ${\sf dist}(x,\partial D)$ as:
    \begin{enumerate}
        \item[(i)] $\|\nabla u_g(x)\|\in \mathcal{O} \left(d^{1+\alpha}{\sf dist}(x,\partial D)^{\alpha-1}\right)$,
        \item[(ii)] $\|\nabla u_g(x)\|\in \mathcal{O} \left( {\sf dist}(x,\partial D)^{\alpha\tilde{\omega}-1}\left|\log{({\sf dist}(x,\partial D))}\right|^\alpha\right)$,
        \item[(iii)] $\|\nabla u_g(x)\|\in \mathcal{O} \left( {\sf dist}(x,\partial D)^{\alpha\tilde{\omega}-1}\left|\log{({\sf dist}(x,\partial D))}\right|^\alpha\right)$,
        \item[(iv)] 
        $\|\nabla u_g(x)\|\in \mathcal{O} \left( \frac{{\sf diam}(D)^\alpha}{d^{\alpha/2}}(1\vee {\sf diam}(D))^\alpha {\sf dist}(x,\partial D)^{\alpha\wedge \log_2{\left(\delta_{\omega}\right)}-1}\right)$
        
        $\phantom{\|\nabla u_g(x)\|}\in \mathcal{O} \left( \frac{{\sf diam}(D)^\alpha}{d^{\alpha/2}}(1\vee {\sf diam}(D))^\alpha {\sf dist}(x,\partial D)^{\alpha\wedge\left[\left(\frac{2}{3}\sin(\omega)\right)^{d-1}d^{-1/2}\right]-1}\right)$.
    \end{enumerate}
    Using the above estimates, it is straightforward to get $L^p(D)$ bounds for $\|\nabla u_g\|$, for exponents $p\geq 1$ that can be made explicit in terms of $\alpha$ and the parameters involved in the exterior sphere/cone/wedge conditions.
\end{rem}

\begin{rem}
     Gradient estimates similar to the ones obtained in \Cref{coro:gradient_u_g} can be deduced for $u_f$ (given by \eqref{eq:prob_sol_split}) as well. 
     The main point is to replace \eqref{eq:general_gradient} with a similar general estimate that is suitable for $u_f$; we leave this for an interested reader.
\end{rem}

\subsubsection{Explicit boundary estimates for the Green function.}\label{sec:Greenfunction}

Let $d\geq 3$, $D\subset \mathbb{R}^d$ be open and bounded, whose boundary points are all regular in the sense mentioned in \Cref{rem:PDEcorrespondence}; this is the case if $D$ satisfies the exterior cone condition at every boundary point.
We denote by $G_D(x,y),x,y\in D$ the Green function associated to the (Dirichlet) Laplacian on $D$.
In particular, $G_D(x,\cdot)$ is continuous on $\overline{D}\setminus \{x\}$ and vanishes on $\partial D$ for every $x\in D$.
Moreover, recall that
\begin{equation}
    G_D(x,y)\leq G(x,y), \quad x,y\in D,
\end{equation}
where $G$ denotes the Newtonian kernel
\begin{equation*}
    G(x,y)=
        \frac{\gamma_d}{\|x-y\|^{d-2}},
    \quad x,y\in \mathbb{R}^d, \quad \gamma_d:=\frac{\Gamma(d/2-1)}{(4\pi)^{d/2}}.
\end{equation*}

Here, we are interested in explicit estimates of $G_D(x,y)$ in terms of $\|x-y\|$ and ${\sf dist}(y,\partial D)$.
To this end, let us begin with some preliminary computations: If $x\in D, y_0\in \partial D$, and $0<l<\|x-y_0\|$, then the function $G_D(x,\cdot)$ is harmonic on $D_{x,l}:=D\setminus \overline{B(x,l)}$. 
Moreover,
\begin{align}
    \sup_{\partial D_{x,l}\ni z\neq y_0\in \partial D}\frac{G_D(x,z)-G_D(x,y_0)}{\|z-y_0\|^\alpha}
    &=\sup_{z\in \partial B(x,l)\cap D}\frac{G_D(x,z)}{\|z-y_0\|^\alpha}
    \leq \sup_{z\in \partial B(x,l)}\frac{G(x,z)}{\|z-y_0\|^\alpha}\\
    &= \sup_{z\in \partial B(x,l)}\frac{\gamma_d l^{2-d}}{\|z-y_0\|^\alpha}
    \leq\frac{\gamma_d l^{2-d}}{\left(\|x-y_0\|-l\right)^\alpha}.
\end{align}
Minimizing the last term over $l\in (0,{\sf dist}(x,\partial D))$ we obtain the optimal choice $l^\ast=\frac{d-2}{d-2+\alpha}\|x-y_0\|$.
Inspired from this optimal choice, let now $x,y\in D$ and $y_0\in \partial D$ such that
\begin{equation}\label{eq:x,y,y_0}
    {\sf dist}(y,\partial D)=\|y-y_0\|<\frac{\alpha}{2\left(d-2+\alpha\right)}\|x-y\|,
\end{equation}
and set $l_y:=\frac{d-2}{d-2+\alpha}\|x-y\|$.
Then $G_D(x,\cdot)$ is harmonic on $D_{x,l_y}$ and $y\in D_{x,l_y}$.
Moreover, 
\begin{equation}
    \|x-y_0\|\geq \|x-y\|-\|y-y_0\|>\frac{2(d-2)+\alpha}{2\left(d-2+\alpha\right)}\|x-y\|=\left(1+\frac{\alpha}{2(d-2)}\right)l_y,
\end{equation} 
hence $y_0\in \partial D_{x,l_y}$.
Furthermore, regarding $G_D(x,\cdot)$ as a function on $\overline{D}_{x,l_y}$,
\begin{align}
    |G_D(x,\cdot)|_\alpha^{y_0}
    &:=\sup_{\partial D_{x,l}\ni z\neq y_0\in \partial D}\frac{G_D(x,z)-G_D(x,y_0)}{\|z-y_0\|^\alpha}
    \leq \frac{\gamma_d l_y^{2-d}}{\left(\|x-y_0\|-l_y\right)^\alpha}\nonumber\\
    &\leq \gamma_d l_y^{2-d-\alpha}\left(\frac{2(d-2)}{\alpha}\right)^\alpha
    =\gamma_d\left(1-\frac{\alpha}{d-2+\alpha}\right)^{2-d-\alpha}\left(\frac{2(d-2)}{\alpha}\right)^\alpha\|x-y\|^{2-d-\alpha}\nonumber\\
    &=:\frac{\gamma_{d,\alpha}}{\|x-y\|^{d-2+\alpha}}, \quad x\in D, \alpha \in (0,1).\label{eq:G_D_alpha}
\end{align}
Note that $\gamma_{d,\alpha}\sim d^{\alpha}\gamma_d$ as $d\to \infty$.

\begin{rem}\label{rem:y_closerto_x}
Note that when the reverse of inequality \eqref{eq:x,y,y_0} is satisfied, namely $\|x-y\|\leq \frac{2(d-2+\alpha)}{\alpha}{\sf dist}(y,\partial D)$, then we trivially have
\begin{equation}
    G_D(x,y)\leq G(x,y)\leq {\left(\frac{2(d-2+\alpha)}{\alpha}\right)}^\alpha\gamma_d \frac{{\sf dist}(y,\partial D)^\alpha}{\|x-y\|^{d-2+\alpha}}, \quad \alpha\in [0,1].
\end{equation}
The goal is to show that a similar bound holds also when \eqref{eq:x,y,y_0} is satisfied, for some $\alpha>0$ explicitly depending on the geometry of $D$ and the dimension $d$. 
\end{rem}

We have the following result:
\begin{prop}\label{coro:G_D_alpha}
    Let $d\geq 3$ and $D\subset \mathbb{R}^d$ be a bounded open subset, and $\alpha\in (0,1)$.
    Then the following assertions hold:
    \begin{enumerate}
        \item[(i)] Assume there exist $r,R$ such that for every $x_0\in \partial D$ there exists an annulus $A(a,r,R)$ with $x_0\in \partial B(a,r)$ and $D\subset A(a,r,R)$.
        Further, let $C(\alpha,d)$ be the explicit constant given by \eqref{e:bmLBDG}, and $c_\alpha$ be the universal constant from \Cref{t:BDG}.
        Then
        for every $x,y\in D$ such that \eqref{eq:x,y,y_0} holds, we have
        \begin{align*}
            G_D(x,y)
            \leq
            {\sf dist}(y,\partial D)^{\alpha}\frac{\gamma_{d,\alpha}}{\|x-y\|^{d-2+\alpha}}
            &\left[C(\alpha,d,R,r)1_{[0,r/d]}({\sf dist}(y,\partial D))\right.\\
            &\left.+{\sf diam}(D)^\alpha \frac{d^{\alpha/2}}{r^\alpha}1_{(r/d,\infty)}({\sf dist}(y,\partial D))\right],
        \end{align*}
        where $C(\alpha,d,R,r)$ is given by \eqref{eq:C(alpha,d,R,r)}, whilst $\gamma_{d,\alpha}$ is given by \eqref{eq:G_D_alpha}.
        \item[(ii)] Let $d=3$ and assume there exist $\omega\in (0,\pi/2]$, $r>0$, $l>0$ such that for any $x_0\in \partial D$ there exists a wedge $\mathcal{W}(\omega,r,l)$ centered at $x_0$, with angle $\omega$, radius $r>0$, and length $l>0$, such that $\mathcal{W}(\omega,r,l)\subset \mathbb{R}^3\setminus D$. 
        Then for every $x,y\in D$ such that \eqref{eq:x,y,y_0} holds, we have
        \begin{align}
            G_D(x,y)
            &\leq
            \frac{\gamma_{d,\alpha}}{\|x-y\|^{1+\alpha}} 
            \left\{C_9\left(\frac{{\sf dist}(y,\partial D)}{r}\right)^{\alpha\tilde{\omega}}\; \left[\log{\left(\frac{r}{{\sf dist}(y,\partial D)}\right)}\right]^\alpha1_{[0,l\wedge(re^{-1/\tilde{\omega}}))}({\sf dist}(y,\partial D))\right.\\
            &\qquad+\left(\frac{{\sf dist}(y,\partial D)}{r}\right)^{\tilde{\omega}}\; \left[C_{10}\log{\left(\frac{r}{{\sf dist}(y,\partial D)}\right)}1_{[0,l\wedge(re^{-1/\tilde{\omega}}))}({\sf dist}(y,\partial D))\right.\\
            &\qquad
            \left.\left.
            +\frac{{\sf diam}(D)^\alpha}{3^{\alpha/2}} \left(e\vee \left(\frac{r}{l}\right)^{\tilde{\omega}}\right)1_{[l\wedge(re^{-1/\tilde{\omega}}),\infty)}({\sf dist}(y,\partial D))\right]
            +{\sf dist}(y,\partial D)^\alpha\right\},
        \end{align}
        where $C_9-C_{10}$ are as in \eqref{eq:C_11-C_13}.
        \item[(iii)] Assume there exist $\omega\in (0, \pi/2]$ and $r>0$ such that for any $x_0\in \partial D$ there exists a cone $\mathcal{C}(\omega,r)$ of vertex $x_0$, angle $\omega$, and radius $r$, such that $\mathcal{C}(\omega,r)\subset \mathbb{R}^d\setminus D$.
        Then for every $x,y\in D$ for which\eqref{eq:x,y,y_0} holds, we have
        \begin{align*}
            G_D(x,y)
            \leq 
            \frac{\gamma_{d,\alpha}}{\|x-y\|^{d-2+\alpha}}
            &\left\{\left[C_1{\sf dist}(y,\partial D)^{\alpha}+C_2{\sf dist}(y,\partial D)^{|\log_2(\delta_\omega)|}\right]1_{[0,r_0/2)}({\sf dist}(y,\partial D))
           \right.\\
            &\left.+\frac{{\sf diam}(D)^\alpha}{d^{\alpha/2}} \left(\frac{2}{r_0}\right)^\alpha{\sf dist}(y,\partial D)^{\alpha}1_{[r_0/2,\infty)}({\sf dist}(y,\partial D))\right\},
        \end{align*}
        where $C_1$ and $C_2$ are given by \eqref{eq:C_1C_2} with $\delta$ replaced by $\delta_\omega$.
    \end{enumerate}
\end{prop}
\begin{proof}
    Assertion (i) follows by \Cref{thm:u_g_ball} and \eqref{eq:G_D_alpha}.
    Assertion (ii) follows by \Cref{thm:u_wedge_3D} and \eqref{eq:G_D_alpha}. Assertion (iii) follows by \Cref{thm:u_wedge_3D} and \eqref{eq:G_D_alpha}.
    Assertion (iv) follows by \Cref{thm:u_cone_3D} and \eqref{eq:G_D_alpha}.
\end{proof}

\begin{rem}
    Similar simplified bounds as in \Cref{rem:simplified_gradients} can be written down for \Cref{coro:G_D_alpha} as well; we leave this to the interested reader.
\end{rem}

\noindent{\bf Final remark.}
Let $d\geq 2$ and $D\subset \mathbb{R}^d$ be a bounded Lipschitz domain with localization radius $r_0$ and Lipschitz constant $L>0$.
On brief, this means that for every point $x_0\in \partial D$ the boundary portion $\partial D\cap B(x_0,r_0)$ is the graph of an $L$-Lipschitz function; for the precise definition we refer to \cite{bogdan2000sharp}. 
In particular, by e.g. \cite{bogdan1997boundary} we have: 
At every point $x_0\in \partial D$  there exists an exterior cone $\mathcal{C}(\omega_L,r_0)$  with $\omega_L=\pi/2-\arctan{L}$.
Consequently, for such domains, the conclusions in \Cref{thm:u_cone_2D}, \Cref{thm:u_cone_3D}, \Cref{coro:G_D_alpha} (ii) and (iv), \Cref{coro:gradient_u_g} (ii) and (iv), are valid for $\omega=\omega_L$, $r=r_0$.

\section{Boundary estimates for the Poisson problem in terms of Brownian exit times}\label{S:general estimates}
Let $D$ be a bounded open set in $\mathbb{R}^d$, and $u_g, u_f$ be given by \eqref{eq:prob_sol_split}, with $f,g$ satisfying $\mathbf{H_f}$ and $\mathbf{H_g}$ from the beginning of the first section.
The main aim of the current section is to show that with no extra regularity assumptions on the domain $D$, the (boundary) H\"older estimates for $u_g$ and $u_f$ can be explicitly reduced to estimates for the Brownian exit time from $D$.
To this end, we first discuss some basic insights on exit times that are useful for our purpose, and then, in the second subsection, we derive the aforementioned general estimates for $u_g$ and $u_f$, separately.

\subsection{Some general aspects of exit times in bounded domains} \label{subsec:exittimegeneral}
Recalling that $v_{D,\alpha}$ is given by \eqref{eq:v_Dalpha}, we begin with the following simple uniform bound.
\begin{lem}[Uniform bound]\label{lem:v_alpha_uniform}
Let $D\subset \mathbb{R}^d$ be open and bounded. 
Then for every $x\in \overline{D}$ and $\alpha>0$ we have
\begin{equation}\label{eq:v_alpha_uniform}
    v_{D,\alpha}(x)\leq C_1(\alpha,d) {\sf diam}(D)^{2\alpha}, \quad \mbox{ where } 
    C_1(\alpha, d)
    =
    \begin{cases}
      \frac{1}{d^\alpha}, \quad &\alpha\leq 1\\
      \frac{1}{c_\alpha}\left(\frac{\alpha\vee 1}{|\alpha-1|}\right)^{\alpha \vee 1} d^{\frac{\alpha}{2}\left(\frac{2}{\alpha}-1\right)^+-1}, \quad &\alpha> 1
    \end{cases}. 
\end{equation}    
\end{lem}
\begin{proof}
    Since $\left(\|B(t)\|^2-dt\right)_{t\geq 0}$ is a martingale, by Doob's stopping theorem we get
    \begin{equation*}
        v_{D}(x)=\mathbb{E}\left[\tau^x_{\partial D}\right]=\mathbb{E}\left[\|B(t)\|^2\right]/d\leq \frac{{\sf diam}(D)^2}{d}, \quad x\in \overline{D}.
    \end{equation*}
    Thus, by Jensen's inequality we get that
    \begin{equation*}
        v_{D,\alpha}(x)\leq\frac{{\sf diam}(D)^{2\alpha}}{d^\alpha}, \quad 0<\alpha\leq 1, x\in \overline{D}.
    \end{equation*}
    Furthermore, by \Cref{lem:tau-B} we have
    \begin{equation*}
        v_{D,\alpha}(x)
        \leq
        \frac{1}{c_\alpha}\left(\frac{\alpha\vee 1}{|\alpha-1|}\right)^{\alpha \vee 1} d^{\frac{\alpha}{2}\left(\frac{2}{\alpha}-1\right)^+-1}{\sf diam}(D)^{2\alpha}, \quad  0<\alpha\neq 1.
    \end{equation*}
\end{proof}
As an obvious consequence of \Cref{lem:v_alpha_uniform}, $v_{D,\alpha}$ is a bounded function on $\overline{D}$ for every $\alpha>0$.
Furthermore, we define 
\begin{equation}\label{e:v_aa}
v_{\infty,\alpha}(s):=\sup\{v_{D,\alpha}(x):d(x,\partial D)\leq s\}, \quad s\geq 0.
\end{equation}

Besides $v_{D,\alpha}$, another special function is going to be crucially employed in this work. 
To define it, we need to fix $\Gamma_0\subset \partial D$ a (measurable) portion of the boundary, and set $\Gamma_1:=\partial D\setminus \Gamma_0$ for the remaining part of the boundary.
Further, define
\begin{equation}\label{def:hDGamma0}
    h_{D,\Gamma_0}(x):=\mathbb{P}\left(\tau^x_{\Gamma_1}<\tau^x_{\Gamma_0}\right), \quad x\in \overline{D}.
\end{equation}
Note that $h_{\Gamma_0}$ is the {\it probabilistic solution} (see \Cref{defi:prob_solution}) to the problem
\begin{equation}\label{e:h}
\Delta h_{D,\Gamma_0}=0 \textrm{ in } D, \quad h_{D,\Gamma_0}=0 \textrm{ on }\Gamma_0 \quad \mbox{ and } \quad h_{D,\Gamma_0}=1 \textrm{ on }\Gamma_1.
\end{equation}  

One method to obtain quantitative boundary estimates for $u_g$ and $u_f$, in a given domain $D$, is to reduce the geometry to some very concrete shapes (like sphere or cones) that stand as {\it barriers} at the boundary of $D$. 
The barrier method is of course a folklore approach in PDEs and potential theory, which relies on analytic comparison principles. 
In this work, the barrier method is going to be employed from a slightly different perspective, the main tools being the comparison of hitting times and maximal inequalities; adopting this strategy, we shall be able to develop a fine analysis for the desired estimates, with a sharp control on the constants. 
Formally, the barrier principle in our case goes as follows:

\begin{lem}[Comparison principle]\label{lem:barrier}

Let $C\subset \mathbb{R}^d$ be measurable such that $D\subset C$,
then
\begin{equation}
    v_{D,\alpha}(x)\leq v_{C,\alpha}(x), \quad x\in D, \alpha \geq 0 .
\end{equation}
\end{lem}
\begin{proof}
    It is obvious since $\tau^x_{\partial D}\leq \tau^x_{\partial C}$, $x\in D$.
\end{proof}

As a first consequence of this little result we obtain the following general lower bound.  

\begin{coro}\label{c:lowb} For any bounded and open set $D\subset \mathbb{R}^d$, $d\geq 2$, and $\alpha>0$,  
\begin{equation}\label{e:lowb}
\frac{1}{d^{(\alpha-1)^+}C_{2\alpha}}d(x,\partial D)^{2\alpha}\le v_{D,\alpha}(x)
\end{equation}
where $C_{2\alpha}$ is the constant from \eqref{e:BDG}.  
\end{coro}

\begin{proof}
For any point $x\in D$, take the ball $C=B(x,{\sf dist}(x,\partial D))\subset D$ and use the above estimate to get that 
\[
v_{C,\alpha}(x)\le v_{D,\alpha}(x).
\]
Now, for the ball, we have from \eqref{e:BDG} from the Appendix applied to each component of the Brownian motion 
\begin{equation}\label{e:tauinf}
\frac{1}{C_{2\alpha}}\Ex\left[\|B^{(k)}_{\tau^x_{C}}\|^{2\alpha}\right]\le \Ex[\tau_{C}^{\alpha}]
\end{equation}
valid for each $k\in \{1,2,\dots,d \}$ because the quadratic variation of $B^k_t$ is $t$ and we use only the second inequality from \eqref{e:BDG} where we use that $|B^{(k)}_{\tau^x_{C}}|\le (B^{(k)})^\ast_{\tau^x_{C}}$.  
Now, it is not hard to check that for any $\alpha>0$, we have for any choice of numbers $x_1,x_2,\dots, x_d\geq 0$, that 
\[
d^{-(\alpha-1)^+}(x_1^2+\cdots+x_d^2)^{\alpha}\le \sum_{i=1}^d x_i^{2\alpha} 
\]
valid for all $x_i\ge 0$ and all $\alpha>0$.
Using this and summing up the inequalities \eqref{e:tauinf} combined with  $\|B_{\tau^x_{C}}\|=d(x,\partial D)$ yields the conclusion.  
\end{proof}

\begin{rem}
The lower bound in \Cref{c:lowb} reveals that the best possible behavior of $v_{D,\alpha}(x)$ as $x$ approaches the boundary is $d(x,\partial D)^{2\alpha}$.  
For $\alpha=1$, it means that $v_D(x)$ can not decay to zero faster than $d(x,\partial D)^2$, as $x$ approaches the boundary.
As a matter of fact, we will show in \Cref{ss:second_approach} that the problem of estimating $v_{D,\alpha}$ can be explicitly reduced to estimating $v_D$.    
\end{rem}

Another immediate yet useful remark is the following; we will use it in \Cref{ss:wedge} for domains that satisfy an {\it exterior wedge condition}.
\begin{lem}\label{lem:products}
    Assume that $C_k\in \mathbb{R}^{d_k}$, $1\leq k\leq 2$ are bounded and open, and set $C:=C_1\times C_2 \subset \mathbb{R}^d$, where $d=d_1+d_2$.
    If $x_{(k)}$ denotes the projection of $x\in C$ onto $C_k$, $1\leq k\leq 2$, then
    \begin{equation}\label{eq:products}
        v_{C,\alpha}(x)\leq v_{C_1,\alpha}\left(x_{(1)}\right)\wedge v_{C_2,\alpha}\left(x_{(2)}\right), \quad \left(x_{(1)},x_{(2)}\right)=x\in C,\quad \alpha\geq 0.
    \end{equation}
    Moreover, if $\Gamma_0^{(1)}\subset \partial C_1$ and $\Gamma_0:=\Gamma_0^{(1)}\times C_2\subset\partial C, \;\Gamma_1=\partial C\setminus \Gamma_0$, then
    \begin{equation}\label{eq:products_h}
         h_{C,\Gamma_0}(x)\leq h_{C_1,\Gamma_0^{(1)}}\left(x_{(1)}\right), \quad \left(x_{(1)},x_{(2)}\right)=x\in C,\quad \alpha\geq 0.
    \end{equation}
\end{lem}
\begin{proof}
It is a simple observation that 
\begin{equation*}
    \tau^x_{\partial C}=\tau^{x_{(1)}}_{\partial C_1}\wedge \tau^{x_{(2)}}_{\partial C_22}, 
\end{equation*}
where $\tau^{x_{(k)}}_{\partial C_k}$ is the hitting time of $\partial C_k$ of the $d_k$-dimensional Brownian motion starting at $x_{(k)}$, $1\leq k \leq 2$.
Consequently, 
\begin{equation*}
    \mathbb{E}\left\{\tau^x_{\partial C}\right\}=\mathbb{E}\left\{\tau^{x_{(1)}}_{\partial C_1}\wedge \tau^{x_{(2)}}_{\partial C_2}\right\}\leq  \mathbb{E}\left\{\tau^{x_{(1)}}_{\partial C_1}\right\}\wedge \mathbb{E}\left\{\tau^{x_{(2)}}_{\partial C_2}\right\}=v_{C_1,\alpha}\left(x_{(1)}\right)\wedge v_{C_2,\alpha}\left(x_{(2)}\right),
\end{equation*}
which proves \eqref{eq:products}.

To prove \eqref{eq:products_h} it is sufficient to note that 
\begin{equation*}
   \left[\tau^x_{\Gamma_0}<\tau^x_{\Gamma_1}\right]\subset  \left[\tau^{x_{(1)}}_{\Gamma_0^{(1)}}<\tau^{x_{(1)}}_{\Gamma_1^{(1)}}\right],
\end{equation*}
where $\Gamma_1^{(1)}=\partial C_1\setminus \Gamma_0^{(1)}$.
Consequently,
\begin{equation*}
    h_{C,\Gamma_0}(x)=\mathbb{P}\left(\tau^x_{\Gamma_0}<\tau^x_{\Gamma_1}\right)\leq\mathbb{P}\left(\tau^{x_{(1)}}_{\Gamma_0^{(1)}}<\tau^{x_{(1)}}_{\Gamma_1^{(1)}}\right)=h_{C_1,\Gamma_0^{(1)}}(x_{}(1)), \quad x\in D.
\end{equation*}
\end{proof}

\subsection{Estimates for \texorpdfstring{$u_g$}{ug} and \texorpdfstring{$u_f$}{uf} in terms of exit times}
\label{subsec:ugufexit}
Let $D$ be a bounded open set in $\mathbb{R}^d$, and $u_g, u_f$ be given by \eqref{eq:prob_sol_split}.
In this subsection we derive explicit (boundary) H\"older estimates for $u_g$ and $u_f$ in terms of the functions $v_{D,\alpha}$ and $h_{D,\Gamma_0}$ given by \eqref{eq:v_Dalpha} and \eqref{def:hDGamma0}, respectively, without further regularity conditions imposed on $D$.

Recalling the notations $|g|_\alpha^{x_0,r_0}$, $|g|_\alpha^{x_0}$, and $|g|_\alpha$ introduced in \eqref{eq:notation_holder_seminorms_1}-\eqref{eq:notation_holder_seminorms_2}, 
the first main estimate of this subsection concerns $u_g$:
\begin{prop}\label{prop:u_g}
Let $0<\alpha\neq 1$, $x_0\in \partial D$, $\Gamma_0\subset \partial D$ be an open portion of the boundary that contains $x_0$, and set 
$
\Gamma_1:=\partial D\setminus \Gamma_0.    
$
Further, let $C(\alpha, d)$ be the constant appearing in \eqref{e:bmLBDG}.
Then, the following estimates hold:
\begin{enumerate}
    \item[(i)] For every $x\in D$ we have
    \begin{equation}
       |u_g(x)-g(x_0)|\leq |g|_\alpha^{x_0,\Gamma_0} 2^{(\alpha-1)^+}\left[C(\alpha,d) v_{D,\alpha/2}(x)+\left\|x-x_0\right\|^{\alpha}\right]+\sup_{y\in \Gamma_1}|g(y)-g(x_0)| h_{D,\Gamma_0}(x).
    \end{equation} where $h_{D,\Gamma_0}$ is defined in \eqref{def:hDGamma0}
    \item[(ii)] For every $x\in D$ we have
    \begin{equation}
        |u_g(x)-g(x_0)|\leq |g|^{x_0}_\alpha 2^{(\alpha-1)^+}\left[C(\alpha,d) v_{D,\alpha/2}(x)+\left\|x-x_0\right\|^{\alpha}\right].
    \end{equation}
    \item[(iii)] If $\alpha<1$ then for every $x,y\in D$ we have
    \begin{equation}
        |u_g(x)-u_g(y)|\leq 2|g|_\alpha\left[C(\alpha,d) v_{\infty,\alpha/2}(\|x-y\|)+\left\|x-y\right\|^{\alpha}\right].
    \end{equation}
\end{enumerate}
\end{prop}
\begin{proof}
(i). Using \eqref{eq:prob_sol_split} we have
\begin{align*}
|u_g(x)-g(x_0)|
&\leq \mathbb{E}\left[\left|g\left(B^x_{\tau^x_{\partial D}}\right)-g(x_0)\right|\right]\\
&\leq \mathbb{E}\left[\left|g\left(B^x_{\tau^x_{\partial D}}\right)-g(x_0)\right|;\tau^x_{\Gamma_0}\leq \tau^x_{\Gamma_1}\right]+\mathbb{E}\left[\left|g\left(B^x_{\tau^x_{\partial D}}\right)-g(x_0)\right|;\tau^x_{\Gamma_0}> \tau^x_{\Gamma_1}\right] \\
&\leq|g|_\alpha^{x_0,\Gamma_0}\mathbb{E}\left[\left\|B^x_{\tau^x_{\partial D}}-x_0\right\|^{\alpha}\right]
+\sup_{y\in \Gamma_1}|g(y)-g(x_0)| h_{D,\Gamma_0}(x)\\
&\leq |g|_\alpha^{x_0,\Gamma_0} 2^{(\alpha-1)^+}\left\{\mathbb{E}\left[\left\|B_{\tau^x_{\partial D}}\right\|^{\alpha}\right]+\left\|x-x_0\right\|^{\alpha}\right\}+\sup_{y\in \Gamma_1}|g(y)-g(x_0)| h_{D,\Gamma_0}(x)\\
\intertext{so by inequality \eqref{e:bmLBDG} we can continue with}
&\leq |g|_\alpha^{x_0,\Gamma_0} 2^{(\alpha-1)^+}\left\{C(\alpha,d)\mathbb{E}\left[\left(\tau^x_{\partial D}\right)^{\alpha/2}\right]+\left\|x-x_0\right\|^{\alpha}\right\}+\sup_{y\in \Gamma_1}|g(y)-g(x_0)| h_{D,\Gamma_0}(x)\\
&=|g|_\alpha^{x_0,\Gamma_0} 2^{(\alpha-1)^+}\left[C(\alpha,d) v_{D,\alpha/2}(x)+\left\|x-x_0\right\|^{\alpha}\right]+\sup_{y\in \Gamma_1}|g(y)-g(x_0)| h_{D,\Gamma_0}(x).
\end{align*}

\medskip
(ii) Follows from (i) by taking $\Gamma_0=\partial D$, since in this case $h_{D,\Gamma_0}=0$.

\medskip
(iii). Again, using \eqref{eq:prob_sol_split} and the H\"older regularity of $g$, we have
\begin{align*}
|u_g(x)-u_g(y)|&\leq \mathbb{E}\left[\left|g\left(B^x_{\tau^x_{\partial D}}\right)-g\left(B^y_{\tau^y_{\partial D}}\right)\right|\right]\leq |g|_\alpha\mathbb{E}\left[\left\|B^x_{\tau^x_{\partial D}}-B^y_{\tau^y_{\partial D}}\right\|^{\alpha}\right]\\
&= |g|_\alpha \left\{\mathbb{E}\left[\left|B^x_{\tau^x_{\partial D}}-B^y_{\tau^y_{\partial D}}\right|^{\alpha};\tau^x_{\partial D}\leq \tau^y_{\partial D} \right]+\mathbb{E}\left[\left|B^x_{\tau^x_{\partial D}}-B^y_{\tau^y_{\partial D}}\right|^{\alpha};\tau^y_{\partial D}\leq \tau^x_{\partial D} \right]\right\}.
\end{align*}
Now,
\begin{align*}
\mathbb{E}\left[\left\|B^x_{\tau^x_{\partial D}}-B^y_{\tau^y_{\partial D}}\right\|^{\alpha};\tau^x_{\partial D}\leq \tau^y_{\partial D} \right]
&\leq \mathbb{E}\left[\left\|B^x_{\tau^x_{\partial D}}-B^y_{\tau^x_{\partial D}}\right\|^{\alpha}\right]+\mathbb{E}\left[\left\|B^y_{\tau^x_{\partial D}}-B^y_{\tau^y_{\partial D}}\right\|^{\alpha};\tau^x
\leq \tau^y \right]\\
&\leq \|x-y\|^{\alpha}+\mathbb{E}\left[\left\|B^y_{\tau^x\wedge\tau^y}-B^y_{\tau^y}\right\|^{\alpha}\right],\\
\intertext{so by the strong Markov property, and by temporarily setting $h(z):=\mathbb{E}\left[\left\|B_{\tau^z_{\partial D}}\right\|^\alpha\right], z\in D$, we have}
\mathbb{E}\left[\left\|B^x_{\tau^x_{\partial D}}-B^y_{\tau^y_{\partial D}}\right\|^{\alpha};\tau^x_{\partial D}\leq \tau^y_{\partial D} \right]&\leq\|x-y\|^{\alpha} + \mathbb{E}\left[h\left(B^y_{\tau^x\wedge\tau^y}\right)\right].\\
\intertext{But now, similarly to the computations in (i), $h(z)\leq  C(\alpha,d) v_{D,\alpha/2}(z)$, hence}
\mathbb{E}\left[\left\|B^x_{\tau^x_{\partial D}}-B^y_{\tau^y_{\partial D}}\right\|^{\alpha};\tau^x_{\partial D}\leq \tau^y_{\partial D} \right]&\leq\|x-y\|^{\alpha} + C(\alpha,d)\mathbb{E}\left[v_{D,\alpha/2}\left(B^y_{\tau^x\wedge\tau^y}\right)\right]\\
&\leq \|x-y\|^{\alpha} + C(\alpha,d) v_{\infty,\alpha/2}(\|x-y\|),
\end{align*}
where the last inequality follows from the fact that $B^y_{\tau^x\wedge\tau^y}$ is at distance at most $\|x-y\|$ from the boundary.

By symmetry, we also have
\begin{equation*}
\mathbb{E}\left[\left\|B^x_{\tau^x_{\partial D}}-B^y_{\tau^y_{\partial D}}\right\|^{\alpha};\tau^y_{\partial D}\leq \tau^x_{\partial D} \right] \leq \|x-y\|^{\alpha} + C(\alpha,d) v_{\infty,\alpha/2}(\|x-y\|), 
\end{equation*}
so (iii) is proved.
\end{proof}

\begin{rem}
    It is worth to emphasizing that in \Cref{prop:u_g}, (i) and (ii), the exponent $\alpha$ is allowed to be greater than $1$. 
    Of course, an $\alpha$-H\"older continuous function $g$ with $\alpha>1$ has to be constant. 
    However, in (i) and (ii) the H\"older regularity of $u_g$ is tested with the boundary point $x_0$ being fixed, and an exponent $\alpha>1$ is possible for a non-constant boundary data; 
    take, e.g. $g(x)=\|x-x_0\|^\alpha, \in \partial D, \alpha>1$.
\end{rem}

\begin{rem}\label{rem:h_g=0}
    A particular case of \Cref{prop:u_g}, (i) is when $g=0$ on $\Gamma_0$. 
    In this case, the estimate in (i) reduces to
    \begin{equation}\label{eq:h_g=0}
        |u_g(x)|\leq |g|_\infty h_{D,\Gamma_0}(x), \quad x\in D.
    \end{equation}
    Note that this inequality can in fact be easily obtained from the maximum principle for harmonic functions.
\end{rem}

In the next result we deal with $u_f$:
\begin{prop}\label{prop:u_f}
For every $p>1, 1/p+1/q=1$, and $f\geq 0$, we have
\begin{equation}
    0\leq u_f(x)\leq \left(v_D(x)\right)^{1/p}\left(u_{f^q}(x)\right)^{1/q}, \quad x\in D. 
\end{equation}
In particular, 
\begin{equation*}
0\leq u_f(x)\leq v_D(x)\|f\|_\infty, \quad x\in D,    
\end{equation*}
and if $0\leq f\in L^\gamma(D)$ with $\gamma>\frac{dq}{2}$, then 
\begin{equation}
    0\leq u_f(x)\leq 
   C(d,D,\gamma,q)\left(v_D(x)\right)^{1/p} \|f\|_{L^\gamma(D)} \quad x\in D,
\end{equation}
where
\begin{equation}\label{eq:C_f}
    C(d,D,\gamma,q)=
    \begin{cases}
    \left[\frac{(d-2)(\gamma-q)}{(d(d-2)\omega_d)^{\frac{q}{\gamma-q}}(2\gamma-qd)}\right]^{\frac{\gamma-q}{\gamma q}}{\sf diam}(D)^{\frac{2\gamma-qd}{pq}} & d\geq 3\\
   \frac{{\sf diam}(D)^{\frac{2(\gamma-q)}{\gamma q}}}{(4\pi)^{\frac{1}{q}}} \left[\Gamma\left(\frac{\gamma}{\gamma-q}+1\right)\right]^{\frac{\gamma-q}{\gamma q}} &\quad d=2
    \end{cases}, \quad x\in D.
\end{equation}
Above, $\Gamma$ denotes the Gamma function.
\end{prop}
\begin{proof}
It is clear from the representation \eqref{eq:prob_sol_split} that $u_f\geq 0$. 
Further, let $p>1, 1/p+1/q=1$. By Holder inequality we get
\begin{align}
u_f(x)&=\mathbb{E}\left\{\int_0^{\tau_{\partial D}^x}f\left(B^x(t)\right) dt\right\} \leq \mathbb{E}\left\{ \left(\tau^x_{\partial D}\right)^{1/p}\left(\int_0^{\tau_{\partial D}^x}f^q\left(B^x(t)\right) dt\right)^{1/q}\right\}\\
&\leq \mathbb{E}\left\{\tau_{\partial D}^x\right\}^{1/p}\left(\mathbb{E}\left\{ \int_0^{\tau_{\partial D}^x}f^q\left(B^x(t)\right) dt\right\}\right)^{1/q}=\left(v_D(x)\right)^{1/p}\left(u_{f^q}(x)\right)^{1/q}, \quad x\in D,\label{eq:u_f_proof}
\end{align}
which proves the first part of the statement.

Further, set $R:={\sf diam}(D)$, $x\in D$ and $D_x:=B(x,R)$, so that $D\subset D_x$ and hence $\tau^x_{\partial D}\leq \tau^x_{\partial D_x}$ for every $x\in D$. 
Moreover, we extend $f$ from $D$ to $\mathbb{R}^d$ by setting $f\equiv 0$ on $\mathbb{R}^d\setminus D$.
Consequently, using \eqref{eq:u_f_proof} and \Cref{rem:PDEcorrespondence}, we have
\begin{align*}
u_f(x)
&\leq \left(v_D(x)\right)^{1/p}\left(\int_{D_x} G_{D_x}(x,y)f^{q}(y) dy \right)^{1/q},
\end{align*}
where recall that $G_{D_x}$ is the Green function of $-\Delta_0$ on $D_x$.
Consider now $q<\gamma$, and apply H\"older inequality to get that
\begin{align*}
u_f(x)
&\leq\left(v_D(x)\right)^{1/p}\left(\int_{D_x} \left(G_{D_x}(x,y)\right)^{\frac{\gamma}{\gamma-q}} dy \right)^{\frac{\gamma-q}{\gamma q}}\left(\int_{D_x}f^{\gamma }(y) dy \right)^{1/\gamma }.
\end{align*}
Further, recall that
\begin{equation*}
    G_{D_x}(x,y)=
    \begin{cases}\frac{1}{d(d-2)\omega_d}\left(\frac{1}{\|x-y\|^{d-2}}-\frac{1}{R^{d-2}}\right),& \quad d\geq 3\\
    \frac{1}{2\pi}\left(\log{R}-\log{\|x-y\|}\right),& \quad d=2
    \end{cases},
\end{equation*}
where $\omega_d$ is the volume of the unit ball in $\mathbb{R}^d$.
Thus, if $d=2$ we obtain
\begin{align*}
\int_{D_x} \left(G_{D_x}(x,y)\right)^{\frac{\gamma }{\gamma -q}} dy&= \frac{1}{(2\pi)^{\frac{\gamma }{\gamma -q}}} \int_{D_x}\left(\log{\frac{R}{\|x-y\|}}\right)^\frac{\gamma }{\gamma -q} dx =  \frac{1}{(2\pi)^{\frac{\gamma }{\gamma -q}}} \int_0^{R}2\pi s\left(\log{\frac{R}{s}}\right)^\frac{\gamma }{\gamma -q} ds,\\
\intertext{so, by the change o variables $s=Re^{-t}$ we can continue with}
&=\frac{R^2}{(2\pi)^{\frac{\gamma }{\gamma -q}-1}} \int_0^{\infty} e^{-2t}t^\frac{\gamma }{\gamma -q} dt=\frac{R^2}{(2\pi)^{\frac{\gamma }{\gamma -q}-1}} \frac{1}{2^{\frac{\gamma }{\gamma -q}+1}} \Gamma\left(\frac{\gamma }{\gamma -q}+1\right)\\
&=\frac{R^2}{(4\pi)^{\frac{\gamma }{\gamma -q}}} \Gamma\left(\frac{\gamma }{\gamma -q}+1\right).
\end{align*}
Consequently, if $d=2$ we obtain
\begin{align*}
u_f(x)
&\leq\left(v_d(x)\right)^{1/p}\frac{R^{\frac{2(\gamma -q)}{\gamma q}}}{(4\pi)^{\frac{1}{q}}} \left[\Gamma\left(\frac{\gamma }{\gamma -q}+1\right)\right]^{\frac{\gamma -q}{\gamma q}} \|f\|_{L^\gamma (D)}.
\end{align*}

For the case $d\geq 3$, taking $1<q<\gamma $ such that $\gamma /q>d/2$, we have:
\begin{align*}
\int_{D_x} \left(G_{D_x}(x,y)\right)^{\frac{\gamma }{\gamma -q}} dy&\leq \frac{1}{(d(d-2)\omega_d)^{\frac{\gamma }{\gamma -q}}} \int_{D_x}\frac{1}{\|x-y\|^\frac{\gamma (d-2)}{\gamma -q}} dx=\frac{d\omega_d}{(d(d-2)\omega_d)^{\frac{\gamma }{\gamma -q}}} \int_0^{R}\frac{1}{t^{\frac{\gamma (d-2)}{\gamma -q}-(d-1)}} \;dt\\
&=\frac{d-2}{(d(d-2)\omega_d)^{\frac{q}{\gamma -q}}}\frac{1}{d-\frac{\gamma (d-2)}{\gamma -q}}R^{d-\frac{\gamma (d-2)}{\gamma -q}}=\frac{(d-2)(\gamma -q)}{(d(d-2)\omega_d)^{\frac{q}{\gamma -q}}(2\gamma -qd)}R^{\frac{2\gamma -qd}{\gamma -q}}.
\end{align*}
Consequently,
\begin{align*}
u_f(x)
&\leq\left(v_D(x)\right)^{1/p}\left[\frac{(d-2)(\gamma -q)}{(d(d-2)\omega_d)^{\frac{q}{\gamma -q}}(2\gamma -qd)}\right]^{\frac{\gamma -q}{\gamma q}}{\sf diam}(D)^{\frac{2\gamma -qd}{\gamma q}}\|f\|_{L^\gamma (D)},
\end{align*}
which finishes the proof.
\end{proof}

\section{Brownian exit time estimates in general bounded domains}\label{sec:exitgrl}

Let us note  that the estimates for $u_g$ obtained in \Cref{prop:u_g} are given in terms of $v_{D,\alpha/2}$ and $h_{D,\Gamma_0}$, whilst those for $u_f$ obtained in \Cref{prop:u_f} are given in terms of $v_{D}$.
The next step is to estimate $v_{D,\alpha/2}$, $h_{D,\Gamma_0}$, and $v_{D}$, in terms of the distance function to the boundary, with  explicit constants depending the geometry of the domain $D$, $g$, $f$, and the dimension $d$.  
Let us recall that the functions $v_D$ and $h_{D,\Gamma_0}$ are both solutions to some nice PDEs, more precisely to \eqref{e:v} and \eqref{e:h}, respectively, which make them easier to estimate by tools like the maximum principle and barrier functions. 
This is not immediately the case for $v_{D,\alpha/2}$ if $\alpha\neq 2$. 
A simple treatment (leading to sub-optimal estimates) is to rely on Jensen's inequality for $\alpha\leq 2$: 
\begin{equation}\label{eq:naive}
  v_{D,\alpha/2}(x)=\mathbb{E}\left[\left(\tau^x_{\partial D}\right)^{\alpha/2}\right]\leq \left(\mathbb{E}\left[\tau^x_{\partial D}\right] \right)^{\alpha/2} =\left(v_D(x)\right)^{\alpha/2}, \quad x\in D, \alpha\leq 2.
\end{equation}
The estimation of $v_{D,\alpha/2}$ can be thus reduced to that of $v_D$.
However, as mentioned, it is not hard to see that this bound does not lead to optimal estimates for $v_{D,\alpha/2}$: For example, by employing \Cref{prop:u_g}, one can easily see that if $\alpha=1$ and $D$ is a ball, by using \eqref{eq:naive} the resulting H\"older exponent of the solution $u_g$ with $g$ Lipschitz on $\partial D$, is $1/2$. 
However, in this case it is well known that $u_g$ is $\alpha'$-H\"older continuous for any $\alpha'<1$; see e.g. \cite{aikawa2002holder}.

The role of this section is to develop techniques that can lead to sharp estimates for $v_{D,\alpha/2}$.
Our approach is general and provides two completely different strategies of estimating $v_{D,\alpha/2}$, $h_{D,\Gamma_0}$, and $v_{D}$ for general domains $D$: The first one is to associate a radial maximal function for a fixed boundary point, taken to be $0$ for simplicity, and any of the previous three functions, and then to rely on the elementary estimate in \Cref{lem:Phi} below in order to get the desired H\"older estimates. 
The second strategy is in our opinion finer, and it relies on Ito's formula applied for some barrier function $W$; we would like to point out here that the key idea is not to average out the martingale part in Ito's formula applied to $W$, but to estimate it properly from above and below using Doob's maximal inequality and BDG inequality \eqref{e:BDG}.

\begin{rem}
Let us point out that when $x_0\in \partial D$ and there exists an infinite cone with vertex $x_0$ lying entirely in the complement of $D$, then estimates for $v_{C,\alpha}(x)$ in terms of $\|x-x_0\|$ have been obtained in \cite[Proposition 2.1]{DeB87} based on the construction of a harmonic majorant of $\|\cdot\|^{2\alpha}$ on the complement of the cone, and techniques previously developed in \cite{burkholder1977exit}.
However, the harmonic majorant is given in terms of some hypergeometric function, and as a consequence, the resulting estimates depend non-explicitly on the dimension $d$ and the angle of the cone $C$.
\end{rem} 

\subsection{The first approach: A reverse doubling inequality}\label{subsec:reversedoubling}

The first approach to proving estimates for $v_{D,\alpha}$ is based on the following elementary lemma, which is in fact a version of \cite[Lemma 8.23]{GTbook}:
\begin{lem}\label{lem:Phi}
    Let $r_0>0$ and $\Phi:[0,r_0]\rightarrow [0,\infty)$ be a  bounded function, for which there exist constants $c>0$, $a\geq 0$, and $\delta\in (0,1)$ such that
    \begin{equation}\label{eq:Phi}
        \Phi(r)\leq c r^a+ \delta \Phi(2r), \quad r\in[0,r_0/2].
    \end{equation}
    Then 
    \begin{equation}
        \Phi(r)\leq
        \begin{cases}
             cr_0^a \frac{r^a}{1-2^a\delta}+r^{|\log_2(\delta)|}\frac{|\Phi|_\infty}{\delta}, \quad &\delta< 1/2^a\\[2mm]
             \left[\frac{cr_0^a}{2^a\delta-1}+\frac{|\Phi|_\infty}{\delta}\right]r^{|\log_2(\delta)|}, \quad &\delta>1/2^a\\[2mm]
             \left[cr_0^a\log_2(1/r)+\frac{|\Phi|_\infty}{\delta}\right]r^a, \quad &\delta=1/2^a
        \end{cases}
        , \quad 0\leq r\leq r_0/2.
    \end{equation}
    In particular, if $\delta\neq 1/4$, then
    \begin{equation}
        \Phi(r)\leq C r^\gamma, \quad r\in[0,r_0/2], \quad \mbox{with } \gamma:=2\wedge|\log_2(\delta)|\;\mbox{and}
        \; C:= |\Phi|_\infty + \frac{cr_0^2}{|1-4\delta|}.
    \end{equation}
\end{lem}
\begin{proof}
First of all, note that by defining $\tilde{\Phi}(r):=\Phi(r_0r), r\in [0,1]$, we have that $\tilde{\Phi}$ satisfies \eqref{eq:Phi} with $r_0=1$ and $r_0^ac$ instead of $c$.
Thus, we shall further assume without loss of generality that $r_0=1$.

Let $r\in (0,1/2]$ and $k\geq 1$ be such that $2^{-(k+1)}< r\leq 2^{-k}$.
Thus, we have $k\leq \log_2(1/r)<k+1$.
Further, by iterating \eqref{eq:Phi} we get
\begin{align*}
\Phi(r)
&\leq c r^a\left[1+(2^a\delta)+\dots+(2^a\delta)^{k-1}\right]+\delta^k \sup_{1/2\leq s\leq 1}\Phi(s)
\leq c r^a\frac{(2^a\delta)^k-1}{2^a\delta-1}+\frac{\delta^{k+1}}{\delta}|\Phi|_\infty\\
&\leq c r^a\frac{(2^a\delta)^k-1}{2^a\delta-1}+r^{|\log_2(\delta)|}\frac{|\Phi|_\infty}{\delta},
\quad 0\leq r\leq 1/2.
\end{align*}
If $\delta<1/2^a$, then
\begin{align*}
    \Phi(r)\leq c r^a\frac{1}{1-2^a\delta}+r^{|\log_2(\delta)|}\frac{|\Phi|_\infty}{\delta} , \quad 0\leq r\leq 1/2.
\end{align*}
If $\delta>1/2^a$, then
\begin{align*}
    \Phi(r)
    \leq c \frac{\delta^k}{2^a\delta-1}+r^{|\log_2(\delta)|}\frac{|\Phi|_\infty}{\delta}
    \leq \left[\frac{c}{2^a\delta-1}+\frac{|\Phi|_\infty}{\delta}\right]r^{|\log_2(\delta)|}, \quad 0\leq r\leq 1/2.
\end{align*}
if $\delta=1/2^a$, then
\begin{equation*}
    \Phi(r)\leq cr^ak+\delta^k |\Phi|_\infty\leq \left[c\log_2(1/r)+\frac{|\Phi|_\infty}{\delta}\right]r^a,  \quad 0\leq r\leq 1/2.
\end{equation*}
\end{proof}

Based on \Cref{lem:Phi}, we can prove the following main result:
 
\begin{thm}\label{thm:Phi_hv}
Let $D$ be a bounded and open set in $\mathbb{R}^d$, $x_0=0\in \partial D$, $r_0>0$, $\Gamma_0:=B(0,r_0)\cap \partial D$, and $\Gamma_1:=\partial D \setminus \Gamma_0$.
Further, let $v_{D,\alpha/2}$ and $h_{D,\Gamma_0}$ be given by \eqref{eq:v_Dalpha}  and \eqref{def:hDGamma0}, respectively, whilst $\delta\in [0,1]$ be given by
\begin{equation}\label{eq:delta}
    \delta=\delta(0,r_0):=\sup\limits_{r\leq r_0/2}\sup\limits_{y\in D\cap S(0,r)}\bar{h}_r(y), \quad \mbox{where} \quad     \bar{h}_r(y)=\mathbb{P}\left(\tau^y_{D\cap S(0,2r)}<\tau^y_{\Gamma_0}\right), \sup_{y\in \emptyset}\overline{h}_r(y):=0.
\end{equation}
If $\delta<1$ then 
\begin{equation}\label{Phi_h}
    h_{D,\Gamma_0}(x)\leq\frac{\|x\|^{|\log_2(\delta)|}}{\delta}, \quad x\in D\cap B(0,r_0/2),
\end{equation}
whilst for every $0<\alpha\neq 1$ and $x\in D\cap B(0,r_0/2)$ we have
\begin{equation}\label{Phi_v}
 \quad v_{D,\alpha/2}(x)
 \leq 
 \begin{cases}
     \frac{cr_0^\alpha}{1-2^\alpha\delta}\|x\|^\alpha
     +\frac{{\sf diam}(D)^{2\alpha}}{\delta d^\alpha}\|x\|^{|\log_2(\delta)|}, \quad & \mbox{if } \delta<1/2^\alpha\\[2mm]
     \left[\frac{cr_0^\alpha}{2^\alpha\delta-1}
     +\frac{{\sf diam}(D)^{2\alpha}}{\delta d^\alpha}\right]\|x\|^{|\log_2(\delta)|}, \quad & \mbox{if } \delta>1/2^\alpha\\[2mm]
     \left[cr_0^\alpha|\log_2(\|x\|)|
     +2^\alpha\frac{{\sf diam}(D)^{2\alpha}}{ d^\alpha}\right]\|x\|^\alpha, \quad & \mbox{if } \delta=1/2^\alpha
 \end{cases},
\end{equation}
where 
\begin{equation}\label{eq:c}
    c=c_{d,\alpha}=\frac{3^\alpha}{c_{\alpha}}\left(\dfrac{\alpha\vee 1}{|\alpha-1|}\right)^{\alpha\vee 1}d^{\frac{\alpha}{2}\left(\frac{2}{\alpha}-1\right)^+-1},
\end{equation}
whilst $c_\alpha$ is the constant appearing in the BDG inequality \eqref{e:BDG} from \Cref{t:BDG}.
\end{thm}

\begin{rem}\label{rem:uniform_delta_far}
    Note that if $x\in D$ satisfies $\|x\|\geq r_0/2$, then by \Cref{lem:v_alpha_uniform} we get
    \begin{equation*}
        v_{D,\alpha/2}(x)
        \leq C_1(\alpha/2,d){\sf diam}(D)^\alpha
        \leq C_1(\alpha/2,d){\sf diam}(D)^\alpha \left(\frac{2}{r_0}\right)^\alpha \|x\|^\alpha,
    \end{equation*}
    where $C_1(\alpha/2,d)$ is given in \eqref{eq:v_alpha_uniform}.
\end{rem}
\begin{rem}\label{rem:delta<delta_hat}
\begin{enumerate}
    \item[(i)] Note that $\delta$ appearing in \Cref{thm:Phi_hv} can be estimated from above by 
    \begin{equation}\label{eq:delta_hat}
        \delta\leq \hat\delta, \quad \mbox{where} \quad \hat\delta:=\sup\limits_{r\leq r_0/2}\sup\limits_{y\in D\cap S(0,r)}\hat{h}_r(y), \quad     \hat{h}_r(y)=\mathbb{P}\left(\tau^y_{D\cap S(0,2r)}<\tau^y_{D^c \cap S(0,2r)}\right).
    \end{equation}
    Furthermore, note that $\hat{h}_r$ is the probabilistic solution to
    \begin{equation}\label{eq:PDE_h}
        \Delta \hat{h}_r=0 \textrm{ in } B(0,2r), \quad \hat{h}_r=0 \textrm{ on }D^c \cap S(0,2r) \quad \mbox{ and } \quad \hat{h}_r=1 \textrm{ on }D \cap S(0,2r),
    \end{equation}
    in particular admits the Poisson representation
    \begin{equation}\label{eq:Poisson_h}
        \hat{h}_r(y)=\int_{D\cap S(0,2r)} \frac{(2r)^2-\|y\|^2}{ 2r\mathrm{Area}(S(0,1))\|y-\xi\|^d}\;\sigma(d\xi),
    \end{equation}
    where $\sigma$ is the surface measure on $S(0,2r)$. 
    \item[(ii)] 
    A relevant class of domains for which we shall show later on that $\hat{\delta}<1$ is given by those domains that satisfy an exterior cone condition; see \Cref{coro:delta_omega} below.
\end{enumerate}
\end{rem}

\begin{proof}[Proof of \Cref{thm:Phi_hv}]
Let us define
\begin{align}
    \Phi_h(r):=\sup_{y\in D\cap S(0,r)} h_{D,\Gamma_0}(y), \quad \mbox{and} \quad \Phi_v(r):=\sup_{y\in D\cap S(0,r)} v_{D,\alpha}(y), \quad   \quad r\leq r_0.
\end{align}
By setting $\delta_r:=\sup\limits_{y\in D\cap S(0,r)}\mathbb{P}\left(\tau^y_{D\cap S(0,2r)}<\tau^y_{\Gamma_0}\right)$, we claim that
\begin{equation}\label{claim:Phi_hv}
    \Phi_h(r)\leq \delta_r \Phi_h(2r),\quad \mbox{and} \quad \Phi_v(r)\leq c r^\alpha + \delta_r \Phi_v(2r),\quad r\leq r_0/2.
\end{equation}
To show \eqref{claim:Phi_hv}, by the strong Markov property we have for $y\in D\cap S(0,r), \; r\leq r_0/2$
\begin{align*}
    h_{D,\Gamma_0}(y)
    &=\mathbb{P}\left(\tau^y_{\Gamma_1}<\tau^y_{\Gamma_0}\right)
    =\mathbb{P}\left(\tau^y_{\Gamma_1}<\tau^y_{\Gamma_0},\tau^y_{D\cap S(0,2r)}<\tau^y_{\Gamma_0}\right)
    =\mathbb{E}\left[h_{D,\Gamma_0}\left(B^y_{\tau^y_{D\cap S(0,2r)}}\right);\tau^y_{D\cap S(0,2r)}<\tau^y_{\Gamma_0}\right]\\
    &\leq \delta_r \Phi_h(2r).
\end{align*}
Thus, we can apply \Cref{lem:Phi} for $\Phi$ replaced by $\Phi_h$ and $c=0$, noting that $|\Phi_h|_\infty\leq 1$, to obtain \eqref{Phi_h}.

Similarly,
\begin{align*}
    v_{D,\alpha/2}(y)
    &=\mathbb{E}\left[\left(\tau^y_{\partial D}\right)^{\alpha/2}\right]
    =\mathbb{E}\left[\left(\tau^y_{\partial D}\right)^{\alpha/2};\tau^y_{D\cap S(0,2r)}<\tau^y_{\partial D}\right]
    +\mathbb{E}\left[\left(\tau^y_{\partial D}\right)^{\alpha/2};\tau^y_{D\cap S(0,2r)}\geq\tau^y_{\partial D}\right]\\
    &\leq\mathbb{E}\left[v_{D,\alpha/2}\left(B^y_{\tau^y_{D\cap S(0,2r)}}\right);\tau^y_{D\cap S(0,2r)}<\tau^y_{\Gamma_0}\right]
    +\mathbb{E}\left[\left(\tau^y_{\partial D\cap S(0,2r)}\right)^{\alpha/2};\tau^y_{D\cap S(0,2r)}\geq\tau^y_{\partial D}\right]\\
    &\leq \delta_r \Phi_v(2r)
    +\mathbb{E}\left[\left(\tau^y_{ S(0,2r)}\right)^{\alpha/2}\right]\\
    &\leq \delta_r \Phi_v(2r)
    +\frac{1}{c_{\alpha}}\left(\dfrac{\alpha\vee 1}{|\alpha-1|}\right)^{\alpha\vee 1}d^{\frac{\alpha}{2}\left(\frac{2}{\alpha}-1\right)^+-1}\left(\mathbb{E}\left[\left\| B_{\tau^y_{ S(0,2r)}}\right\|^2\right]\right)^{\alpha/2}\\
    &=\delta_r \Phi_v(2r)+\frac{1}{c_{\alpha}}\left(\dfrac{\alpha\vee 1}{|\alpha-1|}\right)^{\alpha\vee 1}d^{\frac{\alpha}{2}\left(\frac{2}{\alpha}-1\right)^+-1}3^\alpha r^\alpha,
\end{align*}
where the last inequality follows by \eqref{eq:tau-B} of \Cref{lem:tau-B}.
Thus, we can apply \Cref{lem:Phi} for $\Phi$ replaced by $\Phi_v$ and $c=\frac{3^\alpha}{c_{\alpha}}\left(\dfrac{\alpha\vee 1}{|\alpha-1|}\right)^{\alpha\vee 1}d^{\frac{\alpha}{2}\left(\frac{2}{\alpha}-1\right)^+-1}$, and $a=\alpha$, noting that $|\Phi_v|_\infty\leq \frac{{\sf diam}(D)^{2}\alpha}{d^\alpha}$ according to \eqref{lem:v_alpha_uniform}, to deduce \Cref{Phi_v}.
\end{proof}

In what follows we derive general estimates for $u_g$ and $u_f$ given by \eqref{eq:prob_sol_split}, for domains for which $\delta$ given by \eqref{eq:delta} satisfies $\delta<1$.
We shall fundamentally use the estimate \eqref{Phi_h}, with the mention that for simplicity we shall consider only the cases $\delta<1/{2^\alpha}$ and $\delta>1/{2^\alpha}$; as a matter of fact, there is no loss in doing so because the case $\delta=1/{2^\alpha}$ can be recovered by optimizing over $\delta$ the estimate obtained for $\delta>1/{2^\alpha}$.

\begin{coro}\label{coro:u_g_delta}
Let $D$ be a bounded and open set in $\mathbb{R}^d$, $x_0=0\in \partial D$, $r_0>0$, $\Gamma_0:=B(0,r_0)\cap \partial D$.
Further, assume that $\delta$ given by \eqref{eq:delta} satisfies $\delta<1$.
Further, let $C(\alpha,d)$ be the constant appearing in \eqref{e:bmLBDG}.
Then, the following assertions hold
\begin{enumerate}
    \item[(i)] For every $x\in D$ we have
        \begin{align*}
           |u_g(x)-g(0)| 
           \leq
           &|g|^{0}_\alpha\left\{\left[C_1\|x\|^{\alpha}+C_2\|x\|^{|\log_2(\delta)|}\right]1_{[0,r_0/2)}(\|x\|)
           \right.\\
           &\left.+C_1(\alpha/2,d){\sf diam}(D)^\alpha \left(\frac{2}{r_0}\right)^\alpha\|x\|^{\alpha}1_{[r_0/2,\infty)}(\|x\|)\right\},
        \end{align*}
        where $C_1(\alpha/2,d)$ is given by \eqref{eq:v_alpha_uniform}, 
        \begin{align}\label{eq:C_1C_2}
            &C_1:= 2^{(\alpha-1)^+}\left[C(\alpha,d)\tilde{C}_1+1\right],\quad 
            C_2:= 2^{(\alpha-1)^+}C(\alpha,d)\tilde{C}_2\\
            &\tilde{C}_1:=\begin{cases}
               \frac{cr_0^\alpha}{1-2^\alpha\delta} ,& \delta<1/2^\alpha\\
               0 ,& \delta>1/2^\alpha
            \end{cases}, 
            \quad
            \tilde{C}_2:=\begin{cases}
               \frac{{\sf diam}(D)^{2\alpha}}{\delta d^\alpha} ,& \delta<1/2^\alpha\\
               \frac{cr_0^\alpha}{2^\alpha\delta-1}+\frac{{\sf diam}(D)^{2\alpha}}{\delta d^\alpha} ,& \delta>1/2^\alpha
            \end{cases}
        \end{align}
        whilst $c=c_{d,\alpha}$ is given by \eqref{eq:c}.
    \item[(ii)] Assume in addition that for the given $r_0$, $\delta<1$ can be chosen uniformly for all $x_0\in \partial D$.
    Then, if $\alpha<1$ then for every $x,y\in D$
        \begin{align*}
            |u_g(x)-u_g(y)|
            &\leq  2|g|_\alpha\left\{\left[C_1\|x\|^{\alpha}+C_2\|x\|^{|\log_2(\delta)|}\right]1_{[0,r_0/2)}(\|x\|)
           \right.\\
           &\left.+\frac{{\sf diam}(D)^\alpha}{d^{\alpha/2}} \left(\frac{2}{r_0}\right)^\alpha\|x\|^{\alpha}1_{[r_0/2,\infty)}(\|x\|)\right\}
        \end{align*}
\end{enumerate}
where $C_1$ and $C_2$ are as above.
\end{coro}
\begin{proof}
Both assertions are immediate consequences of \Cref{prop:u_g}, (ii)-(iii), and of the estimate \eqref{Phi_v} in \Cref{thm:Phi_hv}, together with \Cref{rem:uniform_delta_far}.
\end{proof}

\begin{coro}\label{coro:u_f_delta}
Let $D$ be a bounded and open set in $\mathbb{R}^d$, $0\in \partial D$, $r_0>0$, $\Gamma_0:=B(0,r_0)\cap \partial D$.
Further, assume that $\delta$ given by \eqref{eq:delta} satisfies $\delta<1$.
Further, let $C(\alpha,d)$ be the constant appearing in \eqref{e:bmLBDG}.
Then, the following assertions hold
\begin{enumerate}
    \item[(i)] For every $f\in L^{\infty}(D)$ and every $x\in B(0,r_0/2)\cap D$ we have
    \begin{align}\label{eq:theta}
        |u_f(x)|
        &\leq \|f\|_{L^\infty(D)} \theta(\|x\|),\\
        \mbox{ where } \theta(\|x\|)
        &:=
        \frac{{\sf diam}(D)^2}{d} \left(\frac{2}{r_0}\right)^2 \|x\|^2 1_{[r_0/2,\infty)}(\|x\|)\\
        &\quad+1_{[0,r_0/2)}(\|x\|)\begin{cases}
         \frac{cr_0^2}{1-4\delta}\|x\|^2+\frac{{\sf diam}(D)^{4}}{\delta d^2}\|x\|^{|\log_2(\delta)|}, \quad & \mbox{if } \delta<1/4\\[1mm]
         \left[\frac{cr_0^2}{4\delta-1}+\frac{{\sf diam}(D)^{4}}{\delta d^2}\right]\|x\|^{|\log_2(\delta)|}, \quad & \mbox{if } \delta>1/4
        \end{cases}.
    \end{align}
    \item[(ii)] For every $p>1, 1/p+1/q=1$ and $f\in L^\gamma(D)$ with $ \gamma>dq/2$, and for every $x\in B(x,r_0/2)\cap D$
    \begin{equation}
        |u_f(x)|\leq C(d,D,\gamma,q)\|f\|_{L^\gamma(D)}\left[\theta(\|x\|)\right]^{1/p},
    \end{equation}
    where $C(d,D,\gamma,q)$ is given by \eqref{eq:C_f}.
\end{enumerate}
\end{coro}
\begin{proof}
    Both assertions follow by \Cref{prop:u_f} and \Cref{thm:Phi_hv} together with \Cref{rem:uniform_delta_far}.
\end{proof}

\subsection{The second approach: Ito's formula and barrier-type functions}\label{ss:second_approach}

Let us begin with two small results aimed to show that $h_{D,\Gamma_0}$ can be estimated by any function $\varphi$ that satisfies \eqref{eq:phi} from below, or  by $v_D$. The fact that $v_D$ can be estimated by $h_{D,\Gamma_0}$ is also true, at least for a certain class of domains; however, this is nontrivial and it is shown in \Cref{coro:equiv:h-v_D} below, based on the main result of this subsection, \Cref{prop:W}.

\begin{prop}\label{prop:h_by_phi}
Let $D$ be a (bounded and open) domain in $\mathbb{R}^d$, $\Gamma_0\subset \partial D$, and set $\Gamma_1:=\partial D \setminus \Gamma_0$. 
For every $\varphi$ which satisfies
\begin{equation}\label{eq:phi}
\varphi\in H^1_{\sf loc}(D)\cap C(\overline{D}),\quad  \Delta \varphi \leq 0 \mbox{ in } D, \quad \varphi\geq 0 \mbox{ on } \partial D, \quad \varphi\geq 1 \mbox{ on } \Gamma_1    
\end{equation}
we have
\begin{equation}\label{eq:varphi_bound}
    h_{D,\Gamma_0}(x)
        \leq \varphi(x).
\end{equation}
\end{prop}
\begin{proof}
\begin{equation}
h_{D,\Gamma_0}(x)=\mathbb{P}\left(\tau^x_{\Gamma_1}<\tau^x_{\Gamma_0}\right)\leq \mathbb{E}\left[ 1_{\Gamma_1}\left(B^x_{\tau^x_{\partial D}}\right)\right]\leq \mathbb{E}\left[ \varphi\left(B^x_{\tau^x_{\partial D}}\right)\right]\leq \varphi(x),   
\end{equation}
where the last equality follows by the maximum principle, or by the fact that $\varphi\left(B^x_{t\wedge\tau^x_{\partial D}}\right), t\geq 0$ is a continuous supermartingale.
\end{proof}

\begin{prop}\label{prop:general_v}
    Let $D$ be a (bounded and open) domain in $\mathbb{R}^d$, $\Gamma_0\subset \partial D$, and set $\Gamma_1:=\partial D \setminus \Gamma_0$.
    Then
    \begin{equation*}
        h_{D,\Gamma_0}(x) \leq \frac{d}{{\sf dist}(x, \Gamma_1)^2}v_D(x), \quad x\in D. 
    \end{equation*}
\end{prop}
\begin{proof}
Since $\|B(t)\|^2-dt, t\geq 0$ is a martingale, by Doob's stopping theorem
we have for every $x\in D$
\begin{equation*}
    v_{D}(x)=\mathbb{E}\left[\tau^x_{D}\right]=\frac{1}{d}\mathbb{E}\left[\|x-B^x_{\tau^x_{D}}\|^2\right]\geq \frac{1}{d}\mathbb{E}\left[\|x-B^x_{\tau^x_{D}}\|^2 ; \tau^x_{\Gamma_0}>\tau^x_{\Gamma_1}\right] \geq \frac{{\sf dist}(x, \Gamma_1 )^2}{d}\mathbb{P}\left(\tau^x_{\Gamma_0}>\tau^x_{\Gamma_1}\right).
\end{equation*}
\end{proof}

The main results of this subsection are the following proposition and its two corollaries.
\begin{prop}\label{prop:W}
Let $D$ be a (bounded and open) domain in $\mathbb{R}^d$, $\Gamma_0\subset \partial D$ be measurable, and set $\Gamma_1:=\partial D \setminus \Gamma_0$.
Further, let $x\in D$ and take any $W$ which satisfies
    \begin{equation}\label{eq:W}
        W\in C^2(D)\cap C(\overline{D}), \quad W(y)\geq W(x), \quad y\in\Gamma_0. 
    \end{equation}
Then for every $0<\alpha\neq 1$ we have
\begin{align*}
 &v_{D,\alpha/2}(x)\\
 &\leq \left(\dfrac{\alpha\vee 1}{|\alpha-1|}\right)^{\alpha\vee 1}\frac{2^{(\alpha-1)^+}}{c_{\alpha} \inf_{z\in D} \|\nabla W(z)\|^{\alpha}}
 \left\{\sup_{y\in \Gamma_1}|W(y)-W(x)|^\alpha \left(1+\sup_{y\in \Gamma_0}|W(y)-W(x)|^{(\alpha-1)^+}\right) \left[h_{D,\Gamma_0}(x)\right]^{1\wedge \alpha}\right.\\
&\qquad \left.+\sup_{y\in \Gamma_0}|W(y)-W(x)|^{(\alpha-1)^+}\frac{1}{2^{1\wedge \alpha}}\|\Delta W\|_{\infty,D}^{1\wedge \alpha} v_{D,1\wedge \alpha}(x)
+\frac{2^{(\alpha-1)^+}}{2^\alpha}\|\Delta W\|_{\infty,D}^{\alpha}v_{D,\alpha}(x)\right\}.
\end{align*}
In, particular, if $0<\alpha<1$ and using Jensen's inequality to get $v_{D,\alpha}(x)\leq \left[v_{D}(x)\right]^{\alpha}$, the above estimate recasts as
    \begin{equation}\label{eq:W_bound}
        v_{D,\alpha/2}(x)
        \leq \frac{2\sup_{y\in \Gamma_1}|W(y)-W(x)|^\alpha \left[h_{D,\Gamma_0}(x)\right]^{\alpha}+ 2^{1-\alpha}\|\Delta W\|_{\infty,D}^{\alpha} \left[v_{D}(x)\right]^{\alpha}}{(1-\alpha)c_{\alpha} \inf_{z\in D} \|\nabla W(z)\|^{\alpha}}.
    \end{equation}
\end{prop}
\begin{proof} 
Let us write down Ito's formula for $W$, namely
\begin{equation}\label{eq:Ito_W}
W\left(B^x_{\tau^x_{\partial D}}\right)-W(x)=\int_{0}^{\tau^x_{\partial D}} \nabla W(B_s^x) dB_s+\dfrac{1}{2}\displaystyle\int_0^{\tau^x_{\partial D}}\Delta W(B_s^x)ds,    
\end{equation}
so that
\begin{align*}
    \left|\int_{0}^{\tau^x_{\partial D}} \nabla W(B_s^x) dB_s\right|
    &\leq \left|W(B^x_{\tau_{\partial D}})-W(x)\right|+\frac{1}{2}\|\Delta W\|_{\infty,D}\tau^x_{\partial D}.
\end{align*}
Consequently, using the inequality
$|a+b|^\alpha\leq 2^{(\alpha-1)^+}\left(|a|^\alpha+|b|^\alpha\right)$, $\alpha\geq 0$, and also Jensen inequality, we get
\begin{align*}
\mathbb{E}\left[\left|\int_{0}^{\tau^x_{\partial D}} \nabla W(B_s^x) dB_s\right|^{\alpha}\right]
&\leq 2^{(\alpha-1)^+}\mathbb{E}\left[\left|W(B^x_{\tau_{\partial D}})-W(x)\right|^\alpha+\frac{1}{2^\alpha}\|\Delta W\|_{\infty,D}^\alpha\left(\tau^x_{\partial D}\right)^\alpha\right]\\
&= 2^{(\alpha-1)^+}\mathbb{E}\left[\left|W(B^x_{\tau_{\partial D}})-W(x)\right|^\alpha\right]+\frac{2^{(\alpha-1)^+}}{2^\alpha}\|\Delta W\|_{\infty,D}^{\alpha}v_{D,\alpha}(x).
\end{align*}
Further,
\begin{align*}
&\mathbb{E}\left[\left|W(B^x_{\tau_{\partial D}})-W(x)\right|^\alpha\right]\\
&=\mathbb{E}\left[\left|W(B^x_{\tau_{\partial D}})-W(x)\right|^\alpha;\tau_{\Gamma_1}^x<\tau_{\Gamma_0}^x\right]+\mathbb{E}\left[\left|W(B^x_{\tau_{\partial D}})-W(x)\right|^\alpha;\tau_{\Gamma_1}^x\geq\tau_{\Gamma_0}^x\right]\\
&\leq \sup_{y\in \Gamma_1}|W(y)-W(x)|^\alpha h_{D,\Gamma_0}(x)+\sup_{y\in \Gamma_0}|W(y)-W(x)|^{(\alpha-1)^+}\;\mathbb{E}\left\{\left|W(B^x_{\tau_{\partial D}})-W(x)\right|;\tau_{\Gamma_1}^x\geq\tau_{\Gamma_0}^x\right\}^{1\wedge \alpha}\\
&=\sup_{y\in \Gamma_1}|W(y)-W(x)|^\alpha h_{D,\Gamma_0}(x)\\
&\quad +\sup_{y\in \Gamma_0}|W(y)-W(x)|^{(\alpha-1)^+}\;\mathbb{E}\left\{\left[W(B^x_{\tau_{\partial D}})-W(x)\right]+\mathbb{E}\left[W(x)-W(B^x_{\tau_{\partial D}});\tau_{\Gamma_1}^x<\tau_{\Gamma_0}^x\right]\right\}^{1\wedge \alpha}\\
\intertext{where the last equality follows from the assumption that $W(y)\geq W(x), y\in \Gamma_0$.
Thus, using $|a+b|^p\leq |a|^p+|b|^p, 0<p\leq1,$ and Ito's formula \eqref{eq:Ito_W}, we can continue as}
&=\sup_{y\in \Gamma_1}|W(y)-W(x)|^\alpha \left(1+\sup_{y\in \Gamma_0}|W(y)-W(x)|^{(\alpha-1)^+}\right) \left[h_{D,\Gamma_0}(x)\right]^{1\wedge \alpha}\\
&\quad +\sup_{y\in \Gamma_0}|W(y)-W(x)|^{(\alpha-1)^+}\frac{1}{2^{1\wedge \alpha}}\|\Delta W\|_{\infty,D}^{1\wedge \alpha} v_{D,1\wedge \alpha}(x)
\end{align*}
Consequently, on the one hand we have
\begin{align*}
 &\mathbb{E}\left[\left|\int_{0}^{\tau^x_{\partial D}} \nabla W(B_s^x) dB_s\right|^{\alpha}\right]\\
 &\leq 2^{(\alpha-1)^+}\sup_{y\in \Gamma_1}|W(y)-W(x)|^\alpha \left(1+\sup_{y\in \Gamma_0}|W(y)-W(x)|^{(\alpha-1)^+}\right) \left[h_{D,\Gamma_0}(x)\right]^{1\wedge \alpha}\\
&\quad +2^{(\alpha-1)^+}\sup_{y\in \Gamma_0}|W(y)-W(x)|^{(\alpha-1)^+}\frac{1}{2^{1\wedge \alpha}}\|\Delta W\|_{\infty,D}^{1\wedge \alpha} v_{D,1\wedge \alpha}(x)
+\frac{2^{(\alpha-1)^+}}{2^\alpha}\|\Delta W\|_{\infty,D}^{\alpha}v_{D,\alpha}(x).
\end{align*}

On the other hand, by Doob's maximal inequality \Cref{t:Doob} and BDG inequality \eqref{e:BDG}, we have the lower bound
\begin{align*}
\mathbb{E}\left[\left|\int_{0}^{\tau^x_{\partial D}} \nabla W(B_s^x) dB_s\right|^{\alpha}\right]
&\geq \left(\dfrac{|\alpha-1|}{\alpha\vee 1}\right)^{\alpha\vee 1}c_{\alpha}\mathbb{E}\left[\left|\int_{0}^{\tau^x_{\partial D}} \|\nabla W\|^{2}(B_s^x) ds\right|^{{\alpha/2}}\right]\\
&\geq \left(\dfrac{|\alpha-1|}{\alpha\vee 1}\right)^{\alpha\vee 1}c_{\alpha} \inf_{z\in D} \|\nabla W(z)\|^{\alpha} v_{D,{\alpha/2}}(x).
\end{align*}
Putting together the above upper and lower bounds, we get the result.
\end{proof}

\begin{coro}\label{coro:W_special}
Let $D$ be a (bounded and open) domain in $\mathbb{R}^d$, $\Gamma_0\subset \partial D$ as in \Cref{prop:W}, and set $\Gamma_1:=\partial D \setminus \Gamma_0$. 
Further, let $x\in D$ and assume that there exists a hyperplane $\Pi$ that separates $\Gamma_0$ from $x$.
Then, for every $0<\alpha\neq 1$ we have
\begin{equation*}
 v_{D,\alpha/2}(x)
 \leq \left(\dfrac{\alpha\vee 1}{|\alpha-1|}\right)^{\alpha\vee 1}\frac{2^{(\alpha-1)^+}}{c_{\alpha}}
\sup_{y\in \Gamma_1}\|x-y\|^\alpha \left(1+\sup_{y\in \Gamma_0}\|x-y\|^{(\alpha-1)^+}\right)\left[h_{D,\Gamma_0}(x)\right]^{1\wedge \alpha}.
\end{equation*}
In particular, for every $0<\alpha<1$ we have
\begin{equation}\label{eq:v_bound_simplified}
    v_{D,\alpha/2}(x)
    \leq 
    \frac{2\sup_{y\in \Gamma_1}\|x-y\|^\alpha\left[h_{D,\Gamma_0}(x)\right]^{\alpha}}{(1-\alpha)c_\alpha}.
\end{equation}
\end{coro}
\begin{proof}
The idea is to apply \Cref{prop:W}, and to this end let $x_\Pi$ be the projection of $x$ on $\Pi$, and define
\begin{equation*}
W_n(y):= \|y-n(x-x_{\Pi})\|, \quad y\in \overline{D},
\end{equation*}  
so that
\begin{align*}
&\sup_{y\in \Gamma_1}|W_n(y)-W_n(x)|\leq \sup_{y\in \Gamma_1}\|x-y\|, \quad \sup_{y\in \Gamma_0}|W_n(y)-W_n(x)|\leq \sup_{y\in \Gamma_0}\|x-y\|, \\
&\inf_{y\in D}\|\nabla W_n(y)\|=1, \quad \mbox{and} \quad \|\Delta W_n\|_{\infty,D}=\sup_{y\in D}\frac{d-1}{\|y-n(x-x_{\Pi})\|}\mathop{\longrightarrow}\limits_n 0.
\end{align*}
Consequently, if we take $W=W_n$ in \Cref{prop:W}, relation \eqref{eq:W_bound}, and let $n$ go to infinity, we obtain the claimed estimate.
\end{proof}

\begin{coro}\label{coro:equiv:h-v_D}
Let $D$ be a (bounded and open) domain in $\mathbb{R}^d$, $\Gamma_0\subset \partial D$ as in \Cref{prop:W}, and set $\Gamma_1:=\partial D \setminus \Gamma_0$. 
Further, let $x\in D$ and assume that there exists a hyperplane $\Pi$ that separates $\Gamma_0$ from $x$.   
Then
\begin{equation*}
    \frac{{\sf dist}(x, \Gamma_1)^2}{d} h_{D,\Gamma_0}(x) \leq v_D(x)\leq \frac{8}{c_2}\sup_{y\in \Gamma_1}\|x-y\|^2 \left(1+\sup_{y\in \Gamma_0}\|x-y\|\right)h_{D_0,\Gamma_0}(x). 
\end{equation*}
\end{coro}
\begin{proof}
    The first inequality is in fact precisely \Cref{prop:general_v}, whilst the second one follows by \Cref{coro:W_special} choosing $\alpha=2$.
\end{proof}

We end this subsection with a localization result that is going to be used several times in the next sections:
\begin{prop}[Localization] \label{prop:v_a to v}
Let $D$ be a (bounded and open) domain in $\mathbb{R}^d$, $D_0\subset D$ be open, $x\in D_0$, and set
\begin{equation}
    \Gamma_0:=\partial D \cap \overline{D_0},\quad \Gamma_1:= \partial D\setminus \Gamma_0, \quad \Gamma_1':= \partial D_0\setminus \Gamma_0,
\end{equation}
as depicted in Figure $5$ below.
Then for every $\alpha>0$ we have
    \begin{equation}\label{eq:localization}
        v_{D,\alpha/2}(x)\leq  (1+2^{(\alpha/2-1)^+}) v_{D_0,\alpha/2}(x)+2^{(\alpha/2-1)^+}\sup_{z\in \Gamma_1'}v_{D,\alpha/2}(z)h_{D,\Gamma_0}(x).
    \end{equation}
\end{prop}

\begin{center}
\includegraphics[width=7cm, height=7cm]{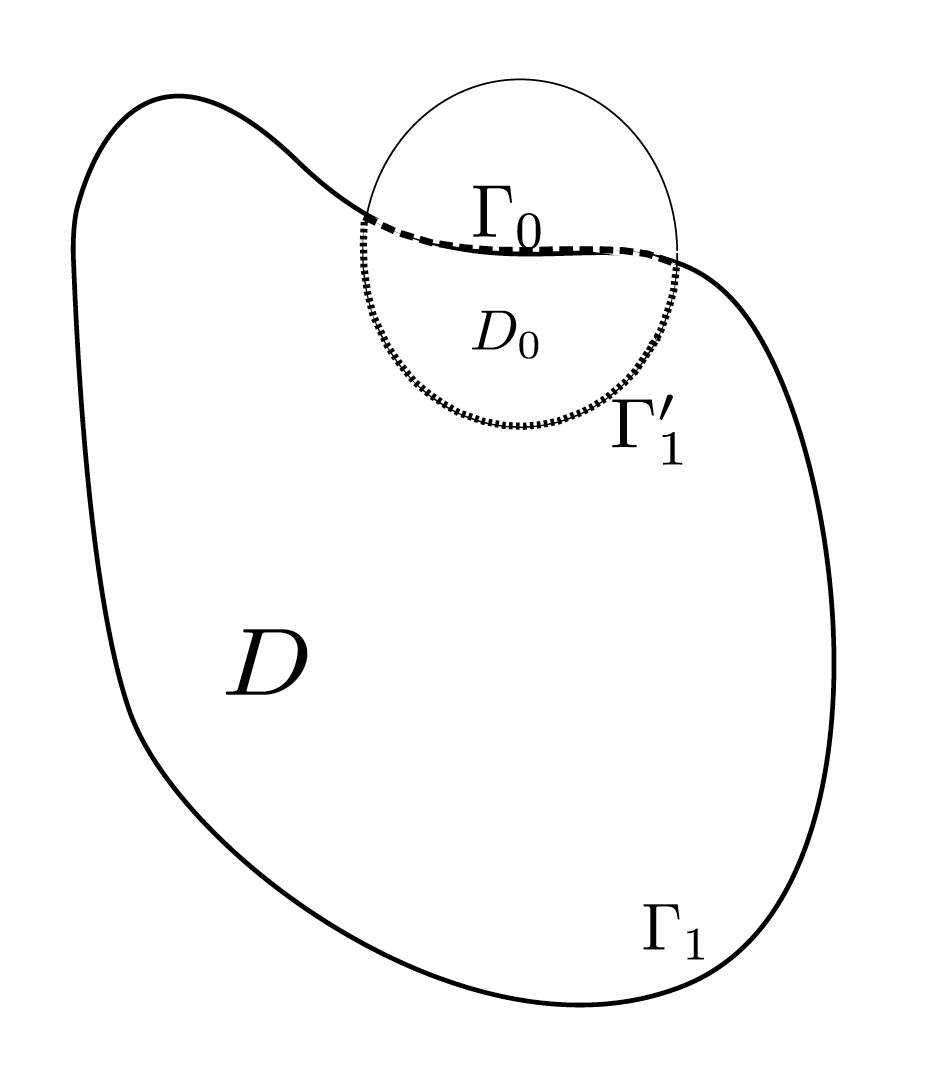}
\centerline{ Figure 5: Localizing ball at the boundary  }
\end{center}

\begin{proof}[Proof of \Cref{prop:v_a to v}]
Recalling that $x\in D_0\cap D$, we proceed as follows
\begin{align*}
    v_{D,{\alpha/2}}(x)
    &=\mathbb{E}\left[\left(\tau^x_{\partial D}\right)^{\alpha/2}\right]
    =\mathbb{E}\left[\left(\tau^x_{\partial D}\right)^{\alpha/2}; \tau^x_{\partial D}\leq \tau^x_{\Gamma_1'}\right]+\mathbb{E}\left[\left(\tau^x_{\partial D}\right)^{\alpha/2}; \tau^x_{\partial D}> \tau^x_{\Gamma_1'}\right],\\
    \intertext{so by the strong Markov property we continue with}
    v_{D,{\alpha/2}}(x)
    &\leq \mathbb{E}\left[\left(\tau^x_{\partial D_0}\right)^{\alpha/2}\right]+\mathbb{E}\left[\left(\tau^x_{\Gamma_1'}+ \tau^{B^x_{\tau^x_{\Gamma_1'}}}_{\partial D}\right)^{\alpha/2}; \tau^x_{\partial D}> \tau^x_{\Gamma_1'}\right]\\
    &\leq  \mathbb{E}\left[\left(\tau^x_{\partial D_0}\right)^{\alpha/2}\right]
    +2^{(\alpha/2-1)^+}\mathbb{E}\left[\left(\tau^x_{\Gamma_1'}\right)^{\alpha/2}; \tau^x_{\partial D}> \tau^x_{\Gamma_1'}\right]
    +2^{(\alpha/2-1)^+}\mathbb{E}\left[v_{D,{\alpha/2}}(B^x_{\tau^x_{D\cap \partial D_0}}); \tau^x_{\partial D}> \tau^x_{\Gamma_1'}\right]\\
    &\leq (1+2^{(\alpha/2-1)^+})\mathbb{E}\left[\left(\tau^x_{\partial D_0}\right)^{\alpha/2}\right] + 2^{(\alpha/2-1)^+}\sup_{z\in \Gamma_1'}v_{D,{\alpha/2}}(z) \mathbb{P}\left(\tau^x_{\partial D}> \tau^x_{\Gamma_1'}\right)\\
    &\leq(1+2^{(\alpha/2-1)^+})v_{D_0,{\alpha/2}}(x)
    +2^{(\alpha/2-1)^+}\sup_{z\in \Gamma_1'}v_{D,{\alpha/2}}(z) \mathbb{P}\left(\tau^x_{\partial D}> \tau^x_{\Gamma_1'}\right),
\end{align*}
which proves the claim since $\left[\tau^x_{\partial D}> \tau^x_{\Gamma_1'}\right]=\left[\tau^x_{\Gamma_0}> \tau^x_{\Gamma_1'}\right]$.
\end{proof}

\section{Explicit bounds for Brownian exit times in domains satisfying the exterior ball condition}\label{S:ball}
In this section we test our second approach developed in \Cref{ss:second_approach} for estimating $v_{D,\alpha/2}$, for domains $D$ satisfying an exterior ball condition. 
The main idea to employ the comparison principle \Cref{lem:barrier} with $C$ being an annular, whilst for an annular domain we apply \Cref{prop:v_a to v}  for a convenient localization,
and then \Cref{prop:W} as well as \Cref{prop:h_by_phi} for convenient choices for $W$ and $\varphi$.

We also rely on the following known estimate, and we recall that the annulus $A(a,r,R)$ is the domain defined in \eqref{eq:annulus}.
\begin{lem}[cf {\cite[Proposition 2.4]{BCILPZ2022}}]\label{lem:v_D_annulus}
If $D:=A(a,r,R)$, then
\begin{equation*}
    v_D(x)\leq {\sf dist}(x,\partial D)\frac{(R-r)R}{r}, \quad x\in D.
\end{equation*}
\end{lem}

\begin{prop}\label{thm:annulus}
Let $D:=A(a,r,R)$, $0<\alpha<1$,  $d\geq 2$. 
Then for every $R_0,x$ such that $(R+r)/2> R_0>\|x\|>r$ we have
\begin{equation}\label{eq:first}
v_{D,\alpha/2}(x) \leq {\sf dist}(x,\partial D)^\alpha\left[\frac{4}{(1-\alpha)c_\alpha}
 +\frac{(R-r)^{\alpha/2}}{(R_0-r)^{\alpha/2}}\right]\left(\frac{R_0}{r}\right)^{\alpha(d-1/2)},
\end{equation}
where the constant $c_\alpha$ depends only on $\alpha$ and is given by \Cref{t:BDG}; e.g., $c_\alpha=\frac{2-\alpha}{4-\alpha}$.
In particular, if $x\in D$ such that ${\sf dist}(x,\partial D)< \frac{r}{d}$ hence we can take $R_0:=r\left(1+\frac{1}{d}\right)$, then
\begin{equation}\label{eq:second}
v_{D,\alpha/2}(x) \leq{\sf dist}(x,\partial D)^\alpha e^{\alpha}\left[\frac{4}{(1-\alpha)c_\alpha}
 +\frac{(R-r)^{\alpha/2}}{r^{\alpha/2}}d^{\alpha/2}\right].    
\end{equation}
\end{prop}

\begin{proof}
First of all, we can assume without loss that $a=0$.
Then, as mentioned, the main idea is to apply \Cref{prop:v_a to v}  for $D:=A(0,r,R)$, $D_0:=B(0,R_0)$ for $r<R_0<R$,
and then \Cref{prop:W} for $D=D_0$ as well as \Cref{prop:h_by_phi}
with
\begin{align}
W(z):=
\begin{cases}
\frac{1}{r^{d-2}}-\frac{1}{\|z\|^{d-2}}, &d\geq 3\\
\log(\|z\|)-\log(r), &d=2
\end{cases}, \quad z\in D,\quad \mbox{ and } \quad  
\varphi(z):=
\begin{cases}
\frac{\frac{1}{r^{d-2}}-\frac{1}{\|z\|^{d-2}}}{\frac{1}{r^{d-2}}-\frac{1}{R_0^{d-2}}}, &d\geq 3\\
\frac{\log(\|z\|)-\log(r)}{\log(R_0)-\log(r)}, &d=2
\end{cases}, \quad z\in D_0,
\end{align}
Clearly, $W$ and $\varphi$ are non-negative and harmonic on their domains of definition, $W=\varphi=0$ on $\Gamma_0=\partial B(0,r)$, and $\varphi=1$ on $\Gamma_1'=\partial B(0,R_0)$.
Furthermore,
\begin{equation}
\inf_{z\in D_0}\|\nabla W(z)\|=
\begin{cases}
\frac{d-2}{R_0^{d-1}}, &d\geq 3\\
\frac{1}{R_0}, &d=2
\end{cases}, \quad \mbox{ whilst } \quad
\sup_{y\in \Gamma_1'}|W(y)-W(x)|=\begin{cases}
\frac{1}{\|x\|^{d-2}}-\frac{1}{R_0^{d-2}}, &d\geq 3\\
\log(R_0/\|x\|), &d=2
\end{cases},
\end{equation}
consequently,
\begin{equation}\label{eq:infW1}
\frac{\sup_{y\in \Gamma_1'}|W(y)-W(x)|}{\inf_{z\in D_0}\|\nabla W(z)\|}=
\begin{cases}
\frac{R_0}{d-2}\left(\left(\frac{R_0}{\|x\|}\right)^{d-2}-1\right), &d\geq 3\\
R_0\left(\log(R_0)-\log(\|x\|)\right), &d=2
\end{cases}.
\end{equation}
Now we make use of \eqref{eq:W_bound} and \eqref{eq:varphi_bound}, according to which we have
\begin{equation}
v_{D_0,\alpha/2}(x) 
\leq  \frac{2}{(1-\alpha) c_\alpha}\frac{\sup_{y\in \Gamma_1'}|W(y)-W(x)|^\alpha}{\inf_{z\in D_0} \|\nabla W(z)\|^{\alpha}}\varphi(x)^\alpha,   
\end{equation}
hence by \eqref{eq:infW1} we obtain
\begin{align}
v_{D_0,\alpha/2}(x)
&\leq 
\begin{cases}
 \frac{2}{(1-\alpha)c_\alpha}\left(\frac{R_0}{d-2}\right)^{\alpha}\left(\left(\frac{R_0}{\|x\|}\right)^{d-2}-1\right)^\alpha
 \varphi(x)^\alpha, \quad &d\geq 3\\
\frac{2}{(1-\alpha)c_\alpha}R_0^\alpha\left(\log(R_0)-\log(\|x\|)\right)^\alpha
 \varphi(x)^\alpha, \quad &d=2
\end{cases}\nonumber\\
&\leq 
\frac{2}{(1-\alpha)c_\alpha}\left(\frac{R_0}{r}\right)^{\alpha}\left(R_0-r\right)^\alpha
 \varphi(x)^\alpha, \quad d\geq 2. \label{eq:vD_0a}
\end{align}
Further, if $d\geq 3$ then
\begin{align}
    \varphi(x) &=\frac{\left(\|x\|-r\right)\left(r^{d-3}+\dots \|x\|^{d-3}\right)}{r^{d-2}\|x\|^{d-2}}\frac{r^{d-2}\|R_0\|^{d-2}}{\left(\|R_0\|-r\right)\left(r^{d-3}+\dots R_0^{d-3}\right)}\nonumber\\
    &\leq\frac{\left(\|x\|-r\right)(d-2)}{r^{d-1}}\frac{R_0^{d-1}}{\left(R_0-r\right)(d-2)}\nonumber\\
    &=\frac{\left(\|x\|-r\right)}{R_0-r}\left(\frac{R_0}{r}\right)^{d-1},\label{eq:varphi3}
\end{align}
whilst if $d=2$ we get
\begin{equation}\label{eq:varphi2}
    \varphi(x)\leq \frac{\|x\|-r}{r(\log(R_0)-\log(r))}\leq \frac{\left(\|x\|-r\right)}{R_0-r}\frac{R_0}{r}.
\end{equation}
Using \eqref{eq:varphi2} and \eqref{eq:varphi3} in \eqref{eq:vD_0a}, we get
\begin{equation}
v_{D_0,\alpha/2}(x)
\leq 
\frac{2}{(1-\alpha)c_\alpha}\left(\frac{R_0}{r}\right)^{\alpha d}
 (\|x\|-r)^\alpha, \quad d\geq 2.
\end{equation}
Further, note that since $\alpha\leq 1$, $\mathbb{P}\left(\tau^x_{\Gamma_0}>\tau^x_{\Gamma_1'}\right)$ in \eqref{eq:localization} may be trivially upper bounded by $\mathbb{P}\left(\tau^x_{\Gamma_0}>\tau^x_{\Gamma_1'}\right)^\alpha$.
In conclusion, using the above estimates in \Cref{prop:v_a to v},\eqref{eq:localization}, as well as \Cref{lem:v_D_annulus} in order to ensure that
\begin{equation*}
\sup_{y\in \Gamma_1'}v_{D,\alpha/2}(y)\leq \sup_{y\in \Gamma_1'}\left(v_{D}(y)\right)^{\alpha/2}\leq \frac{(R_0-r)^{\alpha/2}(R-r)^{\alpha/2}R^{\alpha/2}}{r^{\alpha/2}}, 
\end{equation*}
we obtain
\begin{align}
v_{D,\alpha/2}(x)
&\leq\frac{4}{(1-\alpha)c_\alpha}\left(\frac{R_0}{r}\right)^{\alpha d}
 (\|x\|-r)^\alpha+(R-r)^{\alpha/2}\frac{\left(\|x\|-r\right)^{\alpha}}{(R_0-r)^{\alpha/2}}\left(\frac{R_0}{r}\right)^{\alpha(d-1/2)}\\
 &\leq (\|x\|-r)^\alpha\left[\frac{4}{(1-\alpha)c_\alpha}\left(\frac{R_0}{r}\right)^{\alpha d}
 +\frac{(R-r)^{\alpha/2}}{(R_0-r)^{\alpha/2}}\left(\frac{R_0}{r}\right)^{\alpha(d-1/2)}\right]\\
 &\leq (\|x\|-r)^\alpha\left[\frac{4}{(1-\alpha)c_\alpha}
 +\frac{(R-r)^{\alpha/2}}{(R_0-r)^{\alpha/2}}\right]\left(\frac{R_0}{r}\right)^{\alpha(d-1/2)}
\end{align}
and \eqref{eq:first} is proved.

The estimate \eqref{eq:second} follows immediately from the first one \eqref{eq:first}, taking into account that for $R_0<r(1+1/d)$ we get
\begin{equation*}
   \left(\frac{R_0}{r}\right)^{\alpha(d-1/2)}\leq \left(1+\frac{1}{d}\right)^{\alpha(d-1/2)}\leq e^{\alpha(d-1/2)/d}\leq e^{\alpha}. 
\end{equation*}
\end{proof}

\begin{coro}\label{coro:annulus}
Let $D$ be a bounded open subset in $\mathbb{R}^d$, and $x_0=0\in \partial D$ such that there exists an annulus $A(a,r,R)$ with $0\in 
\partial B(a,r)$ and $D\subset A(a,r,R)$.
If $x\in D$ such that ${\sf dist}(x,\partial D)< \frac{r}{d}$, then
\begin{equation*}
v_{D,\alpha/2}(x) \leq{\sf dist}(x,\partial D)^\alpha e^{\alpha}\left[\frac{4}{(1-\alpha)c_\alpha}
 +\frac{(R-r)^{\alpha/2}}{r^{\alpha/2}}d^{\alpha/2}\right], \quad 0<\alpha< 1.    
\end{equation*}
Moreover,
\begin{equation}
v_D(x)\leq {\sf dist}(x,\partial D)\frac{(R-r)R}{r}, \quad x\in D.
\end{equation}
\end{coro}
\begin{proof}
   It follows from \Cref{thm:annulus}, \Cref{lem:v_D_annulus}, and the comparison principle \Cref{lem:barrier} with $C=A(a,r,R)$.
\end{proof}

\begin{rem}\label{rem:uniform_ball_far}
    Note that if $x\in D$ satisfies ${\sf dist}(x,\partial d)\geq r/d$, then by \Cref{lem:v_alpha_uniform} we get
    \begin{equation*}
        v_{D,\alpha/2}(x)
        \leq \frac{{\sf diam}(D)^\alpha}{d^{\alpha/2}}
        \leq {\sf diam}(D)^\alpha \frac{d^{\alpha/2}}{r^\alpha} {\sf dist}(x,\partial d)^\alpha, \quad \alpha>0.
    \end{equation*}
\end{rem}

\begin{proof}[\bf Proof of \Cref{thm:u_g_ball}] \label{proof:thm:u_g_ball}
(i). Both estimates \eqref{eq:u_g_holder_x0_sphere} and \eqref{eq:u_g_holder_global_sphere} follow directly from the general estimate \Cref{prop:u_g} together with \Cref{coro:annulus} and \Cref{rem:uniform_ball_far}.

(ii). The estimate \eqref{eq:u_f_holder} follows directly from the general estimate \Cref{prop:u_f} and \Cref{coro:annulus}.
Finally, to show \eqref{eq:C in O} recall that
\begin{equation*}
   C(d,D,\gamma,q)
   =\left[\frac{(d-2)(\gamma-q)}{(d(d-2)\omega_d)^{\frac{q}{\gamma-q}}(2\gamma-qd)}\right]^{\frac{\gamma-q}{\gamma q}}{\sf diam}(D)^{\frac{2\gamma-qd}{pq}}.
\end{equation*}
Also, note that by Stirling approximation,
\begin{equation*}
   d(d-2)\omega_d
   \sim 
   {\frac {d(d-2)}{\sqrt {d\pi }}}\left({\frac {2\pi e}{d}}\right)^{d/2}<<1 \quad \mbox{as } d>>1.
\end{equation*}
Moreover, 
\begin{equation*}
    \frac{q}{\gamma-q}<\frac{q}{\frac{dq}{2}-q}=\frac{2}{d-2}
\end{equation*}
Consequently
\begin{align*}
    \frac{(d-2)(\gamma-q)}{(d(d-2)\omega_d)^{\frac{q}{\gamma-q}}(2\gamma-qd)}
    \lesssim
    \frac{(d-2)(\gamma-q)}{\left({\frac {d(d-2)}{\sqrt {d\pi }}}\left({\frac {2\pi e}{d}}\right)^{d/2}\right)^{\frac{2}{d-2}}(2\gamma-qd)}
    \lesssim\frac{d^2(\gamma-q)}{2\gamma-qd}, \quad d>>1.
\end{align*}
Thus, since $\frac{\gamma-q}{\gamma q}\leq \frac{1}{q}$,
\begin{equation*}
   C(d,D,\gamma,q)
   \lesssim \left[\frac{d^2(\gamma-q)}{2\gamma-qd}\right]^{\frac{1}{q}}{\sf diam}(D)^{\frac{2\gamma-qd}{pq}}, \quad \mbox{which is } \eqref{eq:C in O}.
\end{equation*}
\end{proof}

\section{Explicit bounds for Brownian exit times in domains satisfying an exterior cone/wedge condition}\label{S:cone-wedge}

In this section we shall derive explicit bounds for planar domains that satisfy an exterior cone condition, and for three-dimensional domains that satisfy an exterior wedge condition, in separate subsections. 
To this end, as in the case of domains satisfying an exterior sphere condition, we shall rely on the second approach developed in \Cref{ss:second_approach} for estimating $v_{D,\alpha/2}$.

\subsection{The exterior cone condition: $d=2$}\label{ss:cone_2d}

Let us first set the notations for the polar coordinates $(l,\theta)\in[0,\infty)\times [0,2\pi)$:
\begin{equation}\label{eq:polar_coordinates}
l:=\|x\|, 
\quad 
[0,2\pi)\ni\theta:=
\begin{cases}
    \arccos{(x_1/\|x\|)}, \quad &x_2\geq 0\\
    2\pi-\arccos{(x_1/\|x\|)}, \quad &x_2< 0
\end{cases},
\quad x\in\mathbb{R}^2.
\end{equation}

By a straightforward computation in polar coordinates, we get the following:
\begin{prop}\label{prop:w_d=2}  
Set $\tilde{\omega}:=\dfrac{\pi}{2(\pi-\omega)}, \omega \in (0,\pi)$,
and define for $x\in \mathbb{R}^2$
\begin{equation*}
w:\mathbb{R}^2\longrightarrow \mathbb{R}, \quad w(x)= l^{\tilde{\omega}}\cos((\pi-\theta)\tilde{\omega}).     
\end{equation*}
Then $w$ is harmonic on $\mathbb{R}^2$.
\end{prop}

Further, by $B_c(\omega,r)$ we denote the ball of radius $r$ centered at the origin from which we remove the cone $\mathcal{C}(\omega,r)$, namely
\begin{equation}\label{eq:B_c_2d}
    B_c(\omega,r):=B(0,r)\setminus \mathcal{C}(\omega,r),
\end{equation}
where recall that $\mathcal{C}(\omega,r)$ is given by \eqref{eq:C_2d}.

\begin{coro}\label{coro:cone-2D-main}
Let $D=B_c(\omega,r)\subset \mathbb{R}^2$ and $\tilde{\omega}$ be as in \Cref{prop:w_d=2}.
Further, let $\Gamma_0:=\partial C(\omega,r)\setminus \partial B(0,r)$, and $\Gamma_1:=\partial D\setminus \Gamma_0$.
Then, for every $x\in D$ such that $\|x\| < r e^{-1/\tilde{\omega}}$, the following assertions hold:

\begin{enumerate}
    \item[(i)] For $\omega\in (0,\pi)$ we have
    \begin{equation}
    h_{D,\Gamma_0}(x)\leq \tilde{\omega}\left(\frac{\|x\|}{r}\right)^{\tilde{\omega}}\; \log{\left(\frac{r}{\|x\|}\right)}. 
    \end{equation}
    \item[(ii)] For every $0<\alpha\neq 1$ and $\omega\in (0,\pi/2]$ we have
    \begin{equation}
        v_{D,{\alpha/2}}(x)
       \leq C_2(\alpha,r)\left[h_{D,\Gamma_0}(x)\right]^{1\wedge \alpha},
    \end{equation}
\end{enumerate}
where
\begin{equation}\label{eq:C_2}
    C_2(\alpha,r):=\left(\dfrac{\alpha\vee 1}{|\alpha-1|}\right)^{\alpha\vee 1}\frac{2^{(\alpha-1)^+}}{c_{\alpha}}
        (2r)^\alpha \left(1+(2r)^{(\alpha-1)^+}\right).
\end{equation}
\end{coro}
\begin{proof}
Let $\varepsilon\in(0,\omega)$ and set $\omega_\varepsilon:=\omega-\varepsilon, \quad \tilde{\omega}_\varepsilon:=\dfrac{\pi}{2(\pi-\omega_\varepsilon)}$.
Then, by \Cref{prop:w_d=2}, the function
\begin{equation*}
w_\varepsilon:\mathbb{R}^2\longrightarrow \mathbb{R}, \quad w_\varepsilon(x)= l^{\tilde{\omega}_\varepsilon}\cos((\pi-\theta)\tilde{\omega}_\varepsilon),     
\end{equation*}
is harmonic on $\mathbb{R}^2$.
Note that for $\theta\in [\omega,2\pi-\omega)$
\begin{equation}
    -\frac{\pi}{2}\leq (\pi-\theta)\tilde{\omega}_\varepsilon\leq \frac{\pi}{2},
\end{equation}
hence $w_\varepsilon\geq 0$ on $\mathcal{C}(\omega,r)$.
Furthermore, we have
\begin{equation}
    w_\varepsilon(x)\geq r^{\tilde{\omega}_\varepsilon}\cos((\pi-\omega)\tilde{\omega}_\varepsilon) \quad \mbox{on } \partial \mathcal{C}(\omega,r).
\end{equation}
Therefore, if we set
\begin{equation}
    \varphi_\varepsilon(y):=\frac{w_\varepsilon(y)}{r^{\tilde{\omega}_\varepsilon}\cos((\pi-\omega)\tilde{\omega}_\varepsilon)}, \quad y\in \mathcal{C}(\omega,r),
\end{equation}
then $\varphi_\varepsilon\geq 0$ is harmonic on $\mathcal{C}(\omega,r)$ and $\varphi_\varepsilon\geq 1$ on $\Gamma_1$.
Furthermore, note that since $0\leq (\pi-\theta)\tilde{\omega}_\varepsilon\leq \pi/2$ and $\cos(y)\geq 1-\frac{2}{\pi}y$ for $0\leq y\leq \pi/2$, we have
\begin{align*}
    \varphi_\varepsilon(x)
    =\left(\frac{\|x\|}{r}\right)^{\tilde{\omega}_\varepsilon}\frac{\cos((\pi-\theta)\tilde{\omega}_\varepsilon)}{\cos((\pi-\omega)\tilde{\omega}_\varepsilon)}
    &\leq 
    \left(\frac{\|x\|}{r}\right)^{\tilde{\omega}_\varepsilon}\frac{\cos((\pi-\theta)\tilde{\omega}_\varepsilon)}{1-\frac{2(\pi-\omega)}{\pi}\tilde{\omega}_\varepsilon}\\
    &\leq
     \left(\frac{\|x\|}{r}\right)^{\tilde{\omega}_\varepsilon}\frac{1}{1-\frac{2(\pi-\omega)}{\pi}\tilde{\omega}_\varepsilon}, \quad x\in \mathcal{C}(2,\omega,r).
\end{align*}
Since the above estimate holds for every $\varepsilon\in(0,\omega)$, we can optimize it over $\tilde{\omega}_\varepsilon\in (\frac{1}{2},\tilde{\omega})$.
To this end, let us denote $a:=\|x\|/r$ and $q:=\tilde{\omega}_\varepsilon$, so that we have to find
\begin{equation}
    \min_{q\in (1/2,\tilde{\omega})} h(q), \quad h(q):=a^q\frac{1}{1-\frac{2(\pi-\omega)}{\pi}q}>0.
\end{equation}
Solving $0=h'(q)=\log(a)h(q)+\frac{\frac{2(\pi-\omega)}{\pi}}{1-\frac{2(\pi-\omega)}{\pi}q}h(q)$ we easily obtain the critical value
\begin{equation}
    \tilde{\omega}_\varepsilon=q=\tilde{\omega}+\frac{1}{\log{a}}.
\end{equation}
The hypothesis $\|x\| < r e^{-\frac{2(\pi-\omega)}{\omega}}$ was precisely chosen such that the above optimal choice of $\tilde{\omega}_\varepsilon$ belongs to the restriction interval $(\frac{1}{2},\tilde{\omega})$.
Consequently, for the corresponding optimal choice $\varepsilon^\ast$ of $\varepsilon$ we get
\begin{align}
   h_{D,\Gamma_0}(x)\leq \varphi_{\varepsilon^\ast}(x)
    \leq
    \tilde{\omega}\left(\frac{\|x\|}{r}\right)^{\tilde{\omega}}\; \log{\left(\frac{r}{\|x\|}\right)},
\end{align}
which proves (i).
Furthermore, (ii) follows from (i) and \Cref{coro:W_special}.
\end{proof}

\begin{coro} \label{coro:D-cone-2D}
Let $D$ be a bounded open subset in $\mathbb{R}^2$, and $x_0\in \partial D$ such that there exists a cone $\mathcal{C}(\omega,r)$ with vertex $x_0$, radius $r$, and angle $\omega\in (0,\pi/2]$, such that $\mathcal{C}(\omega,r)\subset \mathbb{R}^d\setminus D$.
Then for every $x\in B(x_0,re^{-1/\tilde{\omega}})\cap D$ and $0<\alpha\neq 1$ we have
\begin{equation}\label{eq:estimate-cone-2D}
    v_{D,\alpha/2}(x)\leq C_3\; \left(\frac{\|x-x_0\|}{r}\right)^{(1\wedge\alpha)\tilde{\omega}}\left[\log{\left(\frac{r}{\|x-x_0\|}\right)}\right]^{1\wedge\alpha} +C_4\;\left(\frac{\|x-x_0\|}{r}\right)^{\tilde{\omega}}\; \log{\left(\frac{r}{\|x-x_0\|}\right)},
\end{equation}
where
\begin{equation*}
    C_3:=2C_2(\alpha,r)\tilde{\omega}^\alpha,
    \quad
    C_4:=C_1(\alpha/2,2)\left({\sf diam}(D)+r\right)^{\alpha}\tilde{\omega},
\end{equation*}
where $C_1(\alpha/2,2)$ and $C_2(\alpha,r)$ are given in \eqref{lem:v_alpha_uniform}, \eqref{eq:C_2}, respectively.
\end{coro}

\begin{rem}\label{rem:uniform_cone2D_far}
    Let $x\in D$ such that $\|x-x_0\|\geq re^{-1/\tilde{\omega}}$. 
    Then by \Cref{lem:v_alpha_uniform} we get
    \begin{equation*}
        v_{D,\alpha/2}(x)
        \leq C_1(\alpha/2,d){\sf diam}(D)^\alpha
        \leq C_1(\alpha/2,d){\sf diam}(D)^\alpha e\left(\frac{\|x-x_0\|}{r}\right)^{\tilde{\omega}},\quad \alpha>0,
    \end{equation*}
    where $C_1(\alpha/2,d)$ is given in \eqref{eq:v_alpha_uniform}.
\end{rem}

\begin{proof}[Proof of \Cref{coro:D-cone-2D}]
Recall that $B_c(x_0,\omega,r)=B(x_0,r)\setminus \overline{\mathcal{C}(\omega,r)}$. 
Also, note that by \Cref{lem:barrier} it is sufficient to consider that $D$ is a ball of diameter ${\sf diam}(D)+r$ which contains $B_c(x_0,\omega,r)$ and $\overline{\mathcal{C}(\omega,r)}\in D^c$.
Then, we can apply
\Cref{prop:v_a to v} for $D_0:= B_c(x_0,\omega,r)$ and $\Gamma_0:=\partial C(\omega,r)\setminus\partial B(x_0,r)$ to deduce that
\begin{equation}
    v_{D,\alpha/2}(x)\leq  2 v_{B_c(x_0,\omega,r),\alpha/2}(x)+\sup_{z\in \Gamma_1'}v_{D,\alpha/2}(z)h_{D_0,\Gamma_0}(x),
\end{equation}
where $\Gamma_0=\partial C(\omega,r)\setminus \partial B(x_0,r)$ and $\Gamma_1=\partial B_c(x_0,\omega,r)\setminus\Gamma_0$, and also note that by \Cref{lem:v_alpha_uniform} we get 
\begin{equation*}
\sup_{z\in D}v_{D,\alpha/2}(z)
\leq C_1(\alpha/2,2)\left({\sf diam}(D)+r\right)^{\alpha}.
\end{equation*}
Now, the claimed estimates \eqref{eq:estimate-cone-2D} follows from \Cref{coro:cone-2D-main}, assuming without loss of generality that $x_0=0$.
\end{proof}

\begin{proof}[\bf Proof of \Cref{thm:u_cone_2D}] \label{proof:u_cone_2D}
(i). Both estimates \eqref{eq:u_g_holder_x0} and \eqref{eq:u_g_holder_global} follow from the general estimate \Cref{prop:u_g}, (ii), and \Cref{coro:D-cone-2D} together with \Cref{rem:uniform_cone2D_far}.

(ii). Both \eqref{eq:u_f_infty_2d_cone} and \eqref{eq:u_f_gamma_2d_cone} follow from the general estimate \Cref{prop:u_f}, \Cref{coro:D-cone-2D} and \Cref{rem:uniform_cone2D_far} with $\alpha=2$.
\end{proof}

\subsection{The exterior wedge condition: $d=3$}\label{ss:wedge}
Let $d=3$, $r,l>0$, $\omega \in [0,\pi)$, 
and consider the {\it wedge} $\mathcal{W}(\omega,r,l)$ and $B_w(\omega,r,l)$ the cylinder from which we remove the closure of $\mathcal{W}$, given by
\begin{equation}\label{eq:wedge_2d}
    \mathcal{W}(\omega,r,l)=\mathcal{C}(\omega,r)\times (-l/2,l/2), 
    \quad
    B_{w}(\omega,r,l):=B_c(\omega,r)\times (-l/2,l/2),
\end{equation}
where recall that $\mathcal{C}(\omega,r)$ and $B_c(\omega,r)$ are the $2$-dimensional sets given by \eqref{eq:C_2d} and \eqref{eq:B_c_2d}. 
We have the following:
\begin{coro}\label{coro:wedge-3D-main}
Let $d=3$, $r,l>0$, $\omega \in [0,\pi)$, and $D=B_w(\omega,r,l)$.
Further, assume that $0\in \partial B_c(\omega,r)\times (-l/2,l/2)\subset \partial D$.
Further, let $x\in D$ such that $\|x_{(1)}\|\leq r e^{-1/\tilde{\omega}}$ where $x=(x_{(1)},x_{(2)})\in \mathbb{R}^2\times \mathbb{R}$ whilst $\tilde{\omega}$ is given in \Cref{prop:w_d=2}.
Then the following assertions hold:
\begin{enumerate}
    \item[(i)] Let $\Gamma_0:=\partial \mathcal{W}(\omega,r,l)\cap\partial B_w(\omega,r,l)\subset \partial D$.
    Then
    \begin{equation}\label{eq:h_wedge_3D}
        h_{D,\Gamma_0}(x)\leq \tilde{\omega}\left(\frac{\|x^{(1)}\|}{r}\right)^{\tilde{\omega}}\log\left(\frac{r}{\|x^{(1)}\|}\right).
    \end{equation}
    \item[(ii)] If $0<\alpha\neq 1$ and $\omega\in (0,\pi/2]$ then
    \begin{equation}\label{eq:v_wedge_3D}
    v_{D,\alpha/2}(x)
    \leq
    C_2(\alpha,r)\tilde{\omega}^{1\wedge \alpha}\left(\frac{\|x_{(1)}\|}{r}\right)^{(1\wedge\alpha)\tilde{\omega}}\left[\log{\left(\frac{r}{\|x_{(1)}\|}\right)}\right]^{1\wedge\alpha},
    \end{equation}
    where $C_2(\alpha,r)$ is given in \eqref{eq:C_2}.
\end{enumerate}
\end{coro}
\begin{proof}
To prove \eqref{eq:h_wedge_3D} we use first \Cref{lem:products}, \eqref{eq:products_h} to deduce
\begin{equation*}
    h_{D,\Gamma_0}(x)\leq h_{B_c(\omega,r)}(x_{(1)}), \quad x=(x_{(1)},x_{(2)})\in D.
\end{equation*}
Then, \eqref{eq:h_wedge_3D} follows directly from \Cref{coro:cone-2D-main}, (i).

To prove \eqref{eq:v_wedge_3D} we use the same \Cref{lem:products} to deduce that
\begin{equation*}
    v_{D,\alpha/2}(x)\leq \min\left[v_{B_c(\omega,r),\alpha/2}(x_{(1)}),v_{(-l/2,l/2),\alpha/2}(x_{(2)})\right], \quad \mbox{where } x=(x_{(1)},x_{(2)})\in \mathbb{R}^2\times 
    \mathbb{R}.
\end{equation*}
Now, since $\|x_{(2)}\|\leq r e^{-1/\tilde{\omega}}$, estimate \eqref{eq:v_wedge_3D} follows by \Cref{coro:cone-2D-main}, (ii).
\end{proof}

\begin{coro} \label{coro:D-wedge-3D}
Let $D$ be a bounded open subset in $\mathbb{R}^3$, and $x_0=0\in \partial D$ such that there exists a wedge $\mathcal{W}(\omega,r,l)$ of radius $r$, length $l$, angle $\omega\in [0,\pi/2]$, and $\mathcal{W}(\omega,r,l)\subset \mathbb{R}^3\setminus D$.
Then for every $x\in D$ such that $\|x_{(1)}\|\leq re^{-1/\tilde{\omega}}$ and $|x_{(2)}|\leq l$, where $x=\left(x_{(1)},x_{(2)}\right)\in \mathbb{R}^2\times \mathbb{R}$, and $0<\alpha\neq 1$, we have
\begin{equation}\label{eq:estimate-wedge-3D}
    v_{D,\alpha/2}(x)\leq C_3\; \left(\frac{\|x_{(1)}\|}{r}\right)^{(1\wedge\alpha)\tilde{\omega}}\left[\log{\left(\frac{r}{\|x_{(1)}\|}\right)}\right]^{1\wedge\alpha} +C_8\;\left(\frac{\|x_{(1)}\|}{r}\right)^{\tilde{\omega}}\; \log{\left(\frac{r}{\|x_{(1)}\|}\right)},
\end{equation}
where $C_3$ is given in \Cref{coro:D-cone-2D} whilst
\begin{equation}
C_8:=C_1(\alpha/2,3)\left({\sf diam}(D)+r+l\right)^{\alpha}\tilde{\omega}, 
\quad 
\mbox{with } C_1(\alpha/2,3) \mbox{ given in } \eqref{lem:v_alpha_uniform}.
\end{equation}
\end{coro}
\begin{proof}
We proceed as in the proof of \Cref{coro:D-cone-2D}. 
By \Cref{lem:barrier} it is sufficient to replace $D$ by $\tilde{D}$ which is a ball of diameter ${\sf diam}(D)+r+l$ which contains $B_w(\omega,r,l)$, whilst $\overline{\mathcal{W}(\omega,r,l)}\in \tilde{D}^c$.
Then, we can apply
\Cref{prop:v_a to v} for $D_0:= B_w(\omega,r,l)$ and $\Gamma_0:=\partial \mathcal{W}(\omega,r,l)\cap\partial B_w(\omega,r,l)$ to deduce that
\begin{equation}
    v_{\tilde{D},\alpha/2}(x)\leq  2 v_{D_0,\alpha/2}(x)+\sup_{z\in \tilde{D}}v_{\tilde{D},\alpha/2}(z)h_{D_0,\Gamma_0}(x).
\end{equation}
By \Cref{lem:v_alpha_uniform} we have 
\begin{equation*}
\sup_{z\in \tilde{D}}v_{\tilde{D},\alpha/2}(z)
\leq C_1(\alpha/2,3)\left({\sf diam}(D)+r+l\right)^{\alpha}.
\end{equation*}
Now, the claimed estimates \eqref{eq:estimate-wedge-3D} follows from \Cref{coro:wedge-3D-main}.
\end{proof}

\begin{rem}\label{rem:||_wedge3D}
    Note that in \Cref{coro:wedge-3D-main}, \Cref{coro:D-wedge-3D}, and \Cref{thm:u_wedge_3D}, the statements remain valid if we replace $x_{(1)}$ by $x$, due to the fact that $\|z_{(1)}\|\leq \|z\|$ and that the function $[0,e^{-1/\tilde{\omega}})\ni a\mapsto a^{\tilde{\omega}}\log(1/a)$ is non-decreasing.
\end{rem}

\begin{rem}\label{rem:uniform_wedge3D_far}
    Let $x\in D$ such that $\|x_{(1)}\|> re^{-1/\tilde{\omega}}$ or $|x_{(2)}|> l$, where $x=\left(x_{(1)},x_{(2)}\right)\in \mathbb{R}^2\times \mathbb{R}$, and $0<\alpha\neq 1$. 
    Then $r\left[\left(re^{-1/\tilde{\omega}}\right)\wedge l\right]^{-1}\frac{\|x\|}{r}\geq 1$, hence by \Cref{lem:v_alpha_uniform} we get
    \begin{equation*}
        v_{D,\alpha/2}(x)
        \leq C_1(\alpha/2,d){\sf diam}(D)^\alpha
        \leq C_1(\alpha/2,d){\sf diam}(D)^\alpha \left[e\vee \left(\frac{r}{l}\right)^{\tilde{\omega}}\right]\left(\frac{\|x\|}{r}\right)^{\tilde{\omega}},\quad \alpha>0,
    \end{equation*}
    where $C_1(\alpha/2,d)$ is given in \eqref{eq:v_alpha_uniform}.
\end{rem}

\begin{proof}[\bf Proof of \Cref{thm:u_wedge_3D}] \label{proof:u_wedge_3D}
(i). Both estimates \eqref{eq:u_g_holder_x0} and \eqref{eq:u_g_holder_global} are obtained from the general estimate \Cref{prop:u_g} and \Cref{coro:D-wedge-3D} together with \Cref{rem:||_wedge3D} and \Cref{rem:uniform_wedge3D_far}.

(ii). Both \eqref{eq:u_f_infty_2d_cone} and \eqref{eq:u_f_gamma_2d_cone} follow from the general estimate \Cref{prop:u_f} and \Cref{coro:D-wedge-3D} together with \Cref{rem:||_wedge3D} and \Cref{rem:uniform_wedge3D_far}, taking $\alpha=2$.
\end{proof}

\subsection{The exterior cone condition: $d\geq 3$}\label{ss:cone_3d}
Throughout this section we assume that $d\geq 3$, and we aim at deriving explicit estimates for $v_{D,\alpha/2}$ for domains $D$ satisfying an exterior cone condition.
For such domains, it is not clear to us whether the approach in \Cref{ss:second_approach} can be easily applied in order to derive explicit estimates similar to those obtained in the previous two sections. 
Instead, we shall rely on the approach developed in \Cref{subsec:reversedoubling}.

Further, we denote by $A_h$ the aria of a hyperspherical cap of angle $\omega\in [0,\pi/2]$ of a $d$-dimensional sphere with radius $r>0$:
\begin{align}
    {\sf Cap}_\omega(r)
    &:=\left\{x=(x_{(1)},x_{(2)})\in \mathbb{R}\times \mathbb{R}^{d-1}: \|x\|=r, x_{(1)}\geq r\cos{\omega}\right\}\subset \partial B(0,r)\\
    A_\omega(r)&:=\sigma\left({\sf Cap}_\omega(r)\right),
\end{align}
where $B(0,r)$ denotes the $d$-dimensional ball of radius $r$ centered at the origin, whilst $\sigma$ is the surface measure on $S(0,r)$.

According to \cite{li2010} we have
\begin{equation}\label{eq:A_cap}
    A_\omega(r)=\frac{1}{2}\sigma(S(0,r))I_{\left(\sin{\omega}\right)^2}\left(\frac{d-1}{2},\frac{1}{2}\right),
\end{equation}
where $I_{x}(a,b)$ is the {\it regularized incomplete beta function}
\begin{equation}
    I_x(a,b):=\frac{B(x;a,b)}{B(a,b)}, \quad x, a,b>0,
\end{equation}
whilst $B(a,b), B(x;a,b)$ are the {\it beta function} and the {\it incomplete beta function}, respectively, given by
\begin{equation}
    B(x;a,b)= \int_0^x t^{a-1}(1-t)^{b-1} dt,\quad B(a,b)=B(1;a,b), \quad x, a,b>0.
\end{equation}

\begin{lem}\label{lem:I_lower_bound}
    If $x\geq 0$ and $d\geq 3$ we have
    \begin{equation}
        I_x\left(\frac{d-1}{2},\frac{1}{2}\right)
    \geq\frac{2}{\sqrt{\pi}(d-1)}\left(\frac{d}{2}-\frac{3}{4}\right)^{\frac{1}{2}}x^{\frac{d-1}{2}}.
    \end{equation}
    In particular, for $\omega\in [0,\pi/2]$ and $r>0$ we have
    \begin{equation}
        \frac{A_\omega(r)}{\sigma(\partial B(0,r))}
        \geq \frac{1}{\sqrt{\pi}(d-1)}\left(\frac{d}{2}-\frac{3}{4}\right)^{\frac{1}{2}}(\sin{\omega})^{d-1}. 
    \end{equation}
\end{lem}
\begin{proof}
First of all 
    \begin{equation}
        B\left(x;\frac{d-1}{2},\frac{1}{2}\right)=\int_0^x t^{\frac{d-3}{2}}\frac{1}{(1-t)^\frac{1}{2}}\;dt\geq\int_0^x t^{\frac{d-3}{2}}\;dt=\frac{2}{d-1}x^{\frac{d-1}{2}}, \quad x>0.
    \end{equation}
Further, the identity
\begin{equation}
    B(a,b)=\frac{\Gamma(a)\Gamma(b)}{\Gamma(a+b)}, \quad a,b>0
\end{equation}
is well known (see e.g. \cite{davis1972gamma}), where $\Gamma$ denotes the {\it Gamma function},
\begin{equation}
    \Gamma(a)=\int_0^\infty t^{a-1}e^{-t}\; dt, \quad a>0.
\end{equation}
Consequently,
\begin{equation}
    I_x\left(\frac{d-1}{2},\frac{1}{2}\right)
    \geq\frac{2}{d-1}x^{\frac{d-1}{2}}\frac{\Gamma(\frac{d}{2})}{\Gamma(\frac{d-1}{2})\Gamma(\frac{1}{2})}
    =\frac{2}{\sqrt{\pi}(d-1)}x^{\frac{d-1}{2}}\frac{\Gamma(\frac{d}{2})}{\Gamma(\frac{d-1}{2})}
    , \quad x>0,
\end{equation}
where we have used that $\Gamma(\frac{1}{2})=\sqrt{\pi}$.
Next, to lower bound the ratio of the Gamma functions, we rely on
the following inequality proved in \cite{kershaw1983}, which in fact an extension of Gautschi's inequality \cite{gautschi1959}, and according to which
\begin{equation}
    \frac{\Gamma(x+1)}{\Gamma(x+s)}>\left(x+\frac{s}{2}\right)^{1-s}, \quad \forall s\in (0,1), x>0.
\end{equation}
In particular,
\begin{equation}
\frac{\Gamma(\frac{d}{2})}{\Gamma(\frac{d-1}{2})}=\frac{\Gamma(\frac{d-2}{2}+1)}{\Gamma(\frac{d-2}{2}+\frac{1}{2})} >\left(\frac{d}{2}-\frac{3}{4}\right)^{\frac{1}{2}}, \quad d\geq 3,   
\end{equation}
which further gives
\begin{equation}
    I_x\left(\frac{d-1}{2},\frac{1}{2}\right)
    \geq\frac{2}{\sqrt{\pi}(d-1)}\left(\frac{d}{2}-\frac{3}{4}\right)^{\frac{1}{2}}x^{\frac{d-1}{2}}
    , \quad x>0, d\geq 3.
\end{equation}
\end{proof}

\begin{coro}\label{coro:delta_omega}
Let $D$ be a bounded open subset in $\mathbb{R}^d$, $d\geq 3$, and $x_0=0\in \partial D$ such that there exists a cone $\mathcal{C}(\omega,r_0)$ with vertex $0$, radius $r_0$, and angle $\omega\in (0,\pi/2]$, such that $\mathcal{C}(\omega,r_0)\subset \mathbb{R}^d\setminus D$.
Further, let $\delta$ be given by \eqref{eq:delta}.
Then
\begin{equation}\label{eq:delta_omega}
    \delta
    \leq \delta_\omega:=1-\frac{3}{4\sigma(S(0,1))}\int_{{\sf Cap}_\omega(1)} \frac{1}{ \|y_0-\xi\|^d}\;\sigma(d\xi)
    <1, \quad d\geq 3,
\end{equation}
where $y_0=\left(-\frac{1}{2},0,\dots0\right)\in \mathbb{R}^d$.
Moreover, we have the estimate
\begin{align}
    \delta_\omega
    &\leq 1-\frac{3}{8}\left(\frac{2}{3}\right)^{d-1}I_{\left(\sin{\omega}\right)^2}\left(\frac{d-1}{2},\frac{1}{2}\right)\\
    &\leq 1-\frac{3}{4}\left(\frac{2}{3}\right)^{d-1}\frac{1}{\sqrt{\pi}(d-1)}\left(\frac{d}{2}-\frac{3}{4}\right)^{\frac{1}{2}}(\sin{\omega})^{d-1}, \quad d\geq 3.\label{eq:delta_omega_estimate}
\end{align}
Also, for $d=3$ we have $I_{\left(\sin{\omega}\right)^2}\left(\frac{d-1}{2},\frac{1}{2}\right)=1-\cos{\omega}$.
\end{coro}
\begin{proof}
    First of all, by \Cref{rem:delta<delta_hat} we have $\delta\leq \hat{\delta}$, where $\hat{\delta}$ is given by \eqref{eq:delta_hat}.
    Furthermore, by the assumed exterior cone condition, one can easily see that an upper bound for $\hat{\delta}$ may be obtained replacing $D$ by $B_c(\omega,r_0)$ in the definition of $\hat{\delta}$, where recall that $B_c(\omega,r_0)$ is the ball of radius $r_0$ centered in the origin, from which we remove the cone $\mathcal{C}(\omega,r_0)$.
    Thus, to get an upper bound, we further assume that $D=B_c(\omega,r_0)$.
    In this case, we can estimate $\hat{h}_r$ from \eqref{eq:delta_hat} explicitly, for $r\leq r_0/2$ .
    As a matter of fact, taking into account that we are in the situation given by $D=B_c(\omega,r_0)$, by rescaling \eqref{eq:PDE_h} one can see that $\hat{h}_r$ does not depend on $r$.
    In particular, we can take $r=1/2$ and $r_0=1$, so that we have
    \begin{equation*}
        \delta\leq \hat{\delta}=\sup\limits_{\|y\|=1/2}\hat{h}_{1/2}(y).
    \end{equation*}
    Then, taking $\|y\|=1/2$, we have
    \begin{align}\label{eq:1-hat_h}
    1-\hat{h}_{1/2}(y)
    &=\int_{D^c\cap S(0,1)} \frac{1-\|y\|^2}{ \sigma(S(0,1))\|y-\xi\|^d}\;\sigma(d\xi)\\
    &=\frac{3}{4\sigma(S(0,1))}\int_{{\sf Cap}_\omega(1)} \frac{1}{ \|y-\xi\|^d}\;\sigma(d\xi),\\
    &\geq\frac{3}{4\sigma(S(0,1))}\int_{{\sf Cap}_\omega(1)} \frac{1}{ \|y_0-\xi\|^d}\;\sigma(d\xi)\\
    &=1-\delta_\omega,
    \end{align}
    which proves \eqref{eq:delta_omega}.
    Further, we notice that  for $\|\xi\|=1$ we have $1/2\le \|y-\xi\|\le 3/2$, and thus we get 
    \begin{align}\label{eq:1-hat_h_estimate}
    1-\hat{h}_{1/2}(y)
    &=\frac{3}{4\sigma(S(0,1))}\int_{{\sf Cap}_\omega(1)} \frac{1}{ \|y-\xi\|^d}\;\sigma(d\xi)\\
    &\geq\frac{3}{4}(2/3)^{d-1} \frac{A_\omega(1)}{\sigma(S(0,1))}\\
     &=\frac{3}{8}(2/3)^{d-1}I_{\left(\sin{\omega}\right)^2}\left(\frac{d-1}{2},\frac{1}{2}\right).
    \end{align}
    
    The estimate \eqref{eq:delta_omega_estimate}
    follows by \Cref{lem:I_lower_bound}.  
    
    Finally, the last claimed relation for $d=3$ follows by the equality
    \begin{equation*}
        I_x(1,b)=1-(1-x)^b, \quad x,b>0.
    \end{equation*}
\end{proof}

\begin{proof}[\bf Proof of \Cref{thm:u_cone_3D}] \label{proof:u_cone_3D}
    Both parts follow from \Cref{coro:u_g_delta} and \Cref{coro:u_f_delta}, corroborated with \Cref{coro:delta_omega}.
\end{proof}

\section{Appendix}\label{s:appendix}

\begin{thm}[Burkholder-Davis-Gundy (BDG) inequality (cf. e.g. \cite{revuz2013continuous})]\label{t:BDG}
Let $(M(t))_{t\geq 0}$ be a continuous local martingale defined on some filtered probability space $(\Omega, \mathcal{F},(\mathcal{F})_{t\geq 0},\mathbb{P})$, with $M(0)=0$, $(\langle M \rangle(t))_{t\geq 0}$ denote its quadratic variation process, and $M^\ast(t):=\sup_{0\le s\leq t}|M(s)|$, $t\geq 0$.  
Then, for any $0<\alpha<\infty$ there exist some constants $c_\alpha,C_\alpha>0$ which depend only on $\alpha$ such that
\begin{equation}\label{e:BDG}
c_\alpha \mathbbm{E}[\left(\langle M \rangle(\tau)\right)^{\alpha/2}]\le \mathbbm{E}[(M(\tau)^*)^\alpha]\leq  C_\alpha \mathbbm{E}[\left(\langle M \rangle(\tau)\right)^{\alpha/2}]
\end{equation}
for all stopping times $\tau$. 
Regarding the constants, for $\alpha<2$, we can take (see \cite[Chapter IV, Theorem 4.1] {revuz2013continuous}) $ c_\alpha=\frac{2-\alpha}{4-\alpha}$ and $C_\alpha=\frac{4-\alpha}{2-\alpha}$.  
Alternatively from \cite{ren2008burkholder} $c_\alpha=(2/\alpha)^{\alpha/2}(2-\alpha)/2$ and from \cite{davis2006burkholder,peccati2008burkholder} we can take $C_\alpha=(2/\alpha)^{\alpha/2}\frac{2}{2-p}$.  Also for $\alpha>1$, we can choose $c_\alpha=1/(2\sqrt{2}\alpha)$ and $C_\alpha=2\sqrt{2}\alpha$ (cf. \cite[Theorem 2]{ren2008burkholder}).   
\end{thm}

We shall be interested in applying the a multidimensional version of the BDG inequality, for $\alpha<1$. 
In this respect, the following inequality due to \cite{lenglart1977relation} (see also \cite{mehri2021stochastic} and \cite{geiss2021sharpness}) is useful in order to derive optimal estimates.
\begin{thm}[Lenglart's inequality (cf. e.g. {\cite[Theorem 1.1]{geiss2021sharpness})}]
Let $\left(\Omega,\mathcal{F},\mathcal{F}_t,\mathbb{P}\right)$ be a filtered probability space satisfying the usual hypotheses.
Further, let $X(t),G(t),t\geq 0$ be two non-negative (real-valued) adapted right-continuous  processes, such that $G(t),t\geq 0$ is in addition predictable and non-decreasing, and moreover
\begin{equation*}
\mathbb{E}\left[X(\tau)|\mathcal{F}_0\right]\leq \mathbb{E}\left[G(\tau)|\mathcal{F}_0\right],
\end{equation*}
for all bounded stopping times $\tau$.
Then for all $0<\alpha<1$ we have
\begin{equation}\label{eq:lenglart}
\mathbb{E}\left[\sup_{t\geq 0}X^\alpha(t)|\mathcal{F}_0\right]\leq\frac{\alpha^{-\alpha}}{1-\alpha} \mathbb{E}\left[\sup_{t\geq 0}G^\alpha(t)|\mathcal{F}_0\right].    
\end{equation}
\end{thm}

By combining Lenglart's inequality with the BDG inequality, we obtain the following multidimensional BDG inequality for $\alpha<1$.
\begin{coro}\label{coro:LBDG}
Let $\left(M(t)=\left(M^{(1)}(t),\dots M^{(d)}(t)\right)\right)_{t\geq 0}$ be an $d$-dimensional continuous local martingale defined on some filtered probability space $(\Omega, \mathcal{F},(\mathcal{F})_{t\geq 0},\mathbb{P})$, with $M(0)=0$, $(\langle M ^{(i)}\rangle)_{t\geq 0}$ denote the quadratic variation process for each coordinate $M^{(i)}$, $1\leq i\leq d$, and $M^\ast(t):=\sup_{0\le s\leq t}\|M(s)\|$, $t\geq 0$.  
Then
\begin{equation}\label{e:LBDG}
\mathbb{E}[(M^\ast(\tau))^\alpha]\leq
\begin{cases}
C_\alpha d^{(\alpha-1)^+} \sum_{i=1}^d \mathbb{E}\left[\left(\langle M^{(i)} \rangle(\tau)\right)^{\alpha/2}\right], \quad &0<\alpha\\
\frac{\alpha^{-\alpha}}{1-\alpha} \mathbb{E}\left[\left(\sum_{i=1}^d\langle M^{(i)} \rangle(\tau)\right)^{\alpha/2}\right], \quad &0<\alpha<1
\end{cases}
\end{equation}
for all stopping times $\tau$, where $C_\alpha$ is the constant appearing in the BDG inequality \eqref{e:BDG}.
In particular, if $B(t), t\geq 0$ is the standard $d$-dimensional Brownian motion on a filtered probability space satisfying the usual hypotheses, then for every stopping time $\tau$ we have
\begin{equation}\label{e:bmLBDG}
\mathbb{E}[(B^\ast(\tau))^\alpha]
\leq C(\alpha,d)\mathbb{E}\left[\tau^{\alpha/2}\right], \quad \mbox{where}\quad 
C(\alpha,d):=\begin{cases}
d^{1+(\alpha-1)^+}C_\alpha, \quad &0<\alpha\\
\min\left(d^{\alpha/2}\frac{\alpha^{-\alpha}}{1-\alpha},d^{1+(\alpha-1)^+}C_\alpha\right), \quad &0<\alpha<1
\end{cases}
\end{equation}
\end{coro}
\begin{proof}
On the one hand, by the BDG inequality and the fact that $(a_1+\dots+a_d)^\alpha\leq d^{(\alpha-1)^+}\left(a_1^\alpha+\dots+a_d^\alpha\right)$ for $a,b\geq 0<\alpha$, we get
\begin{align*}
\mathbb{E}[(M^\ast(\tau))^\alpha] \leq d^{(\alpha-1)^+}\mathbb{E}\left[\sum_{i=1}^d(M^{(i),\ast}(\tau))^\alpha\right]
\leq C_\alpha d^{(\alpha-1)^+} \sum_{i=1}^d \mathbb{E}\left[\left(\langle M^{(i)} \rangle(\tau)\right)^{\alpha/2}\right]. 
\end{align*}
On the other hand, if $\alpha<1$, then by Lenglart's inequality
applied with $X(t):=\|M(t\wedge \tau)\|^2$ and $G(t):=\sum_{i=1}^d\langle M^{(i)}\rangle(t\wedge \tau)$ we get
\begin{align*}
\mathbb{E}\left[(M^\ast(\tau))^{2\alpha}\right]  \leq
\frac{\alpha^{-\alpha}}{1-\alpha} \mathbb{E}\left[\left(\sum_{i=1}^d\langle M^{(i)} \rangle(\tau)\right)^{\alpha}\right]. 
\end{align*}
Putting together the above two inequalities we obtain \eqref{e:LBDG}.

Finally, \eqref{e:bmLBDG} follows trivially from \eqref{e:LBDG} since $\langle B^{(i)}\rangle(\tau)=\tau$.
\end{proof}

\begin{thm}[Doob's maximal inequality (cf. e.g. {\cite[Theorem 4.4]{ku23})}]\label{t:Doob} 

Let $(M(t))_{t\geq 0}$ be a non-negative continuous submartingale defined on some space $(\Omega,\mathcal{F},(\mathcal{F}_t)_{t\geq 0},\mathbb{P})$.
Then we have 
\begin{equation}\label{eq:Doob}
\mathbb{E}\left[\sup_{t\geq 0}|M(t)|^\alpha\right]
\leq 
\left\{\begin{array}{ll}
\dfrac{e}{e-1}\left\{ 1+\sup\limits_{t\geq 0}\mathbb{E}\left[|M(t)|\log^+{|M(t)|}\right]\right\} &\,\textrm{ if } \alpha=1\\
\left(\dfrac{\alpha\vee 1}{|\alpha-1|}\right)^{\alpha\vee 1}\sup\limits_{t\geq 0} \mathbb{E}\left[|M(t)| \right] ^\alpha &\,\textrm{ if } \alpha>0, \alpha\neq 1
\end{array}\right..    
\end{equation}
\end{thm}

\begin{lem}\label{lem:tau-B}
If $B(t), t\geq 0$ is the standard $d$-dimensional Brownian motion on a filtered probability space satisfying the usual hypotheses, then for every stopping time $\tau$ and every $0<\alpha\neq 1$ we have 
   \begin{align}\label{eq:tau-B}
       \mathbb{E}\left[\tau^{\alpha/2}\right]
       \leq \frac{1}{c_{\alpha}}\left(\dfrac{\alpha\vee 1}{|\alpha-1|}\right)^{\alpha\vee 1}d^{\frac{\alpha}{2}\left(\frac{2}{\alpha}-1\right)^+-1}\left(\mathbb{E}\left[\left\| B(\tau)\right\|^2\right]\right)^{\alpha/2}.
   \end{align}
\end{lem}
\begin{proof}

Since $B^{(i)}$ is a martingale with $\langle B^{(i)}(t)=t,\rangle, t\geq 0$, $1\leq i\leq d$, we have by \Cref{t:BDG} that
\begin{equation*}
    \mathbb{E}\left[\tau^{\alpha/2}\right]
    =\mathbb{E}\left[\left(\langle B^{(i)}\rangle(\tau)\right)^{\alpha/2}\right]
    \leq \frac{1}{c_{\alpha}}\mathbb{E}\left[\left( B^{(i)\ast}(\tau)\right)^{\alpha}\right].
\end{equation*}
Consequently, using again \Cref{t:BDG} we get
\begin{align*}
    \mathbb{E}\left[\tau^{\alpha/2}\right]
    \leq\frac{1}{c_{\alpha}}\left(\dfrac{\alpha\vee 1}{|\alpha-1|}\right)^{\alpha\vee 1}\frac{1}{d}\sum_{i=1}^d\mathbb{E}\left[\left| B^{(i)}(\tau)\right|^2\right]^{\alpha/2}.
\end{align*}
Using the inequality
\begin{equation*}
    a_1^\beta+\cdots a_d^\beta \leq d^{\beta\left(\frac{1}{\beta}-1\right)^+}\left(a_1+\cdots a_d\right) ^\beta, \quad \mbox{ for all}\quad a_1,\dots,a_d\geq 0, d\geq 1,
\end{equation*}
for $\beta=\alpha/2$, we get
\begin{align*}
    \mathbb{E}\left[\tau^{\alpha/2}\right]
    &\leq\frac{1}{c_{\alpha}}\left(\dfrac{\alpha\vee 1}{|\alpha-1|}\right)^{\alpha\vee 1}d^{\frac{\alpha}{2}\left(\frac{2}{\alpha}-1\right)^+-1}\left(\sum_{i=1}^d\mathbb{E}\left[\left| B^{(i)}(\tau)\right|^2\right]\right)^{\alpha/2}\\
    &=\frac{1}{c_{\alpha}}\left(\dfrac{\alpha\vee 1}{|\alpha-1|}\right)^{\alpha\vee 1}d^{\frac{\alpha}{2}\left(\frac{2}{\alpha}-1\right)^+-1}\left(\mathbb{E}\left[\left\| B(\tau)\right\|^2\right]\right)^{\alpha/2}.
\end{align*}
\end{proof}

\section{Paper specific notations and explicit values of the constants}

\subsection{Specific notations}

\begin{itemize}
\item {\bf Ball:} throughout the paper $B(a,r)$ denotes the Euclidean ball of center $a\in \mathbb{R}^d$ and radius $r>0$.
\item {\bf Annulus:} For two numbers $0<r<R$ and a point $a\in\R^d$ we define the {\it annulus}
\[
 A(a,r,R):=\{ x\in \R^d, r<|x-a|<R \}.
\]

\item {\bf Cones:} for an angle $\omega\in [0,\pi]$ and a radius $r>0$ we define the  circular cone $\mathcal{C} (\omega,r)\in \mathbb{R}^d$ (see Figure $2$) by (see \eqref{eq:C_2d}): 

$$
    \mathcal{C}(\omega,r):=\left\{x=(x_1,\dots,x_d)\in \mathbb{R}^d : \|x\|\leq r, \;x_1/\|x\|\geq \cos{ (\omega) } \right\}.
$$
By a slight abuse of notation, $\mathcal{C}(\omega,r)$ is also used to denote cone with vertex $x_0\in \mathbb{R}^d$, angle $\omega$ and radius $r$, whenever $\mathcal{C}(\omega,r)$ is isometric to the cone represented by \eqref{eq:C_2d}. 

\item {\bf Constants related to angle values:}
\begin{itemize} \item We set in \eqref{eq:tilde_omega}
$$
    \tilde{\omega}:=\dfrac{\pi}{2(\pi-\omega)}, \quad \omega \in [0,\pi).
$$

\item In \eqref{eq:delta_omega} we define
$$
\delta_\omega:=1-\frac{3}{4\sigma(S(0,1))}\int_{{\sf Cap}_\omega(1)} \frac{1}{ \|y_0-\xi\|^d}\;\sigma(d\xi)
    <1, \quad d\geq 3,
$$

\end{itemize}
\item {\bf Wedge:} for $r,l>0$, $\omega \in [0,\pi]$, 
we consider the {\it wedge} $\mathcal{W}(\omega,r,l)\subset \mathbb{R}^3$ given by
$$
    \mathcal{W}(\omega,r,l)=\mathcal{C}(\omega,r)\times (-l/2,l/2), 
$$
where $\mathcal{C}(\omega,r)$ is the cone with vertex $0$ given by \eqref{eq:C_2d}.
By an abuse of notation we say that $\mathcal{W}(\omega,r,l)$ is a wedge centered at $x_0\in \mathbb{R}^3$, angle $\omega$, radius $r$ and length $l$, whenever $\mathcal{W}(\omega,r,l)$ is isometric to the wedge represented above.

\item {\bf Green kernel constants} as used in Subsection~\ref{sec:Greenfunction}.
\begin{itemize} \item $\gamma_d:=\frac{\Gamma(d/2-1)}{(4\pi)^{d/2}}.$
\item $\gamma_{d,\alpha}:=\gamma_d\left(1-\frac{\alpha}{d-2+\alpha}\right)^{2-d-\alpha}\left(\frac{2(d-2)}{\alpha}\right)^\alpha$
\end{itemize}
\end{itemize}

\subsection{Explicit values of the constants}

\begin{itemize}

\item 
$$
{\bf\large C(\alpha,d)}:=\begin{cases}
d^{1+(\alpha-1)^+}C_\alpha, \quad &0<\alpha\\
\min\left(d^{\alpha/2}\frac{\alpha^{-\alpha}}{1-\alpha},d^{1+(\alpha-1)^+}C_\alpha\right), \quad &0<\alpha<1
\end{cases}
$$ defined in \eqref{e:bmLBDG}.

\item The constant $\bf c_\alpha$, as defined in  Theorem~\ref{t:BDG}: for $\alpha<2$, we can take (see \cite[Chapter IV, Theorem 4.1] {revuz2013continuous}) $ c_\alpha=\frac{2-\alpha}{4-\alpha}$ and $C_\alpha=\frac{4-\alpha}{2-\alpha}$.  
Alternatively from \cite{ren2008burkholder} $c_\alpha=(2/\alpha)^{\alpha/2}(2-\alpha)/2$ and from \cite{davis2006burkholder,peccati2008burkholder} we can take $C_\alpha=(2/\alpha)^{\alpha/2}\frac{2}{2-p}$.  Also for $\alpha>1$, we can choose $c_\alpha=1/(2\sqrt{2}\alpha)$ and $C_\alpha=2\sqrt{2}\alpha$ (cf. \cite[Theorem 2]{ren2008burkholder}).

\item The constant $\bf c=c_{d,\alpha}$ is provided in \eqref{eq:c} as:
$$
    c=c_{d,\alpha}=\frac{3^\alpha}{c_{\alpha}}\left(\dfrac{\alpha\vee 1}{|\alpha-1|}\right)^{\alpha\vee 1}d^{\frac{\alpha}{2}\left(\frac{2}{\alpha}-1\right)^+-1},
$$
whilst $c_\alpha$ is the constant appearing in the BDG inequality \eqref{e:BDG} from \Cref{t:BDG}.

\item The constant $\bf C(d,D,\gamma,q)$ as defined in \eqref{eq:C_f}
$$
    C(d,D,\gamma,q)=
    \begin{cases}
    \left[\frac{(d-2)(\gamma-q)}{(d(d-2)\omega_d)^{\frac{q}{\gamma-q}}(2\gamma-qd)}\right]^{\frac{\gamma-q}{\gamma q}}{\sf diam}(D)^{\frac{2\gamma-qd}{pq}} & d\geq 3\\
   \frac{{\sf diam}(D)^{\frac{2(\gamma-q)}{\gamma q}}}{(4\pi)^{\frac{1}{q}}} \left[\Gamma\left(\frac{\gamma}{\gamma-q}+1\right)\right]^{\frac{\gamma-q}{\gamma q}} &\quad d=2
    \end{cases}, \quad x\in D.
$$ where $\Gamma$ denotes the Gamma function.

\item The constant $\bf C_1(\alpha,d)$ is provided in \eqref{eq:v_alpha_uniform}
$$
    C_1(\alpha, d)
    =
    \begin{cases}
      \frac{1}{d^\alpha}, \quad &\alpha\leq 1\\
      \frac{1}{c_\alpha}\left(\frac{\alpha\vee 1}{|\alpha-1|}\right)^{\alpha \vee 1} d^{\frac{\alpha}{2}\left(\frac{2}{\alpha}-1\right)^+-1}, \quad &\alpha> 1
    \end{cases}. 
$$   

\item The constant $\bf C_2(\alpha,r)$ is provided in \eqref{eq:C_2}
$$
    C_2(\alpha,r):=\left(\dfrac{\alpha\vee 1}{|\alpha-1|}\right)^{\alpha\vee 1}\frac{2^{(\alpha-1)^+}}{c_{\alpha}}
        (2r)^\alpha \left(1+(2r)^{(\alpha-1)^+}\right).
$$

\item  The constants $\bf \tilde{C}_1,\tilde{C}_2$ are given in Corollary~\ref{coro:u_g_delta} by

        \begin{align*}
            &\tilde{C}_1:=\begin{cases}
               \frac{cr_0^\alpha}{1-2^\alpha\delta} ,& \delta<1/2^\alpha\\
               0 ,& \delta>1/2^\alpha
            \end{cases}, 
            \quad
            \tilde{C}_2:=\begin{cases}
               \frac{{\sf diam}(D)^{2\alpha}}{\delta 2^\alpha} ,& \delta<1/2^\alpha\\
               \frac{cr_0^\alpha}{2^\alpha\delta-1}+\frac{{\sf diam}(D)^{2\alpha}}{\delta 2^\alpha} ,& \delta>1/2^\alpha
            \end{cases}
        \end{align*}
        whilst $c=c_{d,\alpha}$ is given by \eqref{eq:c}.

\item  The constants $\bf C_1,C_2$ are given in Corollary~\ref{coro:u_g_delta} by

        \begin{align*}
            &C_1:= 2^{(\alpha-1)^+}\left[C(\alpha,d)\tilde{C}_1+1\right],\quad 
            C_2:= 2^{(\alpha-1)^+}C(\alpha,d)\tilde{C}_2
         \end{align*}

\item Constants $\bf C_3$ and $\bf C_4$ are provided in Corollary~\ref{coro:D-cone-2D}
\begin{equation*}
    C_3:=2C_2(\alpha,r)\tilde{\omega}^\alpha,
    \quad
    C_4:=C_1(\alpha/2,2)\left({\sf diam}(D)+r\right)^{\alpha}\tilde{\omega}.
\end{equation*}

\item  Constants $\bf C_5$ and $\bf C_6$, as defined in \eqref{eq:C_1-C_3} 

  \begin{align*}
        &C_5= C(\alpha,2)\tilde{\omega}^\alpha\frac{8r^\alpha}{(1-\alpha)c_\alpha},\quad C_6= C(\alpha,2)\frac{\left({\sf diam}(D)+r\right)^{\alpha}\tilde{\omega}}{2^{\alpha/2}}.
    \end{align*}
    
\item The constant $\bf C_7$ is provided in Theorem~\ref{thm:u_cone_2D} as:

$$
        C_7=C_3+C_4, \quad \mbox{ where } C_3,C_4 \mbox{ are as in } \Cref{coro:D-cone-2D} \mbox{ with } \alpha=2.
$$



\item  The constant $\bf C_8$ given in Corollary~\ref{coro:D-wedge-3D} is:

$$
C_8:=C_1(\alpha/2,3)\left({\sf diam}(D)+r+l\right)^{\alpha}\tilde{\omega}. 
$$

\item The constants $\bf C_9, C_{10}$   are taken to be (in Theorem~\ref{thm:u_wedge_3D}):
$$
        C_9= C(\alpha,3)C_3,\quad C_{10}= C(\alpha,3)C_8.
$$

\item  The constants $\bf C_{11}, C_{12}$  are given in Theorem~\ref{thm:u_wedge_3D} by:
$$
        C_{11}=C_3+ C_8, \quad  \mbox{with }C_3, C_8 \mbox{ given in }\Cref{coro:D-wedge-3D} \mbox{ for } \alpha=2.
$$
$$
        C_{12}=C(3,D,\gamma,q)\left(C_{11}\right)^{1/p}, \mbox{ whilst } C(3,D,\gamma,q) \mbox{ is given by }\eqref{eq:C_f}.
$$

\end{itemize}

\noindent \textbf{Acknowledgements.}
Iulian C\^impean acknowledges support from the 
project  PN-III-P1-1.1-PD-2019-0780, within PNCDI~III. The work of Arghir Zarnescu has been partially supported by the Basque Government through the BERC 2022-2025 program and by the Spanish State Research Agency through BCAM Severo Ochoa excellence accreditation Severo Ochoa CEX2021-00114 and through project PID2023-146764NB-I00
funded by MICIU/AEI/10.13039/501100011033. Arghir Zarnescu was also partially supported by a grant of the Ministry of Research, Innovation and Digitization, CNCS-UEFISCDI, project number PN-IV-P2-2.1-T-TE-2023-1704, within PNCDI IV.

\bibliographystyle{plain}

\end{document}